\documentclass[11pt]{article}
\usepackage{amsfonts, amsmath, amssymb, amscd, amsthm, color, graphicx, mathrsfs, mathabx, wasysym, setspace, mdwlist, calc,float}
\usepackage{setspace}
\usepackage{hyperref}
\usepackage{tikz-cd} 
\usepackage{hyperref}
\usepackage{tocloft}
\usepackage{xcolor}

\usepackage{hyperref}
\hypersetup{linktocpage}

\hypersetup{colorlinks,
    linkcolor={red!50!black},
    citecolor={blue!80!black},
    urlcolor={blue!80!black}}
\usepackage{float}

\newcommand{\e}{\varepsilon}
\newcommand{\NN}{\mathbb{N}}

\newcommand{\ZZ}{\mathbb{Z}}
\newcommand{\RR}{\mathbb{R}}
\newcommand{\CC}{\mathbb{C}}

\newcommand{\N}{\mathcal{N}}

\newtheorem{thm}{Theorem}[section]
\newtheorem{cor}[thm]{Corollary}
\newtheorem{lem}[thm]{Lemma}
\newtheorem{prop}[thm]{Proposition}

\theoremstyle{definition}
\newtheorem{defn}[thm]{Definition}
\newtheorem{conv}[thm]{Convention}

\theoremstyle{remark}
\newtheorem{rem}[thm]{Remark}

\newtheorem{claim}[thm]{Claim}

\newcommand{\stab}{\mathrm{Stab}}
\newcommand{\lk}{\mathrm{Lk}}
\newcommand{\st}{\mathrm{St}}
\newcommand{\elk}{\mathrm{Elk}}
\newcommand{\supp}{\mathrm{supp}}
\newcommand{\diam}{\mathrm{diam}}

\newcommand{\girth}{\mathrm{girth}}
\newcommand{\leaf}{\mathrm{leaf}}
\newcommand{\geo}{\mathrm{Geo}}
\newcommand{\Prob}{\mathrm{Prob}}
\newcommand{\fix}{\mathrm{Fix}}
\renewcommand{\N}{\mathcal{N}}
\newcommand{\C}{\mathscr{C}}

\newcommand{\la}{\langle}
\newcommand{\ra}{\rangle}

\renewcommand{\L}{\mathbf{Lab}}

\newcommand{\F}{\mathcal{F}}
\newcommand{\G}{\mathcal{G}}

\newcommand{\act}{\curvearrowright}

\newcommand{\tensor}{\overline{\otimes}}

\renewcommand{\L}{\mathbf{Lab}}

\newcommand\blfootnote[1]{%
  \begingroup
  \renewcommand\thefootnote{}\footnote{#1}%
  \addtocounter{footnote}{-1}%
  \endgroup
}

\begin{document}

\title{Infinite graph product of groups II: Analytic properties}

\author{Koichi Oyakawa}
\date{}

\maketitle

\vspace{-3mm}

\begin{abstract}
We study analytic properties of graph product of finite groups with a hyperbolic defining graph. This is done by studying dynamics on the Bowditch compactification of the extension graph, or the crossing graph, of graph product. In particular, we provide a new class of convergence groups and identify the if and only if condition for this convergence action to be geometrically finite. We also provide a new class of properly proximal groups, relatively bi-exact groups, and groups with strongly solid group von Neumann algebras.
\end{abstract}

\section{Introduction}
\label{sec:Introduction}
\blfootnote{\textbf{MSC 2020} Primary: 20F65, 22D25. Secondary: 20F67, 37D40.}
\blfootnote{\textbf{Key words and phrases}: graph product, hyperbolic spaces, convergence groups, bi-exact groups.}
This paper is a continuation of \cite{Oya24b}, where geometry of the extension graph of graph product of groups was studied. The extension graph of graph product was already (accepted) introduced by Casals-Ruiz, Kazachkov, and de la Nuez Gonz\'alez in \cite[Definition 3.1]{MIJ24} (see also \cite{EH24}). Recall that given a simplicial graph $\Gamma$ and a collection of groups $\G = \{G_v\}_{v \in V(\Gamma)}$ assigned to each vertex of $\Gamma$, the graph product $\Gamma\G$ is a group obtained by taking quotient of the free product $\ast_{v \in V(\Gamma)} G_v$ by setting that group elements of two adjacent vertices commute (see Definition \ref{def:graph product of groups}). In this paper, we will present more applications of the extension graph by studying analytic properties of graph product of groups that are of particular interest in operator algebras. In the study of operator algebras, it has long been people's interest to investigate structural properties of group von Neumann algebras such as primeness, solidity, and rigidity. To tackle these problems, analytic properties of groups such as bi-exactness and its variants were introduced and turned out to be particularly useful (see \cite{Oza06b, BO08, BIP21}).

While the study of operator algebras associated to graph product is currently very active (see \cite{CF17, CRW18, CDD24b, CDD24a, DK24, CK24, Bor24, IRBDSB24a, IRBDSB24b}), only finite defining graphs have been considered in most research so far and the finiteness condition of a defining graph is used in an essential way in the proofs. In this paper, we present a novel approach to study operator algebras of graph product of groups by exploiting the large scale geometry of a defining graph. This approach allows us to study a much wider class of defining graphs, beyond finite graphs. Among other geometric properties, we will focus on the case where a defining graph is hyperbolic in this paper.

More specifically, we use the dynamics of graph product on a certain compactification of the extension graph (or the crossing graph). It turns out that this dynamics has an intensively studied property called convergence property, which is of independent interest. The notion of a convergence action was introduced by Gehring and Martin in \cite{GM87} to axiomatise dynamical properties of a Kleinian group acting on the ideal sphere of real hyperbolic space and was studied by many people for an action on a compact Hausdorff space in general (see \cite{Tuk94, Tuk98, Bow98, Bow99a, Bow99b, Bow02, Dah03c, Yam04, Ger09}). Remarkable feature of convergence actions is that we can recover group theoretic property from this dynamical property. In fact, if a subgroup of the homeomorphism group $\mathrm{Homeo}(S^1)$ of the circle acts on $S^1$ as a convergence action, then this action is topologically conjugate to the induced action of a Fuchsian group (see \cite{Tuk88, Gab92, CJ94}). Also, a group is hyperbolic if and only if it admits a uniform convergence action on a perfect compact metric space (see \cite{Bow98}). There is also a similar dynamical characterization of relatively hyperbolic groups by a geometrically finite convergence action (see \cite{Yam04}). Yet another result is that an almost finitely presented group acting as a minimal convergence group on a Cantor set without parabolic points is virtually free (see \cite{Bow02}). To summarize, the class of groups that admit a convergence action tends to be restrictive. Therefore, it is interesting to find a new class of convergence groups, which we do in Theorem \ref{thm:intro convergence action} below.
\begin{thm}\label{thm:intro convergence action}
    Suppose that $\Gamma$ is a fine hyperbolic graph with $\diam_\Gamma(\Gamma) > 2$ and $\girth(\Gamma) > 20$ and that $\G=\{G_v\}_{v\in V(\Gamma)}$ is a collection of non-trivial finite groups. Then, the action of the graph product $\Gamma\G$ on the Bowditch compactification $\Delta\Gamma^e$ of the extension graph is a non-elementary non-uniform convergence action. Moreover, the action $\Gamma\G \act \Delta\Gamma^e$ is geometrically finite if and only if the set $V(\Gamma)\setminus \leaf(\Gamma)$ is finite, where $\leaf(\Gamma)$ is the set of all vertices in $V(\Gamma)$ whose valency is at most $1$.
\end{thm}
Although the context is different, the above condition for geometric finiteness is reminiscent of the fact that a Kleinian group is called geometrically finite if it has a fundamental polyhedron in hyperbolic 3-space with finitely many sides. In Theorem \ref{thm:intro convergence action}, the condition that $\{G_v\}_{v\in V(\Gamma)}$ is a collection of finite groups is essential by Proposition \ref{prop: infinite case doesn't admit non-elementary convergence action}. Also, note that fine graphs don't need to be locally finite (see Definition \ref{def:fine graph}).

As a corollary of Theorem \ref{thm:intro convergence action}, we get one analytic property of graph product called proper proximality. Proper proximality of groups was introduced by Boutonnet, Ioana and Peterson in \cite{BIP21} as one generalization of bi-exactness to study rigidity properties of von Neumann algebras associated to groups or ergodic group actions. Proper proximality of graph product of groups was studied in \cite{CKE24} for a finite defining graph, where the finiteness condition was essentially used. We extend proper proximality to a wider class of a defining graph in Corollary \ref{cor:properly proximal} below.
\begin{cor}\label{cor:properly proximal}
    Suppose that $\Gamma$ is a fine hyperbolic graph with $\diam_\Gamma(\Gamma) > 2$ and $\girth(\Gamma) > 20$ and that $\G=\{G_v\}_{v\in V(\Gamma)}$ is a collection of non-trivial finite groups. Then, $\Gamma\G$ is properly proximal.
\end{cor}
\begin{proof}
    This follows from Theorem \ref{thm:intro convergence action} since non-elementary convergence groups are properly proximal by \cite[Proposition 1.6]{BIP21}.
\end{proof}
Another generalization of bi-exactness is relative bi-exactness. This is useful to study group von Neumann algebras, considered as a promising analytic property to solve the open problem on primeness of relatively hyperbolic groups (see \cite[Conjecture 1.2]{Oya23a} and Proposition \ref{prop:relatively biexact groups are prime}). However, groups having this analytic property is not known much except bi-exact groups (see \cite{DK24b}). Importantly, Theorem \ref{thm:intro graph product becomes relatively bi-exact} provides a new class of groups bi-exact relative to almost malnormal subgroups (see the proof of Corollary \ref{cor:graph product becomes prime}), which was studied in \cite{DK24b}.
\begin{thm}\label{thm:intro graph product becomes relatively bi-exact}
    Suppose that $\Gamma$ is a uniformly fine hyperbolic countable graph with $\girth(\Gamma) > 20$ and that $\G=\{G_v\}_{v \in V(\Gamma)}$ is a collection of non-trivial finite groups. Then, $\Gamma\G$ is bi-exact relative to the collection of subgroups $\{\, \la G_w \mid w\in \lk_{\Gamma}(v) \ra \,\}_{v \in V(\Gamma)}$.
\end{thm}
Finally, we will prove strong solidity by restricting to a locally finite defining graph. Instead, we remove the hyperbolicity condition on a defining graph. Strong solidity of von Neumann algebras was introduced by Ozawa and Popa in \cite{OP10} and attracted particular attention as it is known as the strongest indecomposability property that encompasses primeness, solidity, and absence of Cartan subalgebras (see \cite{Cyr10, CS13, PV14, Iso15, BCV18, Cas20, BC24}). Strong solidity of right angled Coxeter groups was studied by Borst and Caspers in \cite{BC24}. The finiteness condition of a defining graph was essentially used in \cite{BC24} because their argument involved induction on the number of vertices of a defining graph. This was later generalized to graph product of von Neumann algebras with a finite defining graph in \cite[Theorem F]{BCC24}. In Theorem \ref{thm:intro graph product becomes bi-exact and strongly solid} below, we significantly broaden the class of graph product of groups whose group von Neumann algebra is strongly solid, because we allow infinite defining graphs. 
\begin{thm}\label{thm:intro graph product becomes bi-exact and strongly solid}
    Suppose that $\Gamma$ is a uniformly locally finite countable graph with $\girth(\Gamma)>4$ and $\diam_\Gamma(\Gamma) > 2$, and that $\G=\{G_v\}_{v \in \Gamma}$ is a collection of finite groups with $\sup_{v \in V(\Gamma)}|G_v| < \infty$, then the group von Neumann algebra $L(\Gamma\G)$ is strongly solid.
\end{thm}
As a byproduct of proving Theorem \ref{thm:intro graph product becomes bi-exact and strongly solid}, we also get bi-exactness of generalization of graph-wreath product below. Theorem \ref{thm:biexactness of graph wreath product} provides a new construction of bi-exact groups. Note that $\Gamma\G$ below is non-amenable in most cases. See Definition \ref{def:graph wreath product} for $\Gamma\G \rtimes G$.
\begin{thm}\label{thm:biexactness of graph wreath product}
    Suppose that $\Gamma$ is a uniformly locally finite countable graph with $\girth(\Gamma)>4$ and that a countable group $G$ acts on $\Gamma$ satisfying $|\stab_G(v)|<\infty$ for any $v \in V(\Gamma)$. Let $\G=\{G_v\}_{v \in \Gamma}$ be a collection of finite groups such that $\sup_{v \in V(\Gamma)}|G_v| < \infty$ and $G_v = G_{gv}$ for any $g \in G$ and $v \in V(\Gamma)$. Then, the following hold.
    \begin{itemize}
        \item[(1)]
        If $G$ is amenable, then $\Gamma\G \rtimes G$ is bi-exact.
        \item[(2)]
        If $G$ is bi-exact, then $\Gamma\G \rtimes G$ is bi-exact relative to $\{\Gamma\G\}$.
    \end{itemize}
\end{thm}

This paper is organized as follows. In Section \ref{sec:Preliminaries}, we explain preliminary definitions and known results that are necessary in this paper. In Section \ref{sec:Convergence action on the extension graph}, we discuss the application to convergence actions and prove Theorem \ref{thm:intro convergence action}. In Section \ref{sec:Relative bi-exactness}, we discuss the application to relative bi-exactness and prove Theorem \ref{thm:intro graph product becomes relatively bi-exact}. In Section \ref{sec:Strong solidity and bi-exactness}, we discuss strong solidity and prove Theorem \ref{thm:intro graph product becomes bi-exact and strongly solid} and Theorem \ref{thm:biexactness of graph wreath product}.


\vspace{2mm}

\noindent\textbf{Acknowledgments.}
I would like to thank Changying Ding for helpful discussions and for teaching me the proof of Proposition \ref{prop:relatively biexact groups are prime}. I would like to thank an anonymous referee for many helpful comments, which improved the exposition of the paper a lot.

\section{Preliminaries}
\label{sec:Preliminaries}

We start with preparing necessary notations about graphs, metric spaces, and group actions. Throughout this paper, we assume that graphs are simplicial (i.e. having no loops nor multiple edges) and a group acts on a graph as graph automorphisms unless otherwise stated.

\begin{defn}\label{def:concepts in graph theory}
    A \emph{graph} $X$ is the pair of a set $V(X)$ and a subset $E(X) \subset V(X)\times V(X)$ satisfying $\forall\, x \in V(X),\, (x,x) \notin E(X)$ and $\forall\, (x,y) \in V(X)^2,\, (x,y) \in E(X) \Leftrightarrow (y,x) \in E(X)$. An element of $V(X)$ is called a \emph{vertex} and an element of $E(X)$ is called an \emph{edge}. For an edge $e=(x,y) \in E(X)$, we denote $x$ by $e_-$ and $y$ by $e_+$, that is, we have $e=(e_-,e_+)$. For a vertex $x\in V(X)$, we define $\lk_X(x), \st_X(x) \subset V(X)$ and $\elk_X(x) \subset E(X)$ by
    \begin{align*}
        \lk_X(x)&=\{y\in V(X) \mid (x,y) \in E(X)\}, \\
        \st_X(x)&=\{x\} \cup \lk_X(x), \\
        \elk_X(x) &= \{ (x,y) \in E(X) \mid y \in \lk_X(x) \}.
    \end{align*}
    We define $\leaf(X)$ by $\leaf(X) = \{x \in V(X) \mid |\lk_X(v)|\le 1 \}$. A \emph{path} $p$ in $X$ is a sequence $p=(p_0,\cdots,p_n)$ of vertices, where $n \in \NN\cup\{0\}$ and $p_i \in V(X)$, such that $(p_i,p_{i+1}) \in E(X)$ for any $i \ge 0$. Given a path $p=(p_0,\cdots,p_n)$ in $X$,
    \begin{itemize}
        \item[-]
        the \emph{length} $|p| \in \NN\cup\{0\}$ of $p$ is defined by $|p|=n$,
        \item[-]
        we denote $p_0$ (the \emph{initial vertex}) by $p_-$ and $p_n$ (the \emph{terminal vertex}) by $p_+$,
        \item[-]
        we define $V(p)=\{p_i \mid 0 \le i \le n\}$ and $E(p) = \{(p_{i-1},p_i),(p_i,p_{i-1}) \mid 1 \le i \le n\}$,
        \item[-]
        a \emph{subsequence of} $V(p)$ is $(p_{i_0},\cdots,p_{i_m})$ with $m \in \NN\cup\{0\}$ such that $i_0 \le \cdots \le i_m$,
        \item[-]
        we say that $p$ \emph{has backtracking at} $p_i$ if $p_{i-1} = p_{i+1}$.
    \end{itemize}
    A \emph{loop} $p$ in $X$ is a path such that $p_-=p_+$. A \emph{circuit} $p=(p_0,\cdots,p_n)$ in $X$ is a loop with $|p|>2$ and without self-intersection except $p_-=p_+$ i.e. $p_i\neq p_j$ for any $i,j$ with $0 \le i <j <n$. For $e \in E(X)$ and $n\in \NN$, we define $\C_X(e,n)$ to be the set of all circuits in $X$ that contain $e$ and have length at most $n$. The \emph{girth} of $X$ $\girth(X) \in \NN$ is defined by $\girth(X)=\min\{|p| \mid \text{$p$ is a circuit in $X$}\}$ if there exists a circuit in $X$. If there is no circuit in $X$, then we define $\girth(X)=\infty$ for convenience. A graph is called \emph{connected} if for any $x,y \in V(X)$, there exists a path $p$ such that $p_-=x$ and $p_+=y$. When a graph $X$ is connected, $X$ becomes a geodesic metric space by a graph metric $d_X$ (i.e. every edge has length 1), hence we also denote this metric space by $X$. Given $L \subset V(X)$, the \emph{induced subgraph on} $L$ is defined by the vertex set $L$ and the edge set $E(X)\cap L^2$.
\end{defn}

\begin{defn}
    Let $X$ be a connected graph. A path $p$ in $X$ is called \emph{geodesic} if $|p|$ is the smallest among all paths from $p_-$ to $p_+$. For $x,y \in V(X)$, we denote by $\geo_X(x,y)$ the set of all geodesic paths in $X$ from $x$ to $y$. For a path $p=(p_0, \cdots, p_n)$ without self-intersection and $i,j$ with $0 \le i \le j \le n$, we denote the subpath $(p_i, \cdots, p_j)$ of $p$ by $p_{[p_i, p_j]}$.
\end{defn}

\begin{rem}
    Since we consider only simplicial graphs throughout this paper, the girth of a graph is always at least 3.
\end{rem}

\begin{rem}\label{rem:vertex with 0 link}
    When a graph $X$ is connected, $x\in V(X)$ satisfies $|\lk_X(x)| = 0$ if and only if $V(X)=\{x\}$.
\end{rem}

\begin{rem}\label{rem:diam when V(Gamma)setminus I(Gamma) is finite}
    If a connected graph $\Gamma$ satisfies $|V(X)\setminus \leaf(X)|<\infty$, then we have $\diam_X(X) < \infty$. Indeed, when $\diam_X(X) \ge 2$, for any $v \in \leaf(X)$ and $w \in \lk_X(v)$, we have $w \notin \leaf(X)$ by $\diam_X(X) \ge 2$. This implies $\diam_X(X) \le \diam_X(V(X)\setminus \leaf(X)) + 2 < \infty$.
\end{rem}

\begin{defn}
    Let $(X,d_X)$ be a metric space. For a subset $A \subset X$, the \emph{diameter} of $A$ $\diam_X(A) \in [0,\infty]$ is defined by $\diam_X(A) = \sup_{x,y \in A}d_X(x,y)$. For $A \subset X$ and $r \in \RR_{>0}$, we define $\N_X(A,r)\subset X$ by $\N_X(A,r) = \{y \in X \mid \exists \, x\in A, d_X(x,y) \le r\}$. When $A$ is a singleton i.e. $A=\{x\}$ with $x \in X$, we denote $\N_X(\{x\},r)$ by $\N(x,r)$ for brevity, that is, $\N(x,r)=\{y \in X \mid d_X(x,y) \le r\}$. For two subsets $A,B \subset X$, we define $d_X(A,B) \in \RR_{\ge0}$ by $d_X(A,B)=\inf_{x\in A, y\in B}d_X(x,y)$.
\end{defn}

\begin{defn}
    Let $G$ be a group. For $g,h \in G$, we define $[g,h] \in G$, by $[g,h] = ghg^{-1}h^{-1}$. For subsets $A, B\subset G$, we define $A B, [A,B] \subset G$ by $A B=\{gh \in G \mid g \in A, h \in B\}$ and $[A, B] = \{[g,h]\in G \mid g \in A, h \in B\}$. For a subset $A \subset G$, we denote by $\la A \ra$ the subgroup of $G$ generated by $A$ and also by $\la\!\la A \ra\!\ra$ the normal subgroup of $G$ generated by $A$, that is, $\la\!\la A \ra\!\ra = \la \bigcup_{g \in G}gAg^{-1} \ra$.
\end{defn}

\begin{defn}
    Let a group $G$ act on a set $X$. We denote by $X/G$ the quotient set of the orbit equivalence relation induced by the action $G \act X$. For $x \in X$, we define $\stab_G(x)\subset G$ by $\stab_G(x)=\{g \in G \mid gx=x\}$. For $g \in G$, we define $\fix_X(g) \subset X$ by $\fix_X(g) = \{ x \in X \mid gx=x \}$.
\end{defn}

\subsection{Graph products of groups}
\label{subsec:Graph products of groups}

Readers are referred to \cite[Definition 3.5]{Gre} for details of graph product of groups.

\begin{defn}\label{def:graph product of groups}
    Let $\Gamma$ be a simplicial graph and $\G=\{G_v\}_{v\in V(\Gamma)}$ be a collection of groups. The \emph{graph product} $\Gamma\G$ is defined by
    \begin{align*}
        \Gamma\G
        =
        \ast_{v\in V(\Gamma)} \, G_v ~ / ~ \la\!\la\, \{ [g_v, g_w] \mid (v,w)\in E(\Gamma), g_v\in G_v, g_w \in G_w \}\, \ra\!\ra.
    \end{align*}
\end{defn}

\begin{rem}\label{rem:vertex subgroup}
    For any $v \in V(\Gamma)$, the group $G_v$ is a subgroup of $\Gamma\G$. We often identify $G_v$ as a subgroup of $\Gamma\G$. Also, for any $v,w \in V(\Gamma)$ with $v\neq w$, we have $G_v \cap G_w = \{1\}$.
\end{rem}

\begin{defn}\label{def:reduced form of graph product}
    Let $\Gamma$ be a simplicial graph and $\G=\{G_v\}_{v\in V(\Gamma)}$ be a collection of groups. Given $g \in \Gamma\G$, a geodesic word of $g$ in the generating set $\bigsqcup_{v\in V(\Gamma)}(G_v\setminus\{1\})$ is called a \emph{normal form of} $g$. We denote the word length of $g$ by $\|g\|$ (i.e. $\|g\|=|g|_{\bigsqcup_{v\in V(\Gamma)}(G_v\setminus\{1\})}$) and call $\|g\|$ the \emph{syllable length} of $g$. Given a normal form $g=g_1\cdots g_n$ of $g$,
    \begin{itemize}
        \item [-] 
        each letter $g_i \in \bigsqcup_{v\in V(\Gamma)}(G_v\setminus\{1\})$ is called a \emph{syllable},
        \item [-]
        and we refer to the process of obtaining the new normal form $g=g_1\cdots g_{i+1}g_i \cdots g_n$, where $1\le i <n$, $g_i \in G_{v_i}$, $g_{i+1} \in G_{v_{i+1}}$, and $(v_i,v_{i+1}) \in E(\Gamma)$, as \emph{syllable shuffling}.
    \end{itemize}
    For $h_1,\cdots,h_n \in \Gamma\G$, we say that the decomposition $h_1\cdots h_n$ is \emph{reduced} if we have $\|h_1\cdots h_n\| = \|h_1\| + \cdots + \|h_n\|$.
\end{defn}

\begin{conv}\label{conv:normal form}
    Throughout this paper, when we say that $g=g_1\cdots g_n$ is a normal form of $g \in \Gamma\G$, we assume that $g_1\cdots g_n$ is a geodesic word in $\bigsqcup_{v\in V(\Gamma)}(G_v\setminus\{1\})$ satisfying $g_i \in \bigsqcup_{v\in V(\Gamma)}(G_v\setminus\{1\})$ for each $i$, even if we don't mention it for brevity.
\end{conv}

\begin{rem}\label{rem:reduced decomposition}
    The decomposition $g=h_1\cdots h_n$ is reduced if and only if for any normal form $h_i=s_{i,1}\cdots s_{i,\|g_i\|}$ of each $h_i$, the word $g=(s_{1,1}\cdots s_{1,\|g_1\|}) \cdots (s_{n,1}\cdots s_{n,\|g_n\|})$ is a normal form of $g$.
\end{rem}

Theorem \ref{thm:normal form theorem} below follows by the same proof as \cite[Theorem 3.9]{Gre}, although the underlying graph $\Gamma$ is assumed to be finite in \cite[Theorem 3.9]{Gre}. That is, we don't need to assume that $\Gamma$ is finite.

\begin{thm}[Normal form theorem]\label{thm:normal form theorem}
    Let $\Gamma$ be a simplicial graph and $\G=\{G_v\}_{v\in V(\Gamma)}$ be a collection of groups. For any $g \in \Gamma\G$ with $g \neq 1$, $g=g_1\cdots g_n$ is a normal form of $g$ if and only if for any pair $(i,j)$ with $1 \le i < j \le n$ satisfying $v_i=v_j$, there exists $k$ with $i<k<j$ such that $(v_k,v_i) \notin E(\Gamma)$. Also, if $g=g_1\cdots g_n$ and $g=h_1\cdots h_m$ are normal forms of $g$, then $n=m$ and we obtain one from the other by finite steps of syllable shuffling.
\end{thm}

\begin{defn}\label{def:support of g}
    Let $\Gamma$ be a simplicial graph and $\G=\{G_v\}_{v\in V(\Gamma)}$ be a collection of groups. Let $g=g_1\cdots g_n$ be a normal form of $g \in \Gamma\G \setminus \{1\}$. For each syllable $g_i$, there exists a unique vertex $v_i \in V(\Gamma)$ with $g \in G_{v_i} \setminus \{1\}$. We define $\supp(g) \subset V(\Gamma)$ by $\supp(g)=\{v_i \mid 1 \le i\le n\}$ and call $\supp(g)$ the \emph{support of} $g$. We define the support of $1\in \Gamma\G$ by $\supp(1)=\emptyset$.
\end{defn}

\begin{rem}
    The support of $g \in \Gamma\G$ is well-defined by Remark \ref{rem:vertex subgroup} and Theorem \ref{thm:normal form theorem}.
\end{rem}

\begin{rem}
    When $g \in G_v\setminus\{1\}$ with $v \in V(\Gamma)$, we often consider the singleton $\supp(g) \,(=\{v\})$ as an element of $V(\Gamma)$ (and denote $\supp(g) \in \st_\Gamma(v)$ for example) by abuse of notation.
\end{rem}

\begin{rem}
    When $g=g_1\cdots g_n$ is a normal form of $g \in \Gamma\G \setminus \{1\}$, we have $\supp(g_i) \neq \supp(g_{i+1})$ for any $i$ since the word $g_1\cdots g_n$ is geodesic.
\end{rem}

Lemma \ref{lem:subsequence of reduced sequence} below easily follows from minimality of the length of a geodesic word.

\begin{lem}\label{lem:subsequence of reduced sequence}
    Let $\Gamma$ be a simplicial graph and $\G=\{G_v\}_{v\in V(\Gamma)}$ be a collection of groups. Let $g=g_1\cdots g_n$ be a normal form of $g \in \Gamma\G$, then any word obtained from $g_1\cdots g_n$ by finite steps of syllable shuffling is a normal form of $g$. Also, for any $i,j$ with $1\le i \le j \le n$, the subword $g_i\cdots g_j$ is a normal form.
\end{lem}

See \cite{Oya24b} for details on the extension graph of graph product.

\begin{defn}\label{def:extension graph}
Let $\Gamma$ be a simplicial graph and $\G=\{G_v\}_{v\in V(\Gamma)}$ be a collection of non-trivial groups. The \emph{extension graph} $\Gamma^e$ is defined as follows.
\begin{align*}
    V(\Gamma^e)
    &=
    \{gG_vg^{-1} \in 2^{\Gamma\G} \mid v\in V(\Gamma), g\in \Gamma\G \}, \\
    E(\Gamma^e)
    &=
    \{ (gG_vg^{-1},hG_wh^{-1}) \in V(\Gamma^e)^2 \mid \text{$gG_vg^{-1}\neq hG_wh^{-1}$ and $[gG_vg^{-1},hG_wh^{-1}]=\{1\}$} \}.
\end{align*}
\end{defn}

\begin{rem}\label{rem:notation on extension graph}
As in \cite[Convention 3.9]{Oya24b}, the map $\iota \colon \Gamma \to \Gamma^e$ defined by $\iota(v)=G_v$ for each $v\in V(\Gamma)$ is a graph isomorphism from $\Gamma$ to the induced subgraph on $\iota(V(\Gamma))$ in $\Gamma^e$. Hence, we will consider $\Gamma$ as a subgraph of $\Gamma^e$ by this embedding and denote $G_v$ by $v$ for each $v\in V(\Gamma)$. Also, $\Gamma\G$ acts on $\Gamma^e$ by $\Gamma\G \times V(\Gamma^e) \ni (g,x) \to gxg^{-1} \in V(\Gamma^e)$. We will denote $gxg^{-1}$ by $g.x$ for simplicity i.e. $g.x = gxg^{-1}$. By \cite[Remark 3.11]{Oya24b}, for any $x \in V(\Gamma^e)$, there exists unique $v(x) \in V(\Gamma)$ such that $x \in \Gamma\G.v$.
\end{rem}

\subsection{Fine hyperbolic graphs}
\label{subsec:Fine hyperbolic graphs}

In this section, we review fine hyperbolic graphs and Bowditch topology on its Gromov bordification. Readers are referred to \cite{BH99} and \cite[Section 8]{Bow12} for details.

\begin{defn}\label{def:gromov product}
    Let $(X,d_X)$ be a metric space. For $x,y,z\in X$, we define $(x,y)_z$ by
\begin{align}\label{eq:gromov product}
    (x,y)_z=\frac{1}{2}\left( d_X(x,z)+d_X(y,z)-d_X(x,y) \right).    
\end{align}
\end{defn}

\begin{prop} \label{prop:hyp sp}
    For any geodesic metric space $(X,d_X)$, the following conditions are equivalent.
    \item[(1)]
    There exists $\delta\in\NN$ satisfying the following property. Let $x,y,z\in X$, and let $p$ be a geodesic path from $z$ to $x$ and $q$ be a geodesic path from $z$ to $y$. If two points $a\in p$ and $b\in q$ satisfy $d_X(z,a)=d_X(z,b)\le (x,y)_z$, then we have $d_X(a,b) \le \delta$.
    \item[(2)] 
    There exists $\delta\in\NN$ such that for any $w,x,y,z \in X$, we have
    \[
    (x,z)_w \ge \min\{(x,y)_w, (y,z)_w\} - \delta.
    \]
\end{prop}

\begin{defn}\label{def:hyperbolic space}
    A geodesic metric space $X$ is called \emph{hyperbolic}, if $X$ satisfies the equivalent conditions (1) and (2) in Proposition \ref{prop:hyp sp}. We call a hyperbolic space $\delta$-\emph{hyperbolic} with $\delta \in \NN$, if $\delta$ satisfies both of (1) and (2) in Proposition \ref{prop:hyp sp}. A connected graph $X$ is called \emph{hyperbolic}, if the geodesic metric space $(X,d_X)$ is hyperbolic.
\end{defn}

\begin{defn}\label{def:fine graph}
    Let $X$ be a simplicial graph. The graph $X$ is called \emph{fine}, if $|\C_X(e,n)|<\infty$ for any $e \in E(X)$ and $n\in\NN$ (see Definition \ref{def:concepts in graph theory}). The graph $X$ is called \emph{uniformly fine}, if there exists a function $f \colon \NN \to \NN$ such that $|\C_X(e,n)| \le f(n)$ for any $e \in E(X)$ and $n\in\NN$.
\end{defn}

\begin{defn}\label{def:Gromov boundary}
    Let $X$ be a fine hyperbolic graph. A \emph{geodesic ray} $p$ (resp. \emph{bi-infinite geodesic path}) in $X$ is a path $p=(p_0,p_1,\cdots)$ (resp. $p=(p_i)_{i\in\ZZ}$) such that every finite subpath is geodesic. Given a geodesic ray $p = (p_0,p_1,\cdots)$ (resp. $p' = (p'_0,p'_{-1},\cdots)$) and $n \in \NN$, we denote the subpath $(p_n, p_{n+1},\cdots)$ (resp. $(p'_{-n}, p'_{-(n+1)},\cdots)$) by $p_{[n,\infty)}$ (resp. by $p'_{(-\infty,-n]}$). Two geodesic rays $p$ and $q$ in $X$ are defined to be equivalent (and denoted $p \sim q$) if their Hausdorff distance is finite i.e. there exists $r \in \RR_{>0}$ such that $p \subset \N_X(q,r)$ and $q \subset \N_X(p,r)$. The Gromov boundary $\partial X$ of $X$ is defined by
    \[
    \partial X = \{\text{geodesic rays}\} / \sim.
    \]
    We define $\Delta X$ by $\Delta X = V(X) \cup \partial X$. For $x \in V(X)$ and $\xi \in \partial X$, we denote by $\geo_X(x,\xi)$ the set of all geodesic rays $p=(p_0,p_1,\cdots)$ such that $x=p_0$ and $\xi=[p]$. For $\xi,\eta \in \partial X$, we denote by $\geo_X(\xi,\eta)$ the set of all bi-infinite geodesic paths $p=(p_i)_{i\in \ZZ}$ such that $\xi=[(p_0,p_{-1},p_{-2},\cdots)]$ and $\eta=[(p_0,p_1,p_2,\cdots)]$.
\end{defn}

\begin{rem}\label{rem:Gromov boundary}
    The set $\partial X$ in Definition \ref{def:Gromov boundary} coincides with the sequential boundary of $X$ (see \cite[Definition 3.12]{BH99}) in the case of fine hyperbolic graphs, although the geodesic boundary and the sequential boundary of a hyperbolic space don't necessarily coincide in general.
\end{rem}

Lemma \ref{lem:induced subgraph of geodesics is locally finite} is not difficult to see from \cite[Lemma 8.2]{Bow12}.

\begin{lem}\label{lem:induced subgraph of geodesics is locally finite}
    Let $X$ be a fine hyperbolic graph. For any $a,b \in \Delta X$, the set $\geo_X(a,b)$ is non-empty and the induced subgraph on the set $\bigcup_{p \,\in\, \geo_X(a,b)} V(p)$ is locally finite.
\end{lem}

\begin{defn}
    Let $X$ be a fine hyperbolic graph. For $a \in \Delta X$ and $I \subset E(X)$, we define $P(a,I),P'(a,I) \subset \Delta X$ by
    \begin{align*}
        P(a,I)&=\{b \in \Delta X \mid \forall\, p \in \geo_X(a,b), I\cap E(p) = \emptyset \},\\
        P'(a,I)&=\{b \in \Delta X \mid \exists\, p \in \geo_X(a,b), I\cap E(p) = \emptyset \}.
    \end{align*}
    By \cite[Section 8]{Bow12}, there exists a unique topology on $\Delta X$ such that for any $a\in \Delta X$, the family $\{P(a,I) \mid \text{$I \subset E(X)$ : finite} \}$ is a neighborhood basis of $a$. This topology is called \emph{Bowditch topology}.
\end{defn}

\begin{rem}
    For any $a\in \Delta X$, the family $\{P'(a,I) \mid \text{$I \subset E(X)$ : finite} \}$ is also a neighborhood basis of $a$ in Bowditch topology (see \cite[Section 8]{Bow12}).
\end{rem}

\begin{rem}\label{rem:convergence in Bowditch topology}
    It's not difficult to see that given $\xi \in \partial X$ and a net $(x_i)_{i \in I}$ in $\Delta X$, the net $(x_i)_{i\in I}$ converges to $a$ in Bowditch topology if and only if $\lim_{i \to \infty}(x_i,\xi)_o = \infty$ for some (equivalently, any) $o \in V(X)$. See \cite[Definition 3.15]{BH99} for the definition of the Gromov product $(x_i,\xi)_o$.
\end{rem}

\begin{prop}{\rm \cite[Proposition 8.4, Proposition 8.6]{Bow12}}
    Let $X$ be a fine hyperbolic graph, then Bowditch topology on $\Delta X$ is Hausdorff and compact.
\end{prop}

\subsection{Convergence actions}\label{subsec:Convergence actions}

Readers are referred to \cite{Bow99a} for details of convergence actions.

\begin{defn}
    A net $(x_i)_{i\in I}$ in a set $X$ is called \emph{wandering} if for any $x \in X$, there exists $i_0 \in I$ such that $x_i \neq x$ for any $i \ge i_0$.
\end{defn}

\begin{defn}\label{def:convergence action}
    Let $M$ be a compact Hausdorff space with $|M| \ge 3$ and $G$ be a group acting on $M$ homeomorphically. For a net $(g_i)_{i\in I}$ in $G$ and $a,b \in M$, we denote
    \begin{align*}
        g_i|_{M\setminus \{a\}} \twoheadrightarrow b
    \end{align*}
    if the map $(g_i)_{i\in I}$ locally uniformly converges to $b$ on $M\setminus \{a\}$ (i.e. for any compact set $K \subset M\setminus \{a\}$ and any open neighborhood $U$ of $b$ in $M$, there exists $i_0\in I$ such that $g_i(K)\subset U$ for any $i \ge i_0$). A net $(g_i)_{i\in I}$ in $G$ is called a \emph{collapsing net} if there exist $a,b \in M$ such that $g_i|_{M\setminus \{a\}} \twoheadrightarrow b$. The action $G \act M$ is called a \emph{convergence action} if every wandering net in $G$ has a collapsing subset.
\end{defn}

\begin{rem}
     In the definition of collapsing net above, we don't require $a \neq b$.
\end{rem}

\begin{rem}
   In the definition of a convergence action $G \act M$, we always assume that $M$ is a compact Hausdorff space with $|M| \ge 3$ and the action is homeomorphic.
\end{rem}

\begin{rem}\label{rem:equivalent definition using distint triples}
    By \cite[Proposition 1.1]{Bow99a}, the action $G\act M$ is a convergence action if and only if the action of $G$ on the space $\Theta(M)$ of distinct triples in $M$ is properly discontinuous i.e. for any compact subsets $K,L$ in $\Theta(M)$, the set $\{g \in G \mid g K \cap L \neq \emptyset\}$ is finite (see \cite[Section 1]{Bow99a}).
\end{rem}

\begin{defn}
    Suppose that a group $G$ acts on a compact Hausdorff space $M$ with $|M| \ge 3$ as a convergence action. A point $x \in M$ is called a \emph{limit point} if there exists $y \in M$ and a net $(g_i)_{i\in I}$ in $G$ such that $(g_i y)_{i \in I} \subset M\setminus \{x\}$ and $g_i y \to x$ in $M$. The \emph{limit set} $\Lambda(G)$ is defined as the set of all limit points in $M$. An element $g \in G$ is called \emph{elliptic} if $g$ has finite order, \emph{parabolic} if $g$ has infinite order and $|\fix_M(g)|=1$, and \emph{loxodromic} if $g$ has infinite order and $|\fix_M(g)|=2$. A point $x \in M$ is called a \emph{conical limit point} if there exist a net $(g_i)_{i\in I}$ in $G$ and two distinct points $a,b \in M$ such that $g_ix \to a$ and $g_iy \to b$ for any $y \in M\setminus \{x\}$. A point $x \in M$ is called a \emph{bounded parabolic point} if the stabilizer $\stab_G(x)$ is infinite, contains no loxodromic element, and acts cocompactly on $\Lambda(G) \setminus \{x\}$. The action $G \act M$ is called
    \begin{itemize}
        \item
        \emph{uniform} if the action of $G$ on the space $\Theta(M)$ of distinct triples is cocompact i.e. there exists a compact set $K \subset \Theta(M)$ such that $\bigcup_{g \in G} gK=\Theta(M)$,
        \item 
        \emph{geometrically finite} if every point $x \in \Lambda(G)$ is either a conical limit point or a bounded parabolic point,
        \item 
        \emph{elementary} if $G$ is finite or there exists a $G$-invariant subset $A$ with $1 \le |A| \le 2$.
    \end{itemize}
\end{defn}

\begin{rem}\label{rem:action on limit set is minimal}
    It is known that the limit set $\Lambda(G)$ is a $G$-invariant closed set in $M$ and that if the convergence action $G\act M$ is non-elementary, then the action $G \act \Lambda(G)$ is minimal (i.e. for any nonempty $G$-invariant closed set $A$ in $\Lambda(G)$, we have $A=\Lambda(G)$) (see \cite[Section 2]{Bow99a}).
\end{rem}

\begin{rem}\label{rem:parabolic fixed point}
    By \cite[Lemma 2.1]{Bow99a}, every element of $G$ is elliptic, parabolic, or loxodromic. If $x\in M$ is a fixed point of a parabolic element $g \in G$ (i.e. $gx=x$), then $\stab_G(x)$ is infinite by $\la g \ra \subset \stab_G(x)$ and contains no loxodromic element by \cite[Lemma 2.2]{Bow99a}. It's also known that a loxodromic element cannot share a fixed point with a parabolic element and that if two loxodromic elements share a fixed point, then they have both fixed points in common.
\end{rem}

\begin{rem}\label{rem:conical limit point and bounded parabolic point are limit points}
    By \cite[Lemma 3.1]{Bow99a}, every conical limit point is a limit point. It's also not difficult to see that every bounded parabolic point is a limit point.
\end{rem}

\begin{rem}\label{rem:equivalent definition of conical limit point}
    A point $x \in M$ is a conical limit point if and only if for some $y \in M\setminus\{x\}$ (equivalently for any $y \in M\setminus\{x\}$) there exist nets $(x_i)_{i\in I}$ in $M\setminus\{x,y\}$ and $(g_i)_{i\in I}$ in $G$ such that $x_i \to x$ and $g_i(x,y,x_i)$ remains in a compact subset of $\Theta^0(M)$, where $\Theta^0(M)$ is the space of distinct ordered triples (see \cite[Section 3]{Bow99a}).
\end{rem}

\begin{rem}\label{rem:uniform convergence action}
    It is known that if a convergence action $G \act M$ is uniform, then every non-isolated point in $M$ is a conical limit point (see \cite[Section 3]{Bow99a}). When $M$ is an infinite compact metric space, the converse also holds by \cite[Theorem 1A]{Tuk98}.
\end{rem}

\subsection{Relative bi-exactness and strong solidity}

Readers are referred to \cite{BO08} and \cite{AP} for details of von Neumann algebras. Definition \ref{def:strongly solid} comes from \cite[p.731]{OP10}.

\begin{defn}\label{def:strongly solid}
    A $\rm{II}_1$ factor $M$ is called \emph{strongly solid} if for every diffuse amenable von Neumann subalgebras $A \subset M$, the normalizer $\N_M(A)''$ generated by $\N_M(A) = \{u \in \mathcal{U}(M)\mid uAu^*=A\}$ is amenable.
\end{defn}

\begin{defn}\label{def:prime}
    A von Neumann algebra $M$ is called \emph{prime} if whenever $M$ is isomorphic to a tensor product $P\tensor Q$ of von Neumann subalgebras $P$ and $Q$, then either $P$ or $Q$ is finite dimensional. 
\end{defn}

In the remainder of this section, we explain relative bi-exactness of groups. For a group $G$, the Banach spaces $\ell^p(G) \,(1\le p <\infty)$, $\ell^\infty(G)$, and $c_0(G)$ are defined by
\begin{align*}
    \ell^p(G)&=\{f \colon G \to \CC \mid \sum_{g \in G}|f(g)|^p<\infty \} ~~(1 \le p < \infty), \\
    \ell^\infty(G)&=\{f \colon G \to \CC \mid \sup_{g \in G}|f(g)|<\infty \}, \\
    c_0(G) &= \{f \in \ell^\infty(G) \mid \forall \, \e>0,\, \#\{g \in G \mid |f(g)|\ge \e\}<\infty \}.
\end{align*}

\begin{defn}\label{def:topological amenability}
    Let $G$ be a group and $K$ be a compact Hausdorff space. A homeomorphic action $G \act K$ is called \emph{topologically amenable} if for any finite set $F \subset G$ and any $\e \in \RR_{>0}$, there exists a continuous map $m \colon K \to \mathrm{Prob}(G)$ such that
    \[
    \max_{g \in F}\sup_{x\in K}\|m(gx) - g.m(x)\|_1< \e.
    \]
\end{defn}

\begin{rem}
    In Definition \ref{def:topological amenability}, the continuity of $m$ is with respect to the weak-$\ast$ topology of $\mathrm{Prob}(G) \subset \ell^1(G)$, where $\ell^1(G)$ is identified with the space of all bounded linear $\CC$-valued functions on $c_0(G)$.
\end{rem}

\begin{lem}\label{lem:G-equivariant map passes topological amenability}
    Suppose that $X$ and $Y$ are compact Hausdorff spaces and a group $G$ acts on $X$ and $Y$ homeomorphically. If the action $G \act X$ is topologically amenable and there exists a $G$-equivariant continuous map $\phi\colon Y \to X$, then the action $G \act Y$ is topologically amenable.
\end{lem}

\begin{proof}
    This follows by considering $m \circ \phi \colon Y \to \Prob(G)$, where $m \colon X\to \Prob(G)$ is the map in Definition \ref{def:topological amenability} taken for a finite set $F \subset G$ and $\e \in \RR_{>0}$.
\end{proof}

\begin{defn}
    Let $G$ be a group and $\G$ be a family of subgroups of $G$. A set $S \subset G$ is called \emph{small relative to} $\G$ if there exist $n\in\NN$, $s_1,\cdots,s_n,t_1,\cdots,t_n \in G$, and $H_1,\cdots,H_n \in \G$ such that $S \subset \bigcup_{i=1}^n s_iH_it_i$. For a function $f \colon G \to \CC$ and $a \in \CC$, we define
    \begin{align*}
        \text{
        $\lim_{g\to \infty/\G}f(g) = a \iff \text{$\forall\, \e \in \RR_{>0}$, the set $\{g\in G \mid |f(g)-a| \ge \e\}$ is small relative to $\G$}$.
        }
    \end{align*}
\end{defn}

\begin{rem}
    $c_0(G) = \{f \in \ell^\infty(G) \mid \lim_{g \to \infty/\{\la 1 \ra\}} f(g)=0\}$.
\end{rem}

\begin{defn}\label{def:bi-exact group}
    A group $G$ is called \emph{exact} if there exists a compact Hausdorff space on which $G$ acts topologically amenably. A group $G$ is called \emph{bi-exact relative to a collection of subgroups} $\G$, if $G$ is exact and there exists a map $\mu\colon G\to \mathrm{Prob}(G)$ such that for every $s,t\in G$, we have $\lim_{x\to \infty/\G}\|\mu(sxt)-s.\mu(x)\|_1=0$. If $G$ is bi-exact relative to $\{\la 1 \ra\}$, then $G$ is called \emph{bi-exact}.
\end{defn}

\begin{defn}\label{def: boundary small at infinity}
    Let $G$ be a group and $\G$ be a collection of subgroups of $G$. Define $c_0(G;\G), A(G;\G) \subset \ell^\infty(G)$ by
    \begin{align*}
        c_0(G;\G) &= \{f \in \ell^\infty(G) \mid \lim_{g \to \infty/\G}f(g)=0\},\\
        A(G; \G) &= \{f \in \ell^\infty(G) \mid \forall\, t \in G,\, f-f^t \in c_0(G;\G) \},    
    \end{align*}
    where $f^t \in \ell^\infty(G)$ is defined by $f^t(g)=f(gt^{-1})$. Note that $A(G;\G)$ is a unital commutative $C^*$-algebra that is invariant under the left $G$-action. Define $\overline{G}^{\G}$ to be the Gelfand dual of $A(G;\G)$. Note that $\overline{G}^{\G}$ is a compact Hausdorff space with $G$-action and satisfies $C(\overline{G}^{\G})=A(G;\G)$, where $C(\overline{G}^{\G})$ is the set of all continuous maps from $\overline{G}^{\G}$ to $\CC$ (see \cite[Theorem 11.18]{Rud91}).
\end{defn}

Proposition \ref{prop:equivalent condition of biexact group} follows from the proof of \cite[Proposition 15.2.3]{BO08}

\begin{prop}\label{prop:equivalent condition of biexact group}
    A group $G$ is bi-exact relative to a collection of subgroups $\G$ of $G$ if and only if the action $G \act \overline{G}^{\G}$ is topologically amenable.
\end{prop}

\begin{prop}\label{prop:Ozawa Prop 11}{\rm \cite[Proposition 11]{Oza06}}
    Suppose that $X$ and $Y$ are compact Hausdorff spaces, $K$ is a countable discrete set, and $G$ is a countable group acting on $X$, $Y$, and $K$ homeomorphically. Suppose that the two conditions below hold.
    \begin{itemize}
        \item[(1)]
        For any $\e \in \RR_{>0}$ and any finite set $F \subset G$, there exists a Borel map $\zeta \colon X \to \Prob(K)$ (i.e. the function $X \ni x \mapsto \zeta(x)(a) \in \RR$ is Borel for every $a \in K$) such that $\max_{g \in F}\sup_{x \in X}\|\zeta(gx) - g\zeta(x)\|_1 < \e$.
        \item[(2)]
        For any $a\in K$, the action $\stab_G(a) \act Y$ is topologically amenable.
    \end{itemize}
    Then, the diagonal action $G \act X \times Y$ (i.e. $g(x,y)=(gx,gy)$) is topologically amenable.
\end{prop}

\section{Convergence action on the extension graph}
\label{sec:Convergence action on the extension graph}

The goal of this section is to prove Theorem \ref{thm:intro convergence action}, which corresponds to Theorem \ref{thm:action on extension graph is convergence action} (1), (6), (7). We start with proving Proposition \ref{prop:action on fine graph with finite edge stabilizer is convergence action}. Although Proposition \ref{prop:action on fine graph with finite edge stabilizer is convergence action} might be known to experts, we record its self-contained proof because I couldn't find a reference stating this result. The lemmas from Lemma \ref{lem:net converging to a vertex} up to Lemma \ref{lem:absorbing dynamics imply convergence action} are preparation to prove  Proposition \ref{prop:action on fine graph with finite edge stabilizer is convergence action}. See Definition \ref{def:concepts in graph theory} for $\elk_X(a)$ and Section \ref{subsec:Fine hyperbolic graphs} for relevant notations.

\begin{lem}\label{lem:net converging to a vertex}
    Let $X$ be a fine hyperbolic graph and let $a \in V(X)$. For any finite set $A \subset \elk_X(a)$, there exists a finite set $F \subset \elk_X(a)$ satisfying the following condition {\rm (\textreferencemark)}. 
    \begin{itemize}
        \item[{\rm (\textreferencemark)}]
        For any $b \in \Delta X$ and $c \in \Delta X \setminus P(a,A)$ such that there exists $p\in \geo_X(b,c)$ with $a \notin V(p)$, we have $b \in \Delta X \setminus P'(a,F)$.
    \end{itemize}
\end{lem}

\begin{proof}
    Let $X$ be $\delta$-hyperbolic with $\delta \in \NN$. By \cite[Lemma 8.2, Lemma 8.3]{Bow12}, there exists a finite set $A_0 \subset \elk_X(a)$ such that $P'(a,A_0) \subset P(a,A)$. We define $F$ by $F=A_0 \cup\{f \in E(X) \mid e \in A_0, \gamma \in \C_X(e, 18\delta), f \in E(\gamma), f_-=a \}$. The set $F$ is finite since $X$ is fine. By \cite[Lemma 8.2, Lemma 8.3]{Bow12}, there exists a finite set $F_0 \subset \elk_X(a)$ such that $P'(a,F_0) \subset P(a,F)$. 
    We will show that $F_0$ satisfies the condition (\textreferencemark) above. 
    
    Let $b \in \Delta X$ and $c \in \Delta X \setminus P(a,A)$ such that there exists $p\in \geo_X(b,c)$ satisfying $a \notin V(p)$. We'll discuss only the case $\{b,c\} \subset \partial X$, because the other cases can be proved similarly. There exists $q \in \geo_X(a,c)$, $r \in \geo_X(a,b)$, and $b_0, c_0 \in V(p)$ such that $p=p_{(-\infty,b_0]}p_{[b_0,c_0]}p_{[c_0,\infty)}$, $q_{[c_0,\infty)}=p_{[c_0,\infty)}$, and $r_{[b_0,\infty)}=p^{-1}_{[b_0,\infty)}$. By taking $b_0$ and $c_0$ that make $d_{\Gamma^e}(a,b_0)+d_{\Gamma^e}(a,b_0)$ minimal among such $q,r,b_0,c_0$, we can assume $V(p_{[b_0,c_0]}\cap V(q_{[a,c_0]})=\{c_0\}$ and $V(p_{[b_0,c_0]}\cap V(r_{[a,b_0]})=\{b_0\}$. By $a \notin V(p)$, we have $a \notin \{b_0,c_0\}$. Let $e_1 \in E(q)$ and $e_2 \in E(r)$ satisfy $e_{1-}=e_{2-}=a$. By $c \in \Delta X \setminus P(a,A) \subset \Delta X \setminus P'(a,A_0)$, we have $e_1 \in A_0$.

    (i) When $(V(q_{[a,c_0]})\cup V(r_{[a,b_0]})) \setminus \{a\} \neq \emptyset$, take $v \in (V(q_{[a,c_0]})\cup V(r_{[a,b_0]})) \setminus \{a\}$ such that $d_{\Gamma^e}(a,v)$ is the minimum among all elements in $(V(q_{[a,c_0]})\cup V(r_{[a,b_0]})) \setminus \{a\}$. If $d_{\Gamma^e}(a,v)=1$, then we have $e_2=e_1 \in A_0$. If $1 < d_{\Gamma^e}(a,v) \le 2\delta$, then the loop $q_{[a,v]}r^{-1}_{[v,a]}$ is a circuit of length at most $4\delta$. Hence, $e_2 \in F$. If $2\delta < d_{\Gamma^e}(a,v)$, then take $w_1 \in V(q_{[a,v]})$ and $w_2 \in V(r_{[a,v]})$ such that $d_{\Gamma^e}(a,w_1)=d_{\Gamma^e}(a,w_2)=2\delta$. We have $d_{\Gamma^e}(w_1,w_2) \le \delta$ since the loops $q_{[a,v]}r^{-1}_{[v,a]}$ a simple geodesic bigon. Take $\alpha \in \geo_{\Gamma^e}(w_1, w_2)$. By $|\alpha| < d_{\Gamma^e}(a,w_1)$, we can take a subpath $\alpha'$ of $\alpha$ such that $\alpha'_- \in V(q_{[a,w_1]})$ and $\alpha'_+ \in V(r_{[a,w_2]})$ and the loop $q_{[a,\alpha'_-]}\alpha' r^{-1}_{[\alpha'_+,a]}$ is a circuit. By $|q_{[a,\alpha'_-]}\alpha' r^{-1}_{[\alpha'_+,a]}| \le 2\delta+\delta+2\delta$, we have $e_2 \in F$.

    (ii) When $V(q_{[a,c_0]})\cup V(r_{[a,b_0]}) = \{a\}$, the loop $q_{[a, c_0]} p^{-1}_{[c_0, b_0]} r^{-1}_{[b_0, a]}$ is a simple geodesic triangle. If $(b_0,c_0)_a > 2\delta$ (see Definition \ref{def:gromov product}), then there exist $w_1 \in V(q_{[a,c_0]})$ and $w_2 \in V(r_{[a,b_0]})$ such that $d_{\Gamma^e}(a,w_1)=d_{\Gamma^e}(a,w_2)=2\delta$ and $d_{\Gamma^e}(w_1,w_2) \le \delta$. Hence, in the same way as case (i), we can see $e_2 \in F$. 
    If $(b_0,c_0)_a \le 2\delta$, then let $z_1 \in V(q_{[a,c_0]})$, $z_2 \in V(r_{[a,b_0]})$, and $z_3 \in V(p_{[b_0,c_0]})$ satisfy $d_{\Gamma^e}(a,z_1)=d_{\Gamma^e}(a,z_2)=(b_0,c_0)_a$ and $d_{\Gamma^e}(c_0,z_3)=(b_0,a)_{c_0}$. When $(b_0,a)_{c_0} > 3\delta$, we take $w_1 \in V(q_{[z_1,c_0]})$ and $w'_1 \in V(p_{[z_3,c_0]})$ satisfying $d_{\Gamma^e}(z_1,w_1)=d_{\Gamma^e}(z_3,w'_1)=3\delta$, and when $(b_0,a)_{c_0} \le 3\delta$, we define $w_1,w'_1$ by $w_1=w'_1=c_0$. We have $d_{\Gamma^e}(w_1,w'_1) \le \delta$. Take $\alpha_1 \in \geo_{\Gamma^e}(w_1, w_1')$. We can take a subpath $\alpha'_1$ of $\alpha_1$ that meets $q_{[a,w_1]}$ in a single point $\alpha'_{1-}$ and meets $p_{[b_0,w'_1]}$ in a single point $\alpha'_{1+}$. Similarly, when $(c_0,a)_{b_0} > 3\delta$, we take $w_2 \in V(r_{[z_2,b_0]})$ and $w'_2 \in V(p_{[b_0,z_3]})$ satisfying $d_{\Gamma^e}(z_2,w_2)=d_{\Gamma^e}(z_3,w'_2)=3\delta$, and when $(c_0,a)_{b_0} \le 3\delta$, we define $w_2,w'_2$ by $w_2=w'_2=b_0$. Also, take $\alpha_2 \in \geo_{\Gamma^e}(w_2, w_2')$ and a subpath $\alpha'_2$ of $\alpha_2$ that meets $r_{[a,w_2]}$ in a single point $\alpha'_{2-}$ and meets $p_{[w'_2,c_0]}$ in a single point $\alpha'_{2+}$. We can see that the loop $q_{[a,\alpha'_{1-}]} \alpha'_1 p^{-1}_{[\alpha'_{1+},\alpha'_{2+}]} \alpha_2^{\prime -1} r^{-1}_{[\alpha'_{2-}, a]}$ is a circuit of length at most $18\delta \,(= 5\delta+\delta+6\delta+\delta+5\delta)$. Hence, $e_2 \in F$.

    By (i) and (ii), $e_2 \in F$. Recall $r \in\geo_X(a,b)$. Hence, $b \in \Delta X \setminus P(a,F)\subset \Delta X \setminus P'(a,F_0)$.
\end{proof}

\begin{cor}\label{cor:net converging to a vertex}
    Let $X$ be a fine hyperbolic graph. Suppose that a net $(x_i)_{i \in I}$ of vertices in $V(X)$ converges to $a \in V(X)$. Then, for any finite set $A \subset \elk_X(a)$, there exists $i_0 \in I$ such that for any $i \ge i_0$, $c \in \Delta X \setminus P(a,A)$, and $p \in \geo_X(x_i, c)$, we have $a \in V(p)$.
\end{cor}

\begin{proof}
    Let $A \subset \elk_X(a)$ be finite. There exists a finite set $F \subset \elk_X(a)$ satisfying the condition (\textreferencemark) of Lemma \ref{lem:net converging to a vertex} for $a$ and $A$. By $x_i \to a$, there exists $i_0 \in I$ such that $x_i \in P(a,F)$ for any $i \ge i_0$. For any $i \ge i_0$, $c \in \Delta X \setminus P(a,A)$ and $p \in \geo_X(x_i,c)$, we have $a \in V(p)$ by $x_i \in P(a,F) \subset P'(a,F)$ and the condition (\textreferencemark).
\end{proof}

\begin{lem}\label{lem:extension graph has absorbing dynamics}
    Suppose that $X$ is a fine hyperbolic graph and $G$ is a group acting on $X$ such that $\stab_G(e) \,(=\stab_G(e_-)\cap \stab_G(e_+))$ is finite for any $e \in E(X)$. Let $(g_i)_{i \in I}$ be a wandering net in $G$ and let $a \in V(X)$. If the net $(g_ia)_{i \in I}$ converges to $b \in \Delta X$, then for any $x \in V(X)$, $(g_ix)_{i \in I}$ converges to $b$ as well.
\end{lem}

\begin{proof}
    If $b \in \partial X$, then it's not difficult to show $\lim_{i \to \infty} g_ix=b$ by using $\sup_{i \in I}d_X(g_ia,g_ix) \le d_X(a,x)$. In the following, we assume $b \in V(X)$. Suppose for contradiction that $(g_ix)_{i \in I}$ doesn't converge to $b$. Since $\Delta X$ is compact, there exist $c \in \Delta X\setminus\{b\}$ and a subnet $(g_jx)_{j \in J}$ of $(g_ix)_{i \in I}$ such that $(g_jx)_{j \in J}$ converges to $c$. If $c \in \partial X$, then this implies $b=\lim_{j \to \infty}g_ja = c$ by $\sup_{j \in J}d_X(g_ja,g_jx) \le d_X(a,x)$ as above, which contradicts the assumption $b \in V(X)$. Hence, $c \in V(X) \setminus \{b\}$. Fix $\alpha \in \geo_X(b,c)$ and let $e_1, e_2 \in E(\alpha)$ be the edges satisfying $e_{1-}=b$ and $e_{2+}=c$. By applying Corollary \ref{cor:net converging to a vertex} to $(g_ja)_{j \in J}$, $b$, and $c \in \Delta X \setminus P(b,\{e_1\})$, there exists $j_0 \in J$ such that for any $j \ge j_0$ and any $p \in \geo_X(g_ja, c)$, we have $b \in V(p)$. In particular, we have $(g_ja)_{j \ge j_0} \subset \Delta X \setminus P(c, \{e_2^{-1}\})$. Hence, by applying Corollary \ref{cor:net converging to a vertex} to $(g_jx)_{j \in J}$, $c$, and $(g_ja)_{j \ge j_0} \subset \Delta X \setminus P(c, \{e_2^{-1}\})$, there exists $j_1 \in J$ such that for any $j' \ge j_1$, $j \ge j_0$, and $p \in \geo_X(g_ja, g_{j'}x)$, we have $c \in V(p)$. Take $j_2 \in J$ satisfying $j_2 \ge j_0$ and $j_2 \ge j_1$, then for any $j \ge j_2$ and any $p \in \geo_X(g_ja, g_jx)$, we have $\{b,c\} \subset V(p)$. In particular, fix $\beta \in \geo_X(a,x)$, then we have $\{b,c\}\subset V(g_j\beta)$ with $d_X(g_ja,b) < d_X(g_ja,c)$ for any $j \ge j_2$. By $d_X(g_ja,b) \le d_X(g_ja,g_jx)=d_X(a,x)$ for any $j \in J$, there exist $n \in \NN\cup\{0\}$ and a subnet $(g_j)_{j \in J_1}$ of $(g_j)_{j \ge j_2}$ such that $d_X(g_ja,b)=n$ for any $j \in J_1$. Let $b_0, c_0 \in V(\beta)$ satisfy $d_X(a,b_0)=n$ and $d_X(a,c_0)=n+d_X(b,c)$, then we have $g_jb_0=b$ and $g_jc_0=c$ for any $j \in J_1$. Since $X$ is fine, the set $\geo_X(b,c)$ is finite by \cite[Lemma 8.2]{Bow12}. Hence, there exist $\gamma \in \geo_X(b,c)$ and a subnet $(g_j)_{j \in J_2}$ of $(g_j)_{j \in J_1}$ such that $g_j\beta_{[b_0,c_0]}=\gamma$ for any $j \in J_2$. Let $e_3 \in E(\beta_{[b_0,c_0]})$ be the edge satisfying $e_{3-}=b_0$ and fix $j_3 \in J_2$, then we have $g_{j_3}^{-1}g_j \in \stab_G(e_3)$ for any $j \in J_2$. Since $(g_j)_{j \in J_2}$ is a subnet of the wandering net $(g_i)_{i \in I}$, the net $(g_j)_{j \in J_2}$ is also wandering. This and $(g_j)_{j \in J_2} \subset g_{j_3}\stab_G(e_3)$ contradict that $\stab_G(e_3)$ is finite. Thus, $(g_ix)_{i \in I}$ converges to $b$.
\end{proof}

\begin{lem}\label{lem:absorbing dynamics imply convergence action}
    Let $X$ be a fine hyperbolic graph and $G$ be a group acting on $X$. Suppose that for any wandering net $(g_i)_{i \in I}$ in $G$ and $a,b \in V(X)$ satisfying $\lim_{i \to \infty}g_ia = b$ in $\Delta X$, and any $x \in V(X)$, we have $\lim_{i \to \infty}g_ix = b$ in $\Delta X$. Then, the action $G \act \Delta X$ is a convergence action.
\end{lem}

\begin{proof}
    Fix $v \in V(X)$ and let $(g_i)_{i \in I}$ be a wandering net in $G$. Since $\Delta X$ is compact, there exist a subnet $(g_j)_{j \in J}$ of $(g_i)_{i \in I}$ and $a,b \in \Delta X$ such that $\lim_{j \to \infty}g_jv = a$ and $\lim_{j \to \infty}g_j^{-1} v = b$. Note that $(g_j)_{j \in J}$ is wandering since $(g_i)_{i \in I}$ is wandering.

    (i) When $\{a, b\} \subset V(X)$, we have $\lim_{j \to \infty}g_j b = a$ and $\lim_{j \to \infty}g_j^{-1} a = b$ by our condition on the action $G \act X$. Let $A \subset \elk_X(a)$ and $B \subset \elk_X(b)$ be finite sets. By taking a subnet of $(g_j)_{j \in J}$ depending on whether $g_j b\neq a$ eventually as $j \to \infty$ or not, we can assume $(\forall j \in J,\, g_j b\neq a) \vee (\forall j\in J,\, g_j b= a)$ without loss of generality.
    
    (i-1) When $\forall j \in J,\, g_j b\neq a$, by $\lim_{j \to \infty}g_j^{-1} a = b$ and Corollary \ref{cor:net converging to a vertex}, there exists $j_0 \in J$ such that for any $j \ge j_0$, $c \in \Delta X \setminus P(b,B)$ and $p' \in \geo_X(g_j^{-1} a,c)$, we have $b \in V(p')$. Hence, for any $j \ge j_0$, $c \in \Delta X \setminus P(b,B)$, and $p \in \geo_X(a, g_j c)$, we have $g_j b \in V(p)$. Meanwhile, by $\lim_{j \to \infty}g_j b = a$, there exists $j_1 \in J$ such that $g_j b \in P(a,A)$ for any $j \ge j_1$. Take $j_2 \in J$ satisfying $j_2 \ge j_0$ and $j_2 \ge j_1$. Let $j \ge j_2$ and $c \in \Delta X \setminus P(b,B)$. Since we have $g_j b \in V(p)$ for any $p \in \geo_X(a,g_j c)$, the subpath $p_{[a,g_j b]}$ is in $\geo_X(a,g_j b)$ and satisfy $E(p)\cap A = E(p_{[a,g_j b]})\cap A = \emptyset$ by $g_j b \neq a$ and $g_j b \in P(a,A)$. Hence, for any $j \ge j_2$ and $c \in \Delta X \setminus P(b,B)$, we have $g_j c \in P(a,A)$

    (i-2) When $\forall j \in J,\, g_j b = a$, since we have $\lim_{j \to \infty} g_j^{-1} e_+=b$ for any $e \in A$ by our condition, we can apply Corollary \ref{cor:net converging to a vertex} to $(g_j^{-1} e_+)_{j\in J}$, $b$, and $B$, and see that there exists $j_0 \in J$ such that for any $j \ge j_0$, $e \in A$, $c \in \Delta X \setminus P(b,B)$, and $p' \in \geo_X(g_j^{-1} e_+,c)$, we have $b \in V(p')$. This implies $a=g_j b \in V(p)$ for any $j \ge j_0$, $e \in A$, $c \in \Delta X \setminus P(b,B)$, and $p \in \geo_X(e_+,g_j c)$. Hence, for any $j \ge j_0$, $c \in \Delta X \setminus P(b,B)$, and $p \in \geo_X(a,g_jc)$, we have $E(p)\cap A = \emptyset$. Thus, $g_j c \in P(a,A)$ for any $j \ge j_0$ and $c \in \Delta X \setminus P(b,B)$.
    
    By (i-1) and (i-2), $(g_j)_{j \in J}$ satisfies $g_j|_{\Delta X \setminus \{b\}} \twoheadrightarrow a$ (see Definition \ref{def:convergence action}). Similarly, we can also see that $(g_j^{-1})_{j \in J}$ satisfies $g_j^{-1}|_{\Delta X \setminus \{a\}} \twoheadrightarrow b$.

    (ii) When $a \in V(X)$ and $b \in \partial X$, we have $\lim_{j \to \infty} g_j a = a$ and $\lim_{j \to \infty} g_j^{-1} a=b$. We first show that $(g_j^{-1})_{j \in J}$ satisfies $g_j^{-1}|_{\Delta X \setminus \{a\}} \twoheadrightarrow b$. Let $A \subset \elk_X(a)$ be finite and let $r \in \NN$. By $\lim_{j \to \infty} g_j^{-1} a=b \in \partial X$, there exists $j_0 \in J$ such that for any $j \ge j_0$, $(b,g_j^{-1} a)_a > r+2\delta$ (see Remark \ref{rem:convergence in Bowditch topology}). Meanwhile, by $\lim_{j \to \infty}g_j a = a$ and Corollary \ref{cor:net converging to a vertex}, there exists $j_1 \in J$ such that for any $j \ge j_1$, $c \in \Delta X \setminus P(a,A)$ and $p' \in \geo_X(g_j a,c)$, we have $a \in V(p')$. This implies $g_j^{-1} a \in V(p)$ for any $j \ge j_1$, $c \in \Delta X \setminus P(a,A)$, and $p \in \geo_X(a,g_j^{-1} c)$. Take $j_2 \in J$ satisfying $j_2 \ge j_0$ and $j_2 \ge j_1$. Let $j \ge j_2$ and $c \in \Delta X \setminus P(a,A)$, then by $\forall\, p \in \geo_X(a,g_j^{-1} c), \, g_j^{-1} a \in V(p)$, we have $(g_j^{-1} a, g_j^{-1} c)_a \ge d_X(a,g_j^{-1} a) \ge (b,g_j^{-1}a)_a > r+2\delta$. This implies
    $
    (b,g_j^{-1}c)_a \ge \min\{(b,g_j^{-1}a)_a,(g_j^{-1}a, g_j^{-1}c)_a\}-\delta > r+2\delta-2\delta = r.
    $ Hence, $(g_j^{-1})_{j \in J}$ satisfies $g_j^{-1}|_{\Delta X \setminus \{a\}} \twoheadrightarrow b$.
    
    We next show that $(g_j)_{j \in J}$ satisfies $g_j|_{\Delta X \setminus \{b\}} \twoheadrightarrow a$. Let $A \subset \elk_X(a)$ be finite and let $r \in \NN$. For $a$ and $A$, there exists a finite set $F \subset \elk_X(a)$ satisfying the condition (\textreferencemark) of Lemma \ref{lem:net converging to a vertex}. By $g_j a \to a$ and $g_j^{-1} a \to b$, there exists $j_0 \in J$ such that for any $j \ge j_0$, we have $g_j a \in P(a,F)$ and $(b, g_j^{-1} a)_a > r+2\delta$. Let $j\ge j_0$ and let $c \in \Delta X$ satisfy $(b, c)_a \le r$. By $(b, g_j^{-1} a)_a > r+2\delta$ and $(b, c)_a \le r$, we have $(g_j^{-1} a, c)_a \le r+2\delta$. This implies $g_j^{-1} a \notin V(p')$ for any $p' \in \geo_X(a,c)$. Indeed, if there exists $p \in \geo_X(a,c)$ such that $g_j^{-1} a \in p$, then we have $(g_j^{-1} a, c)_a \ge d_X(a,g_j^{-1}a) \ge (b, g_j^{-1} a)_a > r+2\delta$, which contradicts $(g_j^{-1} a, c)_a \le r+2\delta$. Hence, $a \notin V(p)$ for any $p \in \geo_X(g_j a, g_j c)$. This, $g_j a \in P(a,F) \subset P'(a,F)$, and the condition (\textreferencemark) imply $g_j c \in P(a,A)$. Hence, $(g_j)_{j \in J}$ satisfies $g_j|_{\Delta X \setminus \{b\}} \twoheadrightarrow a$. 
    
    The argument in the case $a\in \partial X$ and $b \in V(X)$ is the same as case (ii).

    (iii) When $\{a,b\} \subset \partial X$, let $r,r' \in \NN$ and let $c \in \Delta X$ satisfy $(b,c)_v \le r$. By $g_j^{-1} v \to b$ and $g_j v \to a$, there exists $j_0\in J$ such that for any $j \ge j_0$, we have $(b,g_j^{-1} v)_v > r+r'+4\delta$ and $(a,g_j v)_v > r'+2\delta$. Let $j \ge j_0$. By $(b,g_j^{-1} v)_v > r+r'+4\delta$ and $r \ge (b,c)_v \ge \min\{(b,g_j^{-1} v)_v, (g_j^{-1} v, c)_v\} -2\delta$, we have $(g_j^{-1} v, c)_v\le r+2\delta$. This implies
    \begin{align*}
    (g_j v, g_j c)_v
    &=
    (v,c)_{g_j^{-1}v}
    \ge
    d_X(g_j^{-1} v, v) - (g_j^{-1} v, c)_v
    \ge
    (b,g_j^{-1} v)_v - (g_j^{-1} v, c)_v \\
    &>
    (r+r'+4\delta) - (r+2\delta)
    =
    r' + 2\delta.
    \end{align*}
    This and $(a,g_j v)_v > r'+2\delta$ imply $(a, g_j c)_v \ge \min\{(a, g_j v)_v, (g_j v, g_j c)_v\} - 2\delta > r'$. Hence, $(g_j)_{j \in J}$ satisfies $g_j|_{\Delta X \setminus \{b\}} \twoheadrightarrow a$. Similarly, we can also show that $(g_j^{-1})_{j \in J}$ satisfies $g_j^{-1}|_{\Delta X \setminus \{a\}} \twoheadrightarrow b$.
\end{proof}

\begin{prop}\label{prop:action on fine graph with finite edge stabilizer is convergence action}
    Let $X$ be a fine hyperbolic graph and $G$ be a group acting on $X$. If $\stab_G(e) \,(=\stab_G(e_-)\cap \stab_G(e_+))$ is finite for any $e \in E(X)$, then the action $G \act \Delta X$ is a convergence action.
\end{prop}

\begin{proof}
    This follows from Lemma \ref{lem:extension graph has absorbing dynamics} and Lemma \ref{lem:absorbing dynamics imply convergence action}.
\end{proof}

Lemma \ref{lem:geodesic to the same boundary point go to the same plane} below is used to prove Theorem \ref{thm:action on extension graph is convergence action} (7). See Section \ref{subsec:Graph products of groups} for relevant notations.

\begin{lem}\label{lem:geodesic to the same boundary point go to the same plane}
    Suppose that $\Gamma$ is a fine hyperbolic graph with $\girth(\Gamma) > 20$ and that $\G=\{G_v\}_{v\in V(\Gamma)}$ is a collection of non-trivial finite groups. Let $x \in \partial \Gamma^e$ and $g \in \Gamma\G$ and let $p=(x_0,x_1,\cdots)$ be a geodesic ray in $\Gamma^e$ ($\forall n\ge 0, x_n \in V(\Gamma^e)$) such that $\lim_{n \to \infty}x_n = x$ holds and the sequence $(x_n)_{n \ge 0}$ is eventually in $g.\Gamma$. If $q=(y_0,y_1,\cdots)$ is a geodesic ray in $\Gamma^e$ ($\forall n\ge 0, y_n \in V(\Gamma^e)$) with $\lim_{n \to \infty}y_n = x$, then the sequence $(y_n)_{n \ge 0}$ is also eventually in $g.\Gamma$. 
\end{lem}

\begin{proof}
    Let $\Gamma^e$ be $\delta$-hyperbolic with $\delta \in \NN$ and let $N \in \NN$ satisfy $p_{[x_N,\infty)}\subset g.\Gamma$. We first show the statement when $y_0=x_0$. Given $n \in \NN$, define $m \in \NN$ by $m=N+n+\delta+7$ for brevity. We have $d_{\Gamma^e}(x_m,y_m) \le \delta$ by $\lim_{i \to \infty} x_i = \lim_{i \to \infty} y_i =x$. Take $r \in \geo_{\Gamma^e}(x_m,y_m)$. By applying \cite[Corollary 3.36]{Oya24b} and \cite[Proposition 3.38]{Oya24b} to the geodesic triangle $p_{[x_0,x_m]}rq^{-1}_{[y_m,y_0]}$, there exist $m_1,m_2,\ell_1,\ell_2 \in \{0,\cdots,m\}$, $a,b\in V(r)$, and $g_1\in \Gamma\G$ satisfying $m_1 \le m_2$, $\ell_1 \le \ell_2$, and the conditions (i) and (ii) below.
    \begin{itemize}
        \item[(i)]
        $p_{[x_{m_1},x_{m_2}]} \cup q_{[y_{\ell_1},y_{\ell_2}]} \cup r_{[a,b]} \subset g_1.\Gamma$ and $\max\{d_{\Gamma^e}(x_{m_1},y_{\ell_1}), d_{\Gamma^e}(x_{m_2},a), d_{\Gamma^e}(y_{\ell_2},b)\}\le 2$.
        \item[(ii)]
        Either (ii-1) or (ii-2) holds, (ii-1) $x_{m_1} = y_{\ell_1}$, (ii-2) there exist $m_0 \in \{0,\cdots,m_1\}$, $\ell_0 \in \{0,\cdots,\ell_1\}$, and $g_0 \in \Gamma\G$ such that $p_{[x_{m_0},x_{m_1}]} \cup q_{[y_{\ell_0},y_{\ell_1}]} \subset g_0.\Gamma$, $d_{\Gamma^e}(x_{m_1},y_{\ell_1}) \le 2$, and $\min\{|p_{[x_{m_0},x_{m_1}]}|, |q_{[y_{\ell_0},y_{\ell_1}]}|\} \ge 7$.
    \end{itemize}
    We have $m - m_2 = |p_{[x_{m_2}, x_m]}| \le d_{\Gamma^e}(x_{m_2}, a) + d_{\Gamma^e}(a, x_m) \le 2+ d_{\Gamma^e}(x_m,y_m) \le \delta+2$. Similarly, $m - \ell_2 = |q_{[y_{\ell_2}, y_m]}| \le d_{\Gamma^e}(y_{\ell_2}, b) + d_{\Gamma^e}(b, y_m) \le \delta+2$. This and $m=N+n+\delta+7$ imply $\min\{m_2,\ell_2\} \ge N+n+5$.
    
    When $|p_{[x_{m_1}, x_{m_2}]}| \ge 3$. By $m_2-m_1=|p_{[x_{m_1}, x_{m_2}]}| \ge 3$, we have $x_{m_2-3} \in p_{[x_{m_1},x_{m_2}]}$. By $m_2-3 \ge N+n+2 > N$, $p_{[x_N,\infty)}\subset g.\Gamma$, and the condition (i), we have $p_{[x_{m_2-3},x_{m_2}]} \subset g.\Gamma \cap g_1.\Gamma$. Hence, $\diam_{\Gamma^e}(g.\Gamma\cap g_1.\Gamma) \ge 3$. By this and \cite[Remark 3.12 (3)]{Oya24b}, we have $g=g_1$. By $q_{[y_{\ell_1},y_{\ell_2}]} \subset g_1.\Gamma=g.\Gamma$, we have $y_{\ell_2} \in g.\Gamma$. Also, recall $\ell_2 \ge N+n+5$.
    
    When $|p_{[x_{m_1}, x_{m_2}]}| \le 2$. In case (ii-1), we have $\ell_1=m_1 \ge m_2 - 2 \ge N+n+3$ and $y_{\ell_1} = x_{m_1} \in g.\Gamma$ by $m_1 > N$. In case (ii-2), we have $x_{m_1-3} \in p_{[x_{m_0},x_{m_1}]}$ by $|p_{[x_{m_0},x_{m_1}]}| \ge 7$. By $m_1 - 3 \ge N+n$, $p_{[x_N,\infty)}\subset g.\Gamma$, and the condition (ii-2), we have $p_{[x_{m_1-3},x_{m_1}]} \subset g.\Gamma \cap g_0.\Gamma$. Hence, we have $g=g_0$ by \cite[Remark 3.12 (3)]{Oya24b}. By $q_{[y_{\ell_0},y_{\ell_1}]} \subset g_0.\Gamma=g.\Gamma$, we have $y_{\ell_1} \in g.\Gamma$. Also, $\ell_1 \ge m_1 - d_{\Gamma^e}(x_{m_1},y_{\ell_1}) \ge N+n+3 -2$.

    Thus, for any $n \in \NN$, there exists $\ell \in \NN$ such that $\ell > N+n$ and $y_\ell \in g.\Gamma$. Hence, there exists an increasing sequence $(n_i)_{i \in \NN}$ in $\NN$ such that $(y_{n_i})_{i \in \NN} \subset g.\Gamma$. By \cite[Corollary 3.29]{Oya24b}, we have $q_{[y_{n_i},y_{n_{i+1}}]} \subset g.\Gamma$ for any $i \ge 1$. This implies $q_{[y_{n_1,\infty})} \subset g.\Gamma$.
    
    Next, we will show the statement in general case i.e. $y_0 \neq x_0$ can happen. There exists $N_1 \in \NN$ and $\alpha \in \geo_{\Gamma^e}(x_0,y_{N_1})$ such that the path $q'$ defined by $q' =\alpha q_{[y_{N_1},\infty)}$ is a geodesic ray in $\Gamma^e$ (see \cite[Lemma 3.27]{Oya24a}). By applying the above argument to $p$ and $q'$, we can see that $q'$ is eventually in $g.\Gamma$. Thus, $q$ is also eventually in $g.\Gamma$.
\end{proof}

We are now ready to prove Theorem \ref{thm:intro convergence action}, which corresponds to Theorem \ref{thm:action on extension graph is convergence action} (1), (6), (7). In Theorem \ref{thm:action on extension graph is convergence action}, the condition $\diam_\Gamma(\Gamma) > 1$ is natural. Indeed, we can see $\diam_\Gamma(\Gamma) \le 1$ if and only if $|\Delta\Gamma^e| \le 2$ by \cite[Lemma 3.14 (1)]{Oya24b}. However, it is conventional to assume $|M| \ge 3$ in Definition \ref{def:convergence action} because of the equivalent definition of a convergence action using distinct triples (see Remark \ref{rem:equivalent definition using distint triples}). See Definition \ref{def:concepts in graph theory} for $\leaf(\Gamma)$ and Section \ref{subsec:Convergence actions} for relevant definitions.

\begin{thm}\label{thm:action on extension graph is convergence action}
    Suppose that $\Gamma$ is a fine hyperbolic graph with $\girth(\Gamma) > 20$ and $\diam_\Gamma(\Gamma) > 1$ and that $\G=\{G_v\}_{v\in V(\Gamma)}$ is a collection of non-trivial finite groups. Then, the action of $\Gamma\G$ on $\Delta\Gamma^e$ is a convergence action and the following hold.
    \begin{itemize}
        \item[(1)]
        The action $\Gamma\G \act \Delta\Gamma^e$ is non-elementary if and only if $\diam_\Gamma(\Gamma) > 2$.
        \item[(2)]
        We have $\Lambda(\Gamma\G) = \Delta\Gamma^e\setminus\, \Gamma\G.\leaf(\Gamma)$. If $x \in \Gamma\G.\leaf(\Gamma)$, then $x$ is isolated in $\Delta\Gamma^e$.
        \item[(3)]
        The action $\Gamma\G \act \Delta\Gamma^e$ is minimal if and only if $\leaf(\Gamma)=\emptyset$.
        \item[(4)]
        No point in $V(\Gamma^e)$ is a conical limit point.
        \item[(5)]
        A point $v \in V(\Gamma) \setminus \leaf(\Gamma)$ is bounded parabolic if and only if $|\lk_{\Gamma}(v)\setminus \leaf(\Gamma)| < \infty$.
        \item[(6)]
        The action $\Gamma\G \act \Delta\Gamma^e$ is not uniform.
        \item[(7)]
        The action $\Gamma\G \act \Delta\Gamma^e$ is geometrically finite if and only if $|V(\Gamma)\setminus \leaf(\Gamma)| < \infty$.
    \end{itemize}
\end{thm}

\begin{proof}
    Note that $\Gamma^e$ is fine and hyperbolic by \cite[Theorem 1.2 (1)]{Oya24b} and \cite[Theorem 1.2 (3)]{Oya24b}. Since we have $|\stab_{\Gamma\G}(e)| = |G_{e_-}\times G_{e_+}| < \infty$ for any $e \in E(\Gamma)$ by \cite[Remark 3.12 (1)]{Oya24b}, the action $\Gamma\G \act \Delta \Gamma^e$ is a convergence action by Proposition \ref{prop:action on fine graph with finite edge stabilizer is convergence action}.
    
    (1) When $\diam_\Gamma(\Gamma) = 2$, there exists $v \in V(\Gamma)$ such that $V(\Gamma) = \st_\Gamma(v)$, hence $\Gamma\G$ fixes the point $v \in V(\Gamma^e)$. Hence, the action $\Gamma\G \act \Delta\Gamma^e$ is elementary. 
    
     When $\diam_\Gamma(\Gamma) > 2$, $\Gamma$ contains an induced subgraph isomorphic to a line of length $3$ by $\girth(\Gamma) > 20$. Hence, $\Gamma\G$ contains the free group of rank $2$. By \cite[Lemma 2.2]{Bow99a}, $\Gamma\G$ doesn't fix any set of two distinct points in $\Delta \Gamma^e$ setwise since $\Gamma\G$ is not virtually cyclic. Given $x \in V(\Gamma^e)$, let $g \in \Gamma\G$ satisfy $x = g.v(x)$. By \cite[Corollary 3.6]{Oya24b}, we have $\stab_{\Gamma\G}(x) = g \la G_w \mid w \in \st_\Gamma(v(x)) \ra g^{-1}$. Hence, $\Gamma\G \neq \stab_{\Gamma\G}(x)$ for any $x \in V(\Gamma^e)$ by $\diam_\Gamma(\Gamma) > 2$. We next show that $\Gamma\G$ doesn't fix any point in $\partial \Gamma^e$. Let $x\in \partial \Gamma^e$. Fix $v \in V(\Gamma)$ with $|\lk_{\Gamma}(v)|\ge 2$ and a geodesic ray $p \in \geo_{\Gamma^e}(v,x)$. Let $e \in V(p)$ satisfy $e_- = v$. Since $\Gamma^e$ is fine and hyperbolic, by \cite[Lemma 8.2]{Bow12} there exists a finite set $A \subset \elk_{\Gamma^e}(v)$ such that $E(q)\cap A \neq \emptyset$ for any $q \in \geo_{\Gamma^e}(v,x)$. Since $\lk_{\Gamma^e}(v)$ is infinite by \cite[Lemma 3.14 (2)]{Oya24b}, there exists $g\in \stab_{\Gamma\G}(v)$ such that $g.e \notin A$. Hence, we have $g.x \neq x$ by $g.p \in \geo_{\Gamma^e}(v,g.x)$ and $E(g.p)\cap A = \emptyset$. Thus, $\Gamma\G$ doesn't fix any $x \in \partial \Gamma^e$. Since $\Gamma\G$ is infinite and doesn't fix any non-empty subset of cardinality at most $2$ setwise, the action $\Gamma\G \act \Gamma^e$ is non-elementary.

    (2) Let $v \in \leaf(\Gamma)$. By $\diam_\Gamma(\Gamma) > 1$ and Remark \ref{rem:vertex with 0 link}, we have $|\lk_{\Gamma}(v)|=1$. This implies $|\lk_{\Gamma^e}(v)|=1$ by \cite[Lemma 3.14 (1)]{Oya24b}. Hence, $v \in \leaf(\Gamma)$ is isolated in $\Delta\Gamma^e$. Since $G$ acts on $\Delta\Gamma^e$ homeomorphically, every point in $\Gamma\G.\leaf(\Gamma)\,(=\bigcup_{g\, \in \Gamma\G} g.\leaf(\Gamma))$ is isolated in $\Delta\Gamma^e$ as well. This implies $\Lambda(\Gamma\G) \subset \Delta\Gamma^e\setminus\, \Gamma\G.\leaf(\Gamma)$. We will show the converse inclusion. Let $v \in V(\Gamma)\setminus \leaf(\Gamma)$. The subgroup $\stab_{\Gamma\G}(v)$ is infinite since we have $|\lk_{\Gamma}(v)| \ge 2$ and $\Gamma$ contains no triangle. Hence, we can take a wandering net $(g_n)_{n\in \NN}$ in $\stab_{\Gamma\G}(v)$ such that $g_n.v=v$ for any $n \in\NN$. Take $w \in \lk_{\Gamma}(v)$, which is possible by $|\lk_{\Gamma}(v)| \ge 2$, then we have $\lim_{n\to\infty}g_n.w = v$ by Lemma \ref{lem:extension graph has absorbing dynamics}. Note $v \notin \{g_n.w\mid n \in\NN\}$ by \cite[Corollary 3.7]{Oya24b}. Thus, $v \in \Lambda(\Gamma\G)$. Since $\Lambda(\Gamma\G)$ is $\Gamma\G$-invariant, this implies $V(\Gamma^e) \setminus \Gamma\G.\leaf(\Gamma) \subset \Lambda(\Gamma\G)$. Next, let $x \in \partial\Gamma^e$. Fix $v \in V(\Gamma)$ and $p \in \geo_{\Gamma^e}(v,x)$. Given $R \in \NN$, let $y \in V(p)$ and $g \in \Gamma\G$ satisfy $d_{\Gamma^e}(v,y) = R$ and $y = g.v(y)$ (see Remark \ref{rem:notation on extension graph}). We have $|\lk_{\Gamma}(v(y))|\ge 2$ by $|\lk_{\Gamma^e}(y)| \ge 2$ and \cite[Lemma 3.14 (1)]{Oya24b}. Hence, $\stab_{\Gamma\G}(y)$ is infinite. Take a wandering net $(g_n)_{n \in \NN}$ in $\stab_{\Gamma\G}(y)$, then we have $\lim_{n \to \infty} g_n.v = y$ by $\lim_{n \to \infty} g_n.y = y$ and Lemma \ref{lem:extension graph has absorbing dynamics}. By applying Corollary \ref{cor:net converging to a vertex} to $v$, $(g_n.v)_{n\in\NN}$, and $y$, there exists $N \in\NN$ such that for any $n \ge N$ and $q \in \geo_{\Gamma^e}(v, g_n.v)$, we have $y \in V(q)$. This implies $(g_N.v, x)_v \ge R$ (see Remark \ref{rem:convergence in Bowditch topology}). Hence, we have $x \in \overline{\Gamma\G.v}$. This implies $\partial\Gamma^e \subset \Lambda(\Gamma\G)$.

    (3) When $\diam_\Gamma(\Gamma)=2$, the action $\Gamma\G \act \Gamma^e$ is not minimal since $\Gamma\G$ fixes a point $v \in V(\Gamma)$ with $V(\Gamma) = \st_\Gamma(v)$. Also, we have $\leaf(\Gamma)\neq \emptyset$ since $\Gamma$ contains no triangle. When $\diam_\Gamma(\Gamma) > 2$, the statement follows from Theorem \ref{thm:action on extension graph is convergence action} (1) and (2) and Remark \ref{rem:action on limit set is minimal}.

    (4) No point in $\Gamma\G.\leaf(\Gamma)$ is a conical limit point by Theorem \ref{thm:action on extension graph is convergence action} (2) and Remark \ref{rem:conical limit point and bounded parabolic point are limit points}. Let $x \in V(\Gamma^e)\setminus \, \Gamma\G.\leaf(\Gamma)$. Since we have $|\lk_\Gamma(v(x))| \ge 2$ by $x \notin \Gamma\G.\leaf(\Gamma)$ and $\Gamma$ contains no triangle, there exists $g \in \stab_{\Gamma\G}(x)$ that is torsion-free. Since the net $(g^n)_{n \in \NN}$ is wandering and we have $\forall n \in \NN, g^n.x = x$, Lemma \ref{lem:extension graph has absorbing dynamics} implies $\lim_{n \to \infty} g^n.y = x$ for any $y \in \lk_{\Gamma^e}(x)$. Hence, it's not difficult to see $\lim_{n \to \infty} g^n.y = x$ for any $y \in \Delta\Gamma^e$ since every geodesic from $x$ to every other point in $\Delta\Gamma^e$ contains some point in $\lk_{\Gamma^e}(x)$. This implies $\fix_{\Delta \Gamma^e}(g)=\{x\}$, that is, $g$ is parabolic. Since a parabolic fixed point cannot be a conical limit point by \cite[Proposition 3.2]{Bow99a}, $x$ is not a conical limit point.
    

    (5) Let $v \in V(\Gamma) \setminus \leaf(\Gamma)$. Suppose for contradiction that we have $|\lk_{\Gamma}(v)\setminus \leaf(\Gamma)| = \infty$ and there exists a compact set $K \subset \Lambda(\Gamma\G)\setminus \{v\}$ such that $\Lambda(\Gamma\G)\setminus \{v\} = \stab_{\Gamma\G}(v).K$. This implies $\lk_{\Gamma}(v)\setminus \leaf(\Gamma) \subset \Lambda(\Gamma\G)\setminus \{v\} = \stab_{\Gamma\G}(v).K$ by \cite[Corollary 3.7]{Oya24b} and Theorem \ref{thm:action on extension graph is convergence action} (2). Hence, for any $w \in \lk_{\Gamma}(v)\setminus \leaf(\Gamma)$, there exists $g_w \in \stab_{\Gamma\G}(v)$ such that $g_w.w \in K$. By $|\lk_{\Gamma}(v)\setminus \leaf(\Gamma)| = \infty$ and \cite[Corollary 3.7]{Oya24b}, the set $\{g_w.w \in \lk_{\Gamma^e}(v)  \mid w \in \lk_\Gamma(v)\setminus \leaf(\Gamma)\}$ is infinite. Hence, there exists a sequence in $\{g_w. w \mid w \in \lk_{\Gamma}(v) \setminus \leaf(\Gamma) \} (\subset K)$ that converges to $v$. This implies $v \in K$ since $K$ is closed. This contradicts $K \subset \Lambda(\Gamma\G) \setminus \{v\}$. Hence, $v$ is not bounded parabolic when $|\lk_{\Gamma}(v)\setminus \leaf(\Gamma)| = \infty$.

    Conversely, suppose $|\lk_{\Gamma}(v)\setminus \leaf(\Gamma)| < \infty$. By $v \in V(\Gamma) \setminus \leaf(\Gamma)$ and by the proof of Theorem \ref{thm:action on extension graph is convergence action} (5), the group $\stab_{\Gamma\G}(v)$ contains a parabolic element. Hence, $\stab_{\Gamma\G}(v)$ is infinite and contains no loxodromic element by Remark \ref{rem:parabolic fixed point}. For each $w \in \lk_{\Gamma}(v)\setminus \leaf(\Gamma)$, there exists an open set $U_w$ in $\Delta\Gamma^e$ such that $v \in U_w \subset P(v,\{(v,w)\})$ since $P(v,\{(v,w)\})$ is a neighborhood of $v$. The set $\Delta\Gamma^e\setminus U_w$ is compact since $\Delta\Gamma^e$ is compact. Define $K_w$ by $K_w = \Lambda(\Gamma\G) \cap (\Delta\Gamma^e \setminus U_w)$. We claim $\Lambda(\Gamma\G) \setminus \{v\} = \bigcup_{w \in \lk_{\Gamma}(v)\setminus \leaf(\Gamma)} \stab_{\Gamma\G}(v).K_w$. Indeed, given $x \in \Lambda(\Gamma\G) \setminus \{v\}$, fix $p \in \geo_{\Gamma^e}(v,x)$. Let $e \in E(p)$ and $g \in \stab_{\Gamma\G}(v)$ satisfy $e_- = v$ and $e_+ = g.v(e_+)$. Note $v(e_+) \in \lk_{\Gamma}(v)$ by $e_+ \in \lk_{\Gamma^e} (v)$. We can see $v(e_+) \in \lk_{\Gamma}(v) \setminus \leaf(\Gamma)$ by $x \in \Lambda(\Gamma\G) = \Delta\Gamma^e\setminus\, \Gamma\G.\leaf(\Gamma)$ and \cite[Lemma 3.14 (1)]{Oya24b}. On the other hand, we have $g^{-1}. x \in \Delta\Gamma^e \setminus P(v,\{g^{-1}.e\})$ by $g^{-1}.e \in E(g^{-1}.p)$. We also have $g^{-1}.x \in \Lambda(\Gamma\G)$ since $\Lambda(\Gamma\G)$ is $\Gamma\G$-invariant. Hence, $g^{-1}.x \in \Lambda(\Gamma\G) \cap (\Delta\Gamma^e \setminus P(v,\{g^{-1}.e\})) \subset K_{v(e_+)}$. This shows $\Lambda(\Gamma\G) \setminus \{v\} \subset \bigcup_{w \in \lk_{\Gamma}(v)\setminus \leaf(\Gamma)} \stab_{\Gamma\G}(v).K_w$. The converse inclusion is straightforward since $\Lambda(\Gamma\G)$ is $\Gamma\G$-invariant. Since we have $|\lk_{\Gamma}(v)\setminus \leaf(\Gamma)| < \infty$ and $\Lambda(\Gamma\G)$ is closed in $\Delta\Gamma^e$, the set $K$ defined by $K = \bigcup_{\lk_{\Gamma}(v)\setminus \leaf(\Gamma)} K_w$ is compact. The set $K$ also satisfies $\Lambda(\Gamma\G) \setminus \{v\} = \stab_{\Gamma\G}(v). K$. Hence, $v$ is a bounded parabolic point.

    (6) By $\diam_\Gamma(\Gamma) > 1$, there exists $v \in V(\Gamma)$ such that $|\lk_{\Gamma}(v)| \ge 2$. By Theorem \ref{thm:action on extension graph is convergence action} (2) and (4), $v$ is a limit point but not a conical limit point. Hence, $\Gamma\G \act \Delta\Gamma^e$ is not uniform by Remark \ref{rem:uniform convergence action}.

    (7) (i) We first assume $|V(\Gamma) \setminus \leaf(\Gamma)|<\infty$ and show that the action $\Gamma\G \act \Delta\Gamma^e$ is geometrically finite. By Theorem \ref{thm:action on extension graph is convergence action} (5), every point in $V(\Gamma^e) \setminus \Gamma\G.\leaf(\Gamma)$ is bounded parabolic. We will show that every point in $\partial \Gamma^e$ is a conical limit point. Note that $\partial \Gamma^e \neq \emptyset$ implies $V(\Gamma)\setminus \leaf(\Gamma) \neq \emptyset$. Let $x \in \partial \Gamma^e$. Fix $v_0 \in V(\Gamma)\setminus \leaf(\Gamma)$ and $p \in \geo_{\Gamma^e}(v_0,x)$. Let $(x_n)_{n \in \NN\cup\{0\}} \subset V(\Gamma^e)$ satisfy $p=(x_0,x_1,x_2,\cdots)$. Note $x_0=v_0$. For any $n \in \NN$, we have $v(x_n) \in V(\Gamma)\setminus \leaf(\Gamma)$ by $|\lk_{\Gamma^e}(x_n)| \ge 2$ and \cite[Lemma 3.14 (1)]{Oya24b}. For each $n \in \NN$, there exists $g_n \in \Gamma\G$ such that $g_n.(v(x_{n-1}),v(x_n))=(x_{n-1},x_n)$ by \cite[Lemma 3.8]{Oya24b}. Define $(h_n)_{n\in\NN} \subset \Gamma\G$ by $h_1=g_1$ and $h_n=g_{n-1}^{-1}g_n, \forall n \ge 2$. By $v(x_0)=x_0=g_1.v(x_0)$ and $g_{n-1}.v(x_{n-1})=g_n.v(x_{n-1}), \forall n \ge 2$, we have $h_n \in \stab_{\Gamma\G}(v(x_{n-1}))$ for any $n \ge 1$. 
    
    (i-1) When the set $\{h_n \mid n \in \NN\}$ is infinite, by $|V(\Gamma)\setminus \leaf(\Gamma)|<\infty$ there exist $u,v,w \in V(\Gamma)\setminus \leaf(\Gamma)$ and a subsequence $(x_{n_i})_{i \in \NN}$ of $(x_n)_{n \in \NN}$ such that $(v(x_{n_i - 1}), v(x_{n_i}), v(x_{n_i + 1}))=(u,v,w)$ for any $i \ge 1$ and the elements of the sequence $(h_{n_i+1})_{i \in \NN}$ are all distinct. Note $\{(u,v), (v,w)\} \subset E(\Gamma)$. Since the sequence $(h_{n_i+1})_{i \in \NN}$ is wandering and we have $(h_{n_i+1})_{i \in \NN} \subset \stab_{\Gamma\G}(v)$, we have $\lim_{i \to \infty} h_{n_i+1}.w = v$ by Lemma \ref{lem:extension graph has absorbing dynamics}. This implies $\lim_{i \to \infty} g_{n_i}^{-1}.x = v$ since we have $(v,h_{n_i+1}.w) = h_{n_i+1}.(v(x_{n_i}), v(x_{n_i + 1})) = g_{n_i}^{-1}.(x_{n_i}, x_{n_i+1}) \in E(g_{n_i}^{-1}.p_{[x_{n_i}, \infty)})$ for any $i \ge 1$. On the other hand, we have $(g_{n_i}^{-1}.v_0)_{i \in\NN} \subset \Delta\Gamma^e \setminus P(v,\{(u,v)\})$ by $(u,v) = (v(x_{n_i - 1}), v(x_{n_i})) = g_{n_i}^{-1}.(x_{n_i - 1}, x_{n_i}) \in E(g_{n_i}^{-1}.p_{[v_0,x_{n_i}]}), \forall i \ge 1$. Since $P(v,\{(u,v)\})$ is a neighborhood of $v$ and $\Delta\Gamma^e$ is compact, there exist $b \in \Delta\Gamma^e\setminus\{v\}$ and a subnet $(g_{n_j})_{j \in J}$ of $(g_{n_i})_{i \in \NN}$ such that $\lim_{j \to \infty} g_{n_j}^{-1}.v_0 = b$. 
    
    (i-2) When the set $\{h_n \mid n \in \NN\}$ is finite, define $F \subset \Gamma\G$ by $F=\{h_n \mid n \in \NN\}$. We will show that the sequence $\{g_n^{-1}p\}_{n \in \NN}$ of geodesic rays has a subsequence that ``converges to a bi-infinite geodesic path $\gamma$" in the sense of \eqref{eq:converge to gamma}.
    
    Fix $(v',h') \in (V(\Gamma)\setminus \leaf(\Gamma))\times F$. For each $n \in \NN\setminus\{1\}$, define $s_n=(s'_n(k), s''_n(k))_{k \in \ZZ} \in ((V(\Gamma)\setminus \leaf(\Gamma))\times F)^\ZZ$ by
    \begin{align*}
        (s'_n(k), s''_n(k))=
        \begin{cases}
        (v',h')
        & {\rm if~} k \le -(n-1) \\
        (v(x_{k+n-1}),h_{k+n-1})
        & {\rm if~} k > -(n-1).
        \end{cases}
    \end{align*}
    Since the set $(V(\Gamma)\setminus \leaf(\Gamma))\times F$ is finite, there exists $(w(k),f(k))_{k \in \ZZ} \in ((V(\Gamma)\setminus \leaf(\Gamma))\times F)^\ZZ$ and a subsequence $(s_{n_i})_{i \in \NN}$ of $(s_n)_{n \ge 2}$ such that for any $N\in\NN$ and any $i \ge N$ and $k \in\ZZ$ with $|k| \le N$, we have $(s'_{n_i}(k), s''_{n_i}(k))=(w(k),f(k))$. Define the sequence $\gamma = (\gamma(k))_{k \in \ZZ}$ in $V(\Gamma^e)$ by
    \begin{align*}
        \gamma(k)=
        \begin{cases}
        f(2)f(3)\cdots f(k).w(k)
        & {\rm if~} k \ge 2 \\
        w(k)
        & {\rm if~} 0 \le k \le 1 \\
        f(1)^{-1}f(0)^{-1}\cdots f(k+1)^{-1}.w(k)
        & {\rm if~} k \le -1.
        \end{cases}
    \end{align*}
     Let $N \in\NN$ with $N \ge 2$ and let $i > N$. Note $n_i \ge i > N$. If $2 \le k \le N$, then
     \begin{align*}
        g_{n_i}^{-1}.x_{n_i+k-1} 
        &= h_{n_i+1}\cdots h_{n_i+k-1}.v(x_{n_i+k-1}) 
        = s''_{n_i}(2)s''_{n_i}(3)\cdots s''_{n_i}(k).s'_{n_i}(k) \\
        &= f(2)f(3)\cdots f(k).w(k)
        = \gamma(k).
     \end{align*}
     Also, $g_{n_i}^{-1}.(x_{n_i-1},x_{n_i})=(v(x_{n_i-1}),v(x_{n_i}))=(s'_{n_i}(0), s''_{n_i}(1))=(w(0),w(1))=(\gamma(0),\gamma(1))$. If $-N \le k \le -1$, then
     \begin{align*}
         g_{n_i}^{-1}.x_{n_i+k-1} 
         &= h_{n_i}^{-1}h_{n_i-1}^{-1}\cdots h_{n_i+k}^{-1}. v(x_{n_i+k-1}) 
         = s''_{n_i}(1)^{-1} s''_{n_i}(0)^{-1} \cdots s''_{n_i}(k+1)^{-1}. s'_{n_i}(k) \\
         &= f(1)^{-1}f(0)^{-1}\cdots f(k+1)^{-1}.w(k) 
         = \gamma(k).
     \end{align*}
    These imply for any $N \ge 2$ and $i > N$,
    \begin{align}\label{eq:converge to gamma}
    g_{n_i}^{-1}.p_{[x_{n_i-N-1},x_{n_i+N-1}]} = (\gamma(-N),\cdots,\gamma(-1),\gamma(0),\gamma(1),\cdots,\gamma(N)).
    \end{align}
    Hence, $\gamma$ is a bi-infinite geodesic path in $\Gamma^e$. Define $\gamma_-,\gamma_+ \in \partial \Gamma^e$ by $\gamma_- = \lim_{k \to -\infty} \gamma(k)$ and $\gamma_+ = \lim_{k \to \infty} \gamma(k)$. We have $\gamma_- \neq \gamma_+$ since $\gamma$ is bi-infinite geodesic. The above argument also implies $\lim_{i \to \infty}g_{n_i}^{-1}. v_0 = \gamma_-$ and $\lim_{i \to \infty}g_{n_i}^{-1}. x = \gamma_+$.

    By (i-1) and (i-2), there exist $a,b \in \Delta\Gamma^e$ with $a \neq b$ and a subset $(g_{n_j})_{j \in J}$ of $(g_n)_{n \in \NN}$ such that $\lim_{i \to \infty}g_{n_i}^{-1}. x = a$ and $\lim_{i \to \infty}g_{n_i}^{-1}. v_0 = b$. By $|V(\Gamma)| \ge 3 \,(\Leftrightarrow \diam_\Gamma(\Gamma) > 1)$, there exists $c \in V(\Gamma) \setminus \{a,b\}$. By $|V(\Gamma)\setminus \leaf(\Gamma)|<\infty$ and Remark \ref{rem:diam when V(Gamma)setminus I(Gamma) is finite}, we have $d_{\Gamma^e}(g_n.c, x_n)=d_{\Gamma^e}(c, v(x_n)) \le \diam_\Gamma(\Gamma)<\infty$ for any $n \ge 1$. This implies $\lim_{n \to \infty} g_n.c = x$. Define the sequence $(y_n)_{n\in\NN}$ in $V(\Gamma^e)$ by $y_n = g_n.c, \, \forall n \ge 1$, then the net $(y_{n_j})_{j \in J}$ satisfies $\lim_{j \to \infty} y_{n_j} = x$ and $\lim_{j \to \infty} g_{n_j}^{-1}.y_{n_j} = c$. Thus, every point $x \in \partial \Gamma^e$ is a conical limit point by Remark \ref{rem:equivalent definition of conical limit point}. Since every point in $\Delta\Gamma^e\setminus\, \Gamma\G.\leaf(\Gamma)$ is either a conical limit point or a bounded parabolic point and we have $\Lambda(\Gamma\G) = \Delta\Gamma^e\setminus\, \Gamma\G.\leaf(\Gamma)$ by Theorem \ref{thm:action on extension graph is convergence action} (2), the action $\Gamma\G \act \Delta\Gamma^e$ is geometrically finite.

    (ii) Next, we assume $|V(\Gamma) \setminus \leaf(\Gamma)|=\infty$ and show that $\Gamma\G \act \Delta\Gamma^e$ is not geometrically finite. If there exists $v \in V(\Gamma)\setminus \leaf(\Gamma)$ such that $|\lk_\Gamma(v) \setminus \leaf(\Gamma)|=\infty$, then $v$ is a limit point by Theorem \ref{thm:action on extension graph is convergence action} (2), but not a conical limit point nor a bounded parabolic point by Theorem \ref{thm:action on extension graph is convergence action} (4) and (5). Hence, the action $\Gamma\G \act \Delta\Gamma^e$ is not geometrically finite.
    
    In the following, we assume $|\lk_\Gamma(v) \setminus \leaf(\Gamma)|<\infty$ for every $v \in V(\Gamma)\setminus \leaf(\Gamma)$. Since the induced subgraph on $V(\Gamma)\setminus \leaf(\Gamma)$ is a connected infinite locally finite graph, we can see that there exists a geodesic ray $p=(v_0,v_1,\cdots)$ in $\Gamma$, where $\forall n \in\NN, \, v_n \in V(\Gamma)$. Define $x\in \partial \Gamma^e$ by $x = \lim_{n\to\infty}v_n$. We claim $\stab_{\Gamma\G}(x)=\{1\}$. Indeed, let $g \in \Gamma\G$ satisfy $g.x=x$. By $p \subset \Gamma$ and Lemma \ref{lem:geodesic to the same boundary point go to the same plane}, the sequence $(gv_n)_{n \in \NN}$ is also eventually in $\Gamma$. This and $g.p \subset g.\Gamma$ imply $\diam_{\Gamma^e}(\Gamma\cap g.\Gamma)=\infty$. Hence, $g=1$ by \cite[Remark 3.12 (3)]{Oya24b}. By $\stab_{\Gamma\G}(x)=\{1\}$, $x$ is not bounded parabolic. 
    
    Next, we show the claim that if $a,b \in \Delta \Gamma^e$ and a wandering net $(g_i)_{i \in I}$ satisfy $\lim_{i \to \infty}g_i. v_0 = a$ and $\lim_{i \to \infty}g_i.x = b$, then we have $a=b$. We'll discuss three cases (ii-1)-(ii-3), (ii-1) when $\{a,b\}\subset \partial \Gamma^e$, (ii-2) when $a \in V(\Gamma^e)$, (ii-3) when $b \in V(\Gamma^e)$.

    In case (ii-1), suppose $a\neq b$ for contradiction. Since $\Gamma^e$ is fine, we can see that there exist a bi-infinite geodesic path $\gamma=(\cdots,\gamma(-1),\gamma(0),\gamma(1),\cdots)$ in $\Gamma^e$ ($\forall\, k\in \ZZ, \, \gamma(k) \in V(\Gamma^e)$) and a subnet $(g_j)_{j \in J}$ of $(g_i)_{i \in I}$ such that for any $N \in\NN$, there exists $j_N \in J$ such that $\gamma_{[\gamma(-N),\gamma(N)]}$ is a subpath of $g_j.p$ for any $j \ge j_N$. Hence, for any $j,j' \in J$ satisfying $j\ge j_3$ and $j'\ge j_3$, we have $\gamma_{[\gamma(-3),\gamma(3)]} \subset g_j.\Gamma \cap g_{j'}.\Gamma$, hence $\diam_{\Gamma^e}(g_j.\Gamma \cap g_{j'}.\Gamma) \ge 6$. This implies $g_j = g_{j'}$ for any $j,j'\ge j_3$ by \cite[Remark 3.12 (3)]{Oya24b}. This contradicts that the net $(g_j)_{j\in J}$ is wandering.

    In case (ii-2), if there exist $m\in \NN\cup\{0\}$ and a subset $(g_j)_{j\in J}$ of $(g_i)_{i \in I}$ such that $g_j.v_m = a$ for any $j \in J$, then the net $(g_j)_{j\in J}$ satisfies $\lim_{j \to \infty}g_j.x = a$ since $g_j.p_{[v_m,\infty)}$ is a geodesic ray from $a$ to $g_j.x$ for each $j \in J$ and we have $\lim_{j \to \infty}g_j.v_{m+1} = a$ by Lemma \ref{lem:extension graph has absorbing dynamics} (note also that $(g_j)_{j\in J}$ is wandering). This implies $a=\lim_{j \to \infty}g_j.x = \lim_{i \to \infty}g_i.x = b$. 
    
    If there does not exist such $m$ and $(g_j)_{j \in J}$, then there exists $i_0 \in I$ such that $a \notin V(g_i.p)$ for any $i \ge i_0$. Indeed, by \cite[Corollary 3.7]{Oya24b}, we have $|\{n \in \NN\cup\{0\} \mid a \in \Gamma\G. v_n\}| \le 1$ since the vertices $(v_n)_{n \ge 0} \subset \Gamma$ are all distinct, which holds since $p$ is geodesic. If $\{n \in \NN\cup\{0\} \mid a \in \Gamma\G. v_n\}=\emptyset$, then $a \notin V(g_ip)$ for any $i \in I$. If $\{n \in \NN\cup\{0\} \mid a \in \Gamma\G. v_n\}=\{m\}$, then there exist $i_0 \in I$ such that $g_i.v_m\neq a, \forall i \ge i_0$ by our assumption, hence $\forall i \ge i_0,\, a \notin V(g_i.p)$. 
    
    For any finite set $A \subset \elk_{\Gamma^e}(a)$, there exists a finite set $F \subset \elk_{\Gamma^e}(a)$ satisfying the condition (\textreferencemark) of Lemma \ref{lem:net converging to a vertex}. By $\lim_{i \to \infty} g_i.v_0 = a$, there exists $i_1 \in I$ such that $g_i.v_0 \in P(a,F) \subset P'(a,F)$ for any $i \ge i_1$. Take $i_2 \in I$ satisfying $i_2\ge i_0$ and $i_2 \ge i_1$. For any $i \ge i_2$, we have $g_i.x \in P(a,A)$ by $a \notin V(g_i.p)$ and the condition (\textreferencemark). This implies $\lim_{i \to \infty} g_i.x = a$, hence $a=b$.

    In case (ii-3), if there exist $m\in \NN\cup\{0\}$ and a subnet $(g_j)_{j\in J}$ of $(g_i)_{i \in I}$ such that $g_j.v_m = b$ for any $j \in J$, then the net $(g_j)_{j\in J}$ satisfies $\lim_{j \to \infty}g_j.v_0 = b$ by Lemma \ref{lem:extension graph has absorbing dynamics} since $(g_j)_{j\in J}$ is wandering. This implies $b=\lim_{j \to \infty}g_j.v_0 = \lim_{i \to \infty}g_i.v_0 = a$. If there does not exist such $m$ and $(g_j)_{j \in J}$, then as in the case (ii-2), there exists $i_0 \in I$ such that $b \notin V(g_i.p)$ for any $i \ge i_0$. Hence, we can show $\lim_{i \to \infty}g_i.v_0 = b$ by the same argument as (ii-2) using $\lim_{i \to \infty}g_i.x = b$ and Lemma \ref{lem:net converging to a vertex}. This implies $a=b$.

    By (ii-1), (ii-2), and (ii-3), the claim follows. Suppose for contradiction that $x$ is a conical limit point. There exist a net $(g_i)_{i \in I}$ in $\Gamma\G$ and $a,b \in \Delta\Gamma^e$ with $a \neq b$ such that $g_i.x \to b$ and $g_i.y \to a$ for any $y \in \Delta\Gamma^e \setminus \{x\}$. We have $g_i.v_0 \to a$ by $v_0 \neq x$. Also, the net $(g_i)_{i \in I}$ is wandering by $|\Delta\Gamma^e| \ge 3 \,(\Leftrightarrow \diam_\Gamma(\Gamma) > 1)$. Indeed, suppose for contradiction that there exist $g \in \Gamma\G$ and a subnet $(g_j)_{j \in J}$ of $(g_i)_{i \in I}$ such that $g_j=g$ for any $j \in J$. Take distinct points $y_1,y_2  \in \Delta\Gamma^e \setminus \{x\}$ by $|\Delta\Gamma^e| \ge 3$. For each $k\in\{1,2\}$, we have $g.y_k = \lim_{j \to \infty} g_j. y_k = a $ by $x \neq y_k$. This implies $y_1=y_2 \,(=g^{-1}.a)$, which contradicts $y_1\neq y_2$. Since the net $(g_i)_{i \in I}$ is wandering, the claim above implies $a = \lim_{i \to \infty} g_i.v_0 = \lim_{i \to \infty} g_i.x = b$, which contradicts $a \neq b$. Hence, $x$ is neither a conical limit point nor a bounded parabolic point. We also have $x \in \partial\Gamma^e \subset \Lambda(\Gamma\G)$ by Theorem \ref{thm:action on extension graph is convergence action} (2). Thus, the action $\Gamma\G \act \Delta\Gamma^e$ is not geometrically finite.
\end{proof}

In Theorem \ref{thm:action on extension graph is convergence action}, the condition that $\{G_v\}_{v\in V(\Gamma)}$ is a collection of finite groups is essential by Proposition \ref{prop: infinite case doesn't admit non-elementary convergence action} below. Proposition \ref{prop: infinite case doesn't admit non-elementary convergence action} is a variant of \cite[Proposition 7.9]{Sun19}, though it's different in that elements of infinite order don't need to generate the whole group. Although the condition $\girth(\Gamma)>20$ in Theorem \ref{thm:action on extension graph is convergence action} may not be optimal, it is also essential to some extent. Corollary \ref{cor: square case doesn't admit non-elementary convergence action} verifies this point.

\begin{prop}\label{prop: infinite case doesn't admit non-elementary convergence action}
    Suppose that $\Gamma$ is a connected simplicial graph and $\{G_v\}_{v \in V(\Gamma)}$ is a collection of non-trivial groups. If for any $v \in V(\Gamma)$, there exists an element $g_v \in G_v$ of infinite order, then the group $\Gamma\G$ doesn't admit non-elementary convergence action.
\end{prop}

\begin{proof}
    Suppose that $\Gamma\G$ acts on a compact Hausdorff space $M$ with $|M| \ge 3$ as a convergence action. Let $v,w \in V(\Gamma)$ with $(v,w) \in E(\Gamma)$. Since the element $g_v \in G_v$ has infinite order, $g_v$ is either parabolic or loxodromic. Hence, $1\le |\fix_M(g_v)| \le 2$. We will show $\fix_M(g_v)=\fix_M(g_w)$ and $g\fix_M(g_v)=\fix_M(g_v)$ for any $g \in \la G_v , G_w \ra \,(=G_v \times G_w)$. For any $g \in G_w$ and $x \in \fix_M(g_v)$, we have $g_vg x=gg_vx=gx$, hence $gx \in \fix_M(g_v)$. This implies $g\fix_M(g_v)=\fix_M(g_v)$ for any $g \in G_w$. Since the element $g_w \in G_w$ has infinite order, $g_w$ is either parabolic or loxodromic. When $|\fix_M(g_v)|=2$, the element $g_w^2$ fixes $\fix_M(g_v)$ pointwise by $g_w\fix_M(g_v)=\fix_M(g_v)$. Hence, $g_w$ is loxodromic. This implies $\fix_M(g_w)=\fix_M(g_w^2)=\fix_M(g_v)$. When $|\fix_M(g_v)|=1$, the element $g_v$ is parabolic. By $g_w\fix_M(g_v)=\fix_M(g_v)$, we have $\fix_M(g_v) \subset \fix_M(g_w)$. Hence, $g_w$ is also parabolic since a loxodromic element and a parabolic element cannot share their fixed points. This implies $\fix_M(g_v) = \fix_M(g_w)$. Thus, we have $\fix_M(g_w) = \fix_M(g_v)$ in either case. We can show $g\fix_M(g_w)=\fix_M(g_w)$ for any $g \in G_v$ in the same way as we did for $\fix_M(g_v)$ and $G_w$. By this and $\fix_M(g_w) = \fix_M(g_v)$, we have $g\fix_M(g_v) = \fix_M(g_v)$ for any $g \in \la G_v ,G_w \ra $.
    
    Fix $v_0 \in V(\Gamma)$. Since $\Gamma$ is connected, for any $v \in V(\Gamma)$, there exists a sequence $v_0,v_1,\cdots, v_n$ in $V(\Gamma)$ such that $v_n=v$ and $(v_{i-1},v_i) \in E(\Gamma)$ for any $i \in \{1,\cdots,n\}$. For any $i \ge 1$, by applying the above argument to $(v_{i-1}, v_i) \in E(\Gamma)$, we have $\fix_M(g_{v_{i-1}})=\fix_M(g_{v_i})$ and $g\fix_M(g_{v_{i-1}})=\fix_M(g_{v_{i-1}})$ for any $g \in \la G_{v_{i-1}} , G_{v_i} \ra$. Hence, $g\fix_M(g_{v_0})=\fix_M(g_{v_0})$ for any $g \in G_v\,(=G_{v_n})$. Since $v \in V(\Gamma)$ is arbitrary, this implies $g\fix_M(g_{v_0})=\fix_M(g_{v_0})$ for any $g \in \Gamma\G$. By $1\le |\fix_M(g_{v_0})| \le 2$, the action $\Gamma\G \act M$ is elementary.
\end{proof}

\begin{cor}\label{cor: square case doesn't admit non-elementary convergence action}
    Suppose that $\Gamma$ is a circuit of length $4$ and $\{G_v\}_{v \in V(\Gamma)}$ is a collection of non-trivial groups, then $\Gamma\G$ doesn't admit non-elementary convergence action.
\end{cor}

\begin{proof}
    This follows from Proposition \ref{prop: infinite case doesn't admit non-elementary convergence action} since we have $\Gamma\G=(G_1 \ast G_3) \times (G_2 \ast G_4)$, where $\{G_v\}_{v \in V(\Gamma)}=\{G_1,G_2,G_3,G_4\}$, and for any $g_1 \in G_1\setminus\{1\}$, $g_2 \in G_2\setminus\{1\}$, $g_3 \in G_3\setminus\{1\}$, and $g_4 \in G_4\setminus\{1\}$, the elements $g_1g_3 \in G_1 \ast G_3$ and $g_2g_4 \in G_2 \ast G_4$ are torsion free.
\end{proof}

    The girth condition $\girth(\Gamma)>20$ in Theorem \ref{thm:action on extension graph is convergence action} and Theorem \ref{thm:graph product is relatively bi-exact} ultimately comes from the classification results for geodesic bigons and triangles in the extension graph $\Gamma^e$ in \cite[Section 3.4]{Oya24b} (e.g. \cite[Corollary 3.36, Proposition 3.38]{Oya24b}), which played a crucial role to prove geometric properties of $\Gamma^e$ in \cite[Theorem 1.2]{Oya24b}. However, the lower bound 20 has room for improvement depending on specific results. For example, the proof of hyperbolicity of $\Gamma^e$ in \cite[Theorem 1.2 (1)]{Oya24b} doesn't use the condition that a subpath of a geodesic in a copy of $\Gamma$ has length $\ge 7$ (e.g. \cite[Proposition 3.35 (2)]{Oya24b}), so it could be shown if other conditions in \cite[Corollary 3.36, Proposition 3.38]{Oya24b} are proved with a weaker bound.

\section{Relative bi-exactness}
\label{sec:Relative bi-exactness}

The goal of this section is to prove Theorem \ref{thm:intro graph product becomes relatively bi-exact}, which corresponds to Theorem \ref{thm:graph product is relatively bi-exact}. We start by studying right coset representatives of a subgroup of graph product generated by a subset of vertex groups, which is from Definition \ref{def:R_F} up to Lemma \ref{lem:properties of p_F, r_F}.

\begin{defn}\label{def:R_F}
    Let $\Gamma$ be a simplicial graph and $\G=\{G_v\}_{v \in V(\Gamma)}$ be a collection of non-trivial groups. For a subset $F \subset V(\Gamma)$, we define $R_F \subset \Gamma\G$ to be the union of $\{1\}$ and the set of all elements $g \in \Gamma\G\setminus\{1\}$ such that for any normal form $g=s_1\cdots s_n$ of $g$, we have $\supp(s_1) \notin F$ (see Convention \ref{conv:normal form}).
\end{defn}

\begin{rem}
    For example, set $V(\Gamma)=\{a,b\}$ and $F = \{a\}$, then in case $E(\Gamma)=\emptyset$, we have $R_F=\{1\}\cup\{g \in G_a\ast G_b\setminus\{1\} \mid \text{the normal form of $g$ starts with a syllable in $G_b$} \}$ and in case $E(\Gamma)=\{(a,b),(b,a)\}$, we have $R_F=G_b$.
\end{rem}

It turns out that the set $R_F$ in Definition \ref{def:R_F} is a set of right coset representatives of the subgroup $F\G$ of $\Gamma\G$ generated by the set $\bigcup_{v \in F} G_v$, which we show in Lemma \ref{lem:properties of R_F} below.

\begin{lem}\label{lem:properties of R_F}
    Let $\Gamma$ be a simplicial graph and $\G=\{G_v\}_{v \in V(\Gamma)}$ be a collection of non-trivial groups. For any $F \subset V(\Gamma)$, the following hold, where we set $F\G = \la G_v \mid v \in F \ra$.
    \begin{itemize}
        \item[(1)]
        The set $R_F$ is a complete representative of the right cosets of the subgroup $F\G$ in $\Gamma\G$, that is, for any $g \in \Gamma\G$, there exist unique $a \in F\G$ and $b \in R_F$ such that $g=ab$.
        \item[(2)]
        Let $g \in R_F$, $v \in V(\Gamma)$, and $h \in G_v \setminus \{1\}$. If one of the conditions (i)-(iii) below holds, then we have $gh \in R_F$. (i) $\|gh\| < \|g\| + \|h\|$, (ii) $\supp(g) \not\subset \lk_\Gamma(v)$, (iii) $v \notin F$.
    \end{itemize}
\end{lem}

\begin{proof}
    (1) Let $g \in \Gamma\G \setminus \{1\}$. For a normal form $g=s_1\cdots s_n$ of $g$, define $k_{(s_1,\cdots,s_n)} \in \NN\cup\{0\}$ by $k_{(s_1,\cdots,s_n)}=0$ if $\supp(s_1) \notin F$ and $k_{(s_1,\cdots,s_n)}=k \in \NN$ if $\supp(s_1\cdots s_k) \subset F$ and $\supp(s_{k+1})\notin F$. Take a normal form $g=s_1\cdots s_n$ of $g$ such that $k_{(s_1,\cdots,s_n)}$ is the maximum among all normal forms of $g$. Let $k=k_{(s_1,\cdots,s_n)}$ for brevity, then we have $s_{k+1}\cdots s_n \in R_F$ by maximality of $k_{(s_1,\cdots,s_n)}$. Hence, $\Gamma\G=F\G \cdot R_F$.

    Let $a,a' \in F\G$ and $b,b' \in R_F$ satisfy $ab=a'b'$. Suppose $a\neq a'$ for contradiction. Let $a^{\prime-1}a=s_1\cdots s_n$, $b=t_1\cdots t_m$, $b=t'_1\cdots t'_{m'}$ be normal forms of $a^{\prime-1}a$, $b$, and $b'$ respectively. Note that $s_1$ exists by $a^{\prime-1}a\neq 1$. If the word $s_1\cdots s_n t_1\cdots t_m$ is not geodesic, then by Theorem \ref{thm:normal form theorem}, there exist syllables $s_i$ and $t_j$ such that $\supp(s_i)=\supp(t_j)$ and $\{\supp(s_{i'}) \mid i<i'\} \cup \{\supp(t_{j'}) \mid j'<j\} \subset \lk_\Gamma(\supp(t_j))$. Hence, $b=t_j t_1\cdots t_{j-1}t_{j+1}\cdots t_m$. By $\{a,a'\}\subset F\G$, we have $\supp(t_j) = \supp(s_i) \in F$, which contradicts $b \in R_F$. If the word $s_1\cdots s_n t_1\cdots t_m$ is geodesic, then by $(a^{\prime-1}a)b=b'$ and Theorem \ref{thm:normal form theorem}, the word $s_1\cdots s_n t_1\cdots t_m$ is obtained from the word $t'_1\cdots t'_{m'}$ by a finite step of syllable shuffling. Hence, there exists a syllable $t'_j$ such that $t'_j=s_1$ and $\{\supp(t'_{j'}) \mid j'<j\} \subset \lk_\Gamma(\supp(t'_j))$. Hence, $b'=t'_j t'_1\cdots t'_{j-1}t'_{j+1}\cdots t'_{m'}$. By $\supp(t'_j) = \supp(s_1) \in F$, this contradicts $b' \in R_F$. Thus, $a=a'$. This and $ab=a'b'$ imply $b=b'$.

    (2) Suppose $gh \notin R_F$ for contradiction. Let $g=s_1\cdots s_n$ be a normal form of $g$. 
    
    (i) When $\|gh\| <\|g\| + \|h\|$, there exists a syllable $s_i$ such that $\supp(s_i)=v$ and $\{\supp(s_{i'}) \mid i<i'\}\subset \lk_\Gamma(\supp(s_i))$. Note $\supp(s_i) \in \lk_\Gamma(\supp(s_{i'}))$ for any $i'>i$. If $s_ih \neq 1$, then $gh=s_1\cdots s_{i-1}(s_ih)s_{i+1} \cdots s_n$ is a normal form of $gh$, and if $s_ih = 1$, then $gh=s_1\cdots s_{i-1}s_{i+1} \cdots s_n$ is a normal form of $gh$. In either case, by $gh \notin R_F$, we can see that there exists $j \in \{1,\cdots,n\}$ such that $\supp(s_j) \in F$ and $\{\supp(s_{i'}) \mid i'<j\}\subset \lk_\Gamma(\supp(s_j))$, which contradicts $g \in R_F$. 
    
    (ii)-(iii) When $\supp(g) \not\subset \lk_\Gamma(v)$ or $v \notin F$. Since we've discussed case (i), we further assume $\|gh\| = \|g\| + \|h\|$. Hence, $gh=s_1\cdots s_nh$ is a normal form of $gh$. Set $s_{n+1}=h$. By $gh \notin R_F$, there exists a syllable $s_j$ such that $\supp(s_j) \in F$ and $\{\supp(s_{i'}) \mid i'<j\}\subset \lk_\Gamma(\supp(s_j))$. By $\supp(g) \not\subset \lk_\Gamma(v) \, \vee \,v \notin F$, we have $j\neq n+1$, hence $1 \le j \le n$. This contradicts $g \in R_F$.
\end{proof}

Below we introduce notations to decompose an element of $\Gamma\G$ into that of the subgroup $F\G$ and the set $R_F$ of right coset representatives, which is possible by Lemma \ref{lem:properties of R_F} (1).

\begin{defn}\label{def:p_F, r_F}
    Let $\Gamma$ be a simplicial graph and $\G=\{G_v\}_{v \in V(\Gamma)}$ be a collection of non-trivial groups. Let $F$ be a subset of $V(\Gamma)$. For $g \in \Gamma\G$, we define $p_F(g) \in F\G$ and $r_F(g) \in R_F$ to be the unique elements satisfying $g=p_F(g) r_F(g)$.
\end{defn}

In Lemma \ref{lem:properties of p_F, r_F} below, we see how $p_F(g)$ and $r_F(g)$ change when we multiply group elements from left and right. This plays an important role in the proof of Theorem \ref{thm:graph product is relatively bi-exact}.

\begin{lem}\label{lem:properties of p_F, r_F}
    Let $\Gamma$ be a simplicial graph and $\G=\{G_v\}_{v \in V(\Gamma)}$ be a collection of non-trivial groups. For any $F\subset V(\Gamma)$, the following hold.
    \begin{itemize}
        \item[(1)]
        For any $g \in\Gamma\G$ and $h \in F\G$, $p_F(hg)=h  p_F(g)$.
        \item[(2)]
        For any $g \in \Gamma\G$, $v \in V(\Gamma)$, and $h \in G_v \setminus \{1\}$, the following (i)-(iii) are all equivalent: (i) $p_F(g) \neq p_F(gh)$, (ii) $v \in F$ and $\supp(r_F(g)) \subset \lk_\Gamma(v)$, (iii) $p_F(gh) = p_F(g)h$.
    \end{itemize}    
\end{lem}

\begin{proof}
    (1) Note $hg=h p_F(g) r_F(g)$ and $hp_F(g) \in F\G$. Hence, $p_F(hg)=h p_F(g)$ by Lemma \ref{lem:properties of R_F} (1).

    (2) $\underline{(i)\Rightarrow(ii)}$ If $r_F(g)h \in R_F$, then we have $p_F(gh) = p_F(g)$ by $gh=p_F(g)r_F(g)h$ and Lemma \ref{lem:properties of R_F} (1), which contradicts the assumption $p_F(gh) \neq p_F(g)$. Hence, $r_F(g)h \not\in R_F$. This and Lemma \ref{lem:properties of R_F} (2) imply $\|r_F(g)h\| = \|r_F(g)\| + \|h\|$, $\supp(r_F(g)) \subset \lk_\Gamma(v)$, and $v \in F$.

    $\underline{(ii)\Rightarrow(iii)}$ By $v \in F$ and $\supp(r_F(g)) \subset \lk_\Gamma(v)$, we have $gh=p_F(g)r_F(g)h=p_F(g)hr_F(g)$ and $p_F(g)h \in F\G$. This implies $p_F(gh) = p_F(g)h$ by Lemma \ref{lem:properties of R_F} (1).

    $\underline{(iii)\Rightarrow(i)}$ Trivial by $h \neq 1$.
\end{proof}

Before working on the proof of Theorem \ref{thm:graph product is relatively bi-exact}, we need to prove two more lemmas, Lemma \ref{lem:tight fine graph admits a sequence for amenable action} and Lemma \ref{lem:finite quotient of biexact-group is biexact}. The proof of of Lemma \ref{lem:tight fine graph admits a sequence for amenable action} below is essentially the same as the proofs of \cite[Theorem 1.33]{Kai04} and \cite[Lemma 8]{Oza06}, but because Claim \ref{claim: cardinality of S(x,z,ell,m)} is a new ingredient in the proof and some constants need to be adjusted, we give full details of the proof to make it self-contained. Note that Lemma \ref{lem:tight fine graph admits a sequence for amenable action} doesn't follow from \cite[Lemma 8]{Oza06}, because $X$ doesn't need to be uniformly fine. In fact, the extension graph $\Gamma^e$ is not uniformly fine when $\sup_{v \in V(\Gamma)} |G_v| = \infty$ and $\girth(\Gamma) < \infty$ in Theorem \ref{thm:graph product is relatively bi-exact}.

In Lemma \ref{lem:tight fine graph admits a sequence for amenable action}, the set $\partial X$ is a $G_\delta$ subset of $\Delta X$, hence Borel. In its proof below, given $a,b\in V(X)$ and $r\in \NN$, we define $V(a,b), V(a,b\,;r) \subset V(X)$ by (see \cite[Section 5.2]{Oya24b})
\begin{align}\label{eq:V(a,b,r)}
\begin{array}{rcl}
V(a,b) &=& \bigcup\{V(p) \mid p \in \geo_X(a,b) \}, \\
V(a,b\,;r) &=& \bigcup\{ V(a',b') \mid a'\in \N_X(a,r),\, b'\in \N_X(b,r) \}.    
\end{array}
\end{align}

\begin{lem}\label{lem:tight fine graph admits a sequence for amenable action}
    Let $X$ be a fine hyperbolic countable graph satisfying the condition of \cite[Propositiion 5.5 (2)]{Oya24b}. Then, there exists a sequence $(\eta_n)_{n=1}^\infty$ of Borel maps from $V(X) \times \partial X $ to $\Prob(V(X))$ such that for any $d \in \NN$,
    \begin{align}\label{eq:eta_n(x,z)}
        \lim_{n \to \infty}\sup_{z \in \partial X} \sup_{x,x' \in V(X), d_{X}(x,x') \le d} \|\eta_n(x,z) -\eta_n(x',z)\|_1 = 0.
    \end{align} 
\end{lem}

\begin{proof}
   Let $X$ be $\delta$-hyperbolic with $\delta\in \NN$. For $x \in V(X)$, $z \in \partial X$, and $\ell,m \in \NN$, we define $S(x,z,\ell,m) \subset V(X)$ by
\begin{align*}
    S(x,z,\ell,m)
    =
    \{\alpha(\ell) \mid x' \in \N_X(x,m), \alpha\in \geo_X(x',z) \}.
\end{align*}
Let $P_1, k_1 \in \NN$ be the constants as in \cite[Propositiion 5.5 (2)]{Oya24b} for $k=2\delta$.
\begin{claim}\label{claim: cardinality of S(x,z,ell,m)}
    For any $x \in V(X)$, $z \in \partial X$, and $\ell,m \in \NN$ with $\ell \ge 2m+k_1+4\delta$, we have $|S(x,z,\ell,m)| \le  P_1\cdot (2m+4\delta+1)$.
\end{claim}
\begin{proof}[Proof of claim \ref{claim: cardinality of S(x,z,ell,m)}]
    Fix $\alpha \in \geo_X(x,z)$ and take $b \in V(\alpha)$ satisfying $d_{X}(x,b)=3\ell$. Let $y \in S(x,z,\ell,m)$. There exist $x' \in \N_{X}(x,m)$ and $\beta \in \geo_X(x',z)$ such that $y=\beta(\ell)$. By $d_{X}(y,x')=\ell > m \ge d_{X}(x,x')$, there exists $w \in V(\alpha)$ such that $d_{X}(y,w) \le 2\delta$. Note $|d_{X}(x,w) - \ell| =|d_{X}(x,w) - d_{X}(y,x')| \le d_{X}(x,x') + d_{X}(y,w) \le m+2\delta$. Similarly, by $d_{X}(x,b)=3\ell$ and $\ell \ge 2m+k_1+4\delta$, there exists $b' \in V(\beta)$ such that $d_{X}(b,b') \le 2\delta$ and $y\in \beta_{[x', b']}$. Hence, $y \in V(x,b \, ;m+2\delta)\cap \N_{X}(w,2\delta)$ (see \eqref{eq:V(a,b,r)}). By $d_{X}(w,\{x,b\}) \ge k_1+m+2\delta$ and \cite[Propositiion 5.5 (2)]{Oya24b}, we have $|V(x,b \, ;m+2\delta)\cap \N_{X}(w,2\delta)| \le P_1$. Also, there are at most $2(m+2\delta)+1$ choices for $w$ by $|d_{X}(x,w) - \ell| \le m+2\delta$. Thus, $|S(x,z,\ell,m)| \le  P_1(2m+4\delta+1)$. This finishes the proof of Claim \ref{claim: cardinality of S(x,z,ell,m)}.
\end{proof}
For a finite set $S \subset V(X)$, we denote the characteristic function $V(X) \to \{0,1\}$ of $S$ by $1_S$ and define $\chi_S \in \Prob(V(X))$ by $\chi_S = \frac{1}{|S|}\cdot1_S$. For $x \in V(X)$, $z \in \partial X$, and $n\in \NN$ with $n \ge k_1+4\delta$, define $\eta_n(x,z) \in \Prob(V(X))$ by
\begin{align*}
    \eta_n(x,z) = \frac{1}{n} \sum_{i=n+1}^{2n} \chi_{S(x,z,6n,i)}.
\end{align*}
We show that $(\eta_n)_{n \ge k_1+4\delta}$ satisfies \eqref{eq:eta_n(x,z)}. Let $d \in\NN$, $z\in \partial X$, and $x,x' \in V(X)$ satisfy $d_{X}(x,x') \le d$. Let $n \ge k_1+4\delta+2d$ and denote $S_i=S(x,z,6n,i)$ and $S'_i=S(x',z,6n,i)$ for brevity. We have $S_i \cup S'_i \subset S_{i+d}$ and $S_{i-d} \subset S_i \cap S'_i$ for any $i \in \{n+1,\cdots,2n\}$. This implies for any $i \in \{n+1,\cdots,2n\}$,
\begin{align*}
    \|\chi_{S_i} - \chi_{S'_i} \|_1
    =
    2\Big(1 - \frac{|S_i\cap S'_i|}{\max\{|S_i|, |S'_i|\}} \Big)
    \le
    2\Big( 1 - \frac{|S_{i-d}|}{|S_{i+d}|} \Big).
\end{align*}
For any $i \in \{n-d,\cdots,2n+d\}$, we have $|S_i| \le P_1(2i+4\delta+1)$ by $6n \ge 2(2n+d)+k_1+4\delta$ and Claim \ref{claim: cardinality of S(x,z,ell,m)}. Hence,
\begin{align*}
    \|\eta_n(x,z) -\eta_n(x',z)\|_1
    &\le
    \frac{1}{n} \sum_{i=n+1}^{2n} \|\chi_{S_i} - \chi_{S'_i} \|_1
    \le
    2\Big( 1 - \frac{1}{n} \sum_{i=n+1}^{2n} \frac{|S_{i-d}|}{|S_{i+d}|} \Big) \\
    &\le
    2\bigg( 1 - \Big(\prod_{i=n+1}^{2n} \frac{|S_{i-d}|}{|S_{i+d}|}\Big)^{\frac{1}{n} } \bigg) 
    =
    2\bigg( 1 - \Big( \frac{\prod_{i=n+1-d}^{n+d} |S_i|}{ \prod_{i=2n+1-d}^{2n+d} |S_i|}\Big)^{\frac{1}{n} } \bigg) \\
    &\le
    2\bigg( 1 - \Big(P_1 \cdot \big( 2(2n+d)+4\delta+1 \big) \Big)^{\frac{-2d}{n}} \bigg).
\end{align*}
By $\lim_{n \to \infty} \Big(P_1 \cdot \big( 2(2n+d)+4\delta+1 \big) \Big)^{\frac{-2d}{n}} = 1$, the equation \eqref{eq:eta_n(x,z)} follows.

For $x,y \in V(X)$, define $T(x,y) \subset \partial X$ by $T(x,y)=\{z \in \partial X \mid \exists \alpha \in \geo_X(x,z), y \in V(\alpha) \}$. Since $X$ is fine, $T(x,y)$ is closed in $\partial X$. For any $x,y \in V(X)$ and $\ell,m \in \NN$, we have
\begin{align*}
    \{z\in \partial X \mid y \in S(x,z,\ell,m) \}
    =
    \bigcup\{ T(x',y) \mid \text{ $x' \in V(X)$ s.t. $d_X(x,x') \le m \wedge d_X(x',y) = \ell$ } \},
\end{align*}
hence the set $\{z\in \partial X \mid y \in S(x,z,\ell,m) \}$ is Borel in $\partial X$. This implies that the function $\chi_{S(x,z,6n,i)} \colon V(X) \times \partial X \to \Prob(V(X))$ is Borel for fixed $n$ and $i$ (see Proposition \ref{prop:Ozawa Prop 11} (1) for the definition of a map to $\Prob(V(X))$ being Borel). Hence, the maps $(\eta_n)_{n \ge k_1+4\delta}$ are Borel.
\end{proof}

Lemma \ref{lem:finite quotient of biexact-group is biexact} below is used to show that vertex stabilizers of the extension graph of graph product of finite groups are bi-exact in Theorem \ref{thm:graph product is relatively bi-exact}. Lemma \ref{lem:finite quotient of biexact-group is biexact} is a direct consequence of \cite[Theorem 1.2]{Sak09} when a group $G$ is countable, but the proof below doesn't assume countability of groups.

\begin{lem}\label{lem:finite quotient of biexact-group is biexact}
    Suppose that $G$ is a group and $H$ is a normal subgroup of $G$. If $H$ is finite and the quotient group $G/H$ is bi-exact, then $G$ is bi-exact.
\end{lem}

\begin{proof}
    Let $q\colon G \to G/H$ be the quotient map. Set $K=G/H$ for brevity. In the following, for any $K$-action $\alpha \colon K \act X$ on a set $X$, we will consider the $G$-action $\alpha \circ q \colon G\act X$. 
    
    Define the $G$-equivariant unital $\ast$-homomorphism $\Phi \colon \ell^\infty(K) \to \ell^\infty(G)$ by $\Phi(f)=f \circ q$. By $|H| <\infty$, we can see $\Phi(c_0(K)) \subset c_0(G)$. This implies $\Phi(A(K;\{\la 1 \ra\})) \subset A(G;\{\la 1 \ra\})$. Hence, there exists a $G$-equivariant continuous map $\widetilde{q} \colon \overline{G}^{\{\la 1 \ra\}} \to \overline{K}^{\{\la 1 \ra\}}$ such that $\widetilde{q}|_{G} = q$. Let $F \subset G$ be a finite set and let $\e \in \RR_{>0}$. Since $K$ is bi-exact, the action $K\act \overline{K}^{\{\la 1 \ra\}}$ is topologically amenable by Proposition \ref{prop:equivalent condition of biexact group}. Hence, there exists a continuous map $m \colon \overline{K}^{\{\la 1 \ra\}} \to \Prob(K)$ such that $\max_{g \in q(F)}\sup_{x\in \overline{K}^{\{\la 1 \ra\}}}\|m(gx) - g.m(x)\|_1< \e$. Define $\widetilde{m} \colon \overline{K}^{\{\la 1 \ra\}} \to \Prob(G)$ by $\widetilde{m}(x) = \frac{1}{|H|} m(x) \circ q$. The map $\widetilde{m}\circ \widetilde{q} \colon \overline{G}^{\{\la 1 \ra\}} \to \Prob(G)$ is continuous and satisfies $\max_{g \in F}\sup_{x\in \overline{G}^{\{\la 1 \ra\}}}\|\widetilde{m}\circ \widetilde{q}(gx) - g.(\widetilde{m}\circ \widetilde{q})(x)\|_1< \e$ by $\forall g\in G, \forall x \in \overline{G}^{\{\la 1 \ra\}}, \|\widetilde{m}\circ \widetilde{q}(gx) - g.(\widetilde{m}\circ \widetilde{q})(x)\|_1=\|m(q(g)\widetilde{q}(x)) - q(g).m(\widetilde{q}(x))\|_1$. Hence, $G$ is bi-exact.
\end{proof}

We are finally ready to prove Theorem \ref{thm:graph product is relatively bi-exact}. See Definition \ref{def: boundary small at infinity} for relevant concepts in the proof.

\begin{thm}\label{thm:graph product is relatively bi-exact}
    Suppose that $\Gamma$ is a uniformly fine hyperbolic countable graph with $\girth(\Gamma) > 20$ and that $\G=\{G_v\}_{v \in V(\Gamma)}$ is a collection of non-trivial finite groups. Then, $\Gamma\G$ is bi-exact relative to the collection of subgroups $\{\, \la G_w \mid w\in \lk_{\Gamma}(v) \ra \,\}_{v \in V(\Gamma)}$.
\end{thm}

\begin{proof}
By \cite[Theorem 1.2 (1)]{Oya24b} and \cite[Theorem 1.2 (3)]{Oya24b}, $\Gamma^e$ is fine and $\delta$-hyperbolic with $\delta\in \NN$. By \cite[Propositiion 5.5 (2)]{Oya24b} and Lemma \ref{lem:tight fine graph admits a sequence for amenable action}, there exists a sequence $(\eta_n)_{n=1}^\infty$ of Borel maps from $V(\Gamma^e) \times \partial \Gamma^e$ to $\Prob(V(\Gamma^e))$ that satisfies the condition \eqref{eq:eta_n(x,z)}. Fix $o\in V(\Gamma^e)$ and define for each $n \in \NN$, the map $\zeta_n \colon \Delta \Gamma^e \to \Prob(V(\Gamma^e))$ by 
\begin{align*}
    \zeta_n(z)=
    \begin{cases}
        \eta_n(o,z) &{\rm ~if~} z\in \partial \Gamma^e \\
        1_{\{z\}} &{\rm ~if~} z \in V(\Gamma^e).
    \end{cases}
\end{align*}
The map $\zeta_n$ is Borel since $\eta_n$ is Borel. By \eqref{eq:eta_n(x,z)}, we have
\begin{align*}
    \forall\, g \in \Gamma\G,\, \lim_{n \to\infty} \sup_{z \, \in \Delta \Gamma^e}\|\zeta_n(g.z) - g.\zeta_n(z)\|_1 =0.
\end{align*}
Indeed, if $z\in \partial \Gamma^e$, then $\|\zeta_n(g.z) - g.\zeta_n(z)\|_1=\|\zeta_n(o,g.z) - \eta_n(g.o,g.z)\|_1$, and if $z\in V(\Gamma^e)$, then $\|\zeta_n(g.z) - g.\zeta_n(z)\|_1 = \|1_{\{g.z\}} - 1_{\{g.z\}}\|_1=0$.

For brevity, we define $\F$ by $\F=\{\, \la G_w \mid w\in \lk_{\Gamma}(v) \ra \,\}_{v \in V(\Gamma)}$. Next, we will show that the action $\stab_{\Gamma\G}(v) \act \overline{\Gamma\G}^{\F}$ is topologially amenable for any $v \in V(\Gamma)$, where $\stab_{\Gamma\G}(v)$ is the stabilizer for the action $\Gamma\G \act \Gamma^e$. Define the map $\Phi \colon \ell^\infty(\stab_{\Gamma\G}(v)) \to \ell^\infty(\Gamma\G)$ by 
\begin{align*}
\Phi(f)(g)=f(p_{\st_\Gamma(v)}(g))    
\end{align*}
for $f \in \ell^\infty(\stab_{\Gamma\G}(v))$ and $g \in \Gamma\G$, where $p_{\st_\Gamma(v)}(g)$ is as in Definition \ref{def:p_F, r_F} for the subset $\st_\Gamma(v) \subset V(\Gamma)$. Here, note $\stab_{\Gamma\G}(v)=\la G_w \mid w \in \st_\Gamma(v) \ra$ by \cite[Corollary 3.6]{Oya24b}. The map $\Phi$ is a unital $\ast$-homomorphism and $\stab_{\Gamma\G}(v)$-equivariant by Lemma \ref{lem:properties of p_F, r_F} (1). Let $f \in A(\stab_{\Gamma\G}(v);\{\la1\ra\})$, $w \in V(\Gamma)$, $a\in G_w\setminus\{1\}$, and $\e \in \RR_{>0}$. For any $g\in \Gamma\G$ satisfying $|\Phi(f)(g)-\Phi(f)(ga^{-1})| \ge \e$, we have $p_{\st_\Gamma(v)}(g)\neq p_{\st_\Gamma(v)}(ga^{-1})$. This implies $w \in \st_\Gamma(v)$, $\supp(r_{\st_\Gamma(v)}(g)) \subset \lk_\Gamma(w)$, and $p_{\st_\Gamma(v)}(ga^{-1}) = p_{\st_\Gamma(v)}(g)a^{-1}$ by Lemma \ref{lem:properties of p_F, r_F} (2). Hence, $|f(p_{\st_\Gamma(v)}(g)) - f(p_{\st_\Gamma(v)}(g)a^{-1})|=|\Phi(f)(g)-\Phi(f)(ga^{-1})| \ge \e$. Define $E \subset \stab_{\Gamma\G}(v)$ by
\begin{align*}
    E=\{x \in \stab_{\Gamma\G}(v) \mid |f(x) - f(xa^{-1})| \ge \e \},
\end{align*}
then we have $g=p_{\st_\Gamma(v)}(g)r_{\st_\Gamma(v)}(g) \in E\cdot\la G_z \mid z\in \lk_{\Gamma}(w) \ra$ by $\supp(r_{\st_\Gamma(v)}(g)) \subset \lk_\Gamma(w)$. The set $E$ is finite by $f \in A(\stab_{\Gamma\G}(v),\{\la1\ra\})$. Hence, the set
\begin{align*}
    \{g \in \Gamma\G \mid |\Phi(f)(g)-\Phi(f)(ga^{-1})| \ge \e\}
\end{align*}
is small relative to $\F$. This implies $\Phi(f)-\Phi(f)^a \in c_0(\Gamma\G; \F)$ for any $f \in A(\stab_{\Gamma\G}(v);\{\la1\ra\})$, $w \in V(\Gamma)$, and $a\in G_w\setminus\{1\}$. Since $c_0(\Gamma\G; \F)$ is invariant under the right $\Gamma\G$-action and the set $\bigcup_{w \in V(\Gamma)}G_{w}$ generates $\Gamma\G$, we have $\Phi(f)-\Phi(f)^t \in c_0(\Gamma\G; \F)$ for any $f \in A(\stab_{\Gamma\G}(v);\{\la1\ra\})$ and $t \in \Gamma\G$. Hence, the restriction $\Phi|_{A(\stab_{\Gamma\G}(v);\{\la1\ra\})}$ is a $\stab_{\Gamma\G}(v)$-equivariant unital $\ast$-homomorphism from $A(\stab_{\Gamma\G}(v);\{\la1\ra\})$ to $A(\Gamma\G;\F)$. Hence, there exists a $\stab_{\Gamma\G}(v)$-equivariant continuous map $\widetilde{\Phi} \colon \overline{\Gamma\G}^{\F} \to \overline{\stab_{\Gamma\G}(v)}^{\{\la1\ra\}}$. Recall $\stab_{\Gamma\G}(v)=G_v \times (\ast_{w \in \lk_\Gamma(v)} G_w)$ by $\girth(\Gamma) > 3$ and \cite[Corollary 3.6]{Oya24b}. Since finite groups are bi-exact, $\stab_{\Gamma\G}(v)$ is bi-exact by \cite[Lemma 5.1]{Oya23b} and Lemma \ref{lem:finite quotient of biexact-group is biexact}. By Proposition \ref{prop:equivalent condition of biexact group}, the action $\stab_{\Gamma\G}(v) \act \overline{\stab_{\Gamma\G}(v)}^{\{\la1\ra\}}$ is topologically amenable. By the map $\widetilde{\Phi}$ and Lemma \ref{lem:G-equivariant map passes topological amenability}, this implies that the action $\stab_{\Gamma\G}(v) \act \overline{\Gamma\G}^{\F}$ is topologically amenable.

Hence, for any $v \in V(\Gamma)$ and $g \in \Gamma\G$, the action $g\stab_{\Gamma\G}(v)g^{-1} \act \overline{\Gamma\G}^{\F}$ is topologically amenable. Indeed, for any finite set $F \subset g\stab_{\Gamma\G}(v)g^{-1}$ and $\e \in \RR_{>0}$, we can take a continuous map $m \colon \overline{\Gamma\G}^{\F} \to \mathrm{Prob}(\stab_{\Gamma\G}(v))$ satisfying $\max_{s \in g^{-1}Fg}\sup_{x\in \overline{\Gamma\G}^{\F}}\|m(sx) - s.m(x)\|_1< \e$. Define $m_1 \colon \overline{\Gamma\G}^{\F} \to \mathrm{Prob}(g\stab_{\Gamma\G}(v)g^{-1})$ by $m_1(x)=g.m(g^{-1}x)^{g^{-1}}$, then we have for any $s \in F$ and $x \in \overline{\Gamma\G}^{\F}$,
\begin{align*}
    \|m_1(sx) - s.m_1(x)\|_1 
    &= \|g.m(g^{-1}sx)^{g^{-1}} - sg.m(g^{-1}x)^{g^{-1}}\|_1 \\
    &= \|m(g^{-1}sgg^{-1}x) - g^{-1}sg.m(g^{-1}x)\|_1 
    < \e.
\end{align*}
By applying Proposition \ref{prop:Ozawa Prop 11} to $G=\Gamma\G$, $X = \Delta \Gamma^e$, $Y = \overline{\Gamma\G}^{\F}$, and $K=V(\Gamma^e)$, the diagonal action $\Gamma\G \act  \Delta \Gamma^e \times \overline{\Gamma\G}^{\F}$ is topologically amenable.

Finally, we will show that there exists a $\Gamma\G$-equivariant continuous map from $\overline{\Gamma\G}^{\F}$ to $\Delta \Gamma^e$. Define the map $\sigma \colon C(\Delta \Gamma^e) \to \ell^\infty(\Gamma\G)$ by $\sigma(f)(g) = f(g.o)$, where $f \in C(\Delta \Gamma^e)$ and $g \in \Gamma\G$ (recall that $C(\Delta \Gamma^e)$ is the set of all continuous maps from $\Delta \Gamma^e$ to $\CC$). Suppose for contradiction that there exist $f \in C(\Delta \Gamma^e)$, $t \in \Gamma\G$, and $\e \in \RR_{>0}$ such that the set $\Omega \subset \Gamma\G$ defined by
\[
\Omega = \{g \in \Gamma\G \mid |f(g.o)-f(gt^{-1}.o)| \ge \e\}
\]
is infinite, then we can take a wandering net $(g_i)_{i\in I}$ in $\Omega$. Since $\Delta \Gamma^e$ is compact, there exist $b \in \Delta \Gamma^e$ and a subnet $(g_j)_{j \in J}$ of $(g_i)_{i \in I}$ such that $\lim_{j\to\infty}g_j.o=b$. Since the net $(g_j)_{j \in J}$ is wandering, we have $\lim_{j\to\infty}g_jt^{-1}.o=b$ by Lemma \ref{lem:extension graph has absorbing dynamics}. This implies $\lim_{j\to\infty} |f(g_j.o)-f(g_jt^{-1}.o)| = |f(b)-f(b)| = 0$, which contradicts $(g_j)_{j \in J} \subset \Omega$. Hence, $\sigma(f)-\sigma(f)^t \in c_0(\Gamma\G;\{\la 1 \ra\}) \subset c_0(\Gamma\G;\F)$ for any $f \in C(\Delta \Gamma^e)$ and $t \in \Gamma\G$. Since this implies that the map $\sigma$ is a $\Gamma\G$-equivariant unital $\ast$-homomorphism from $C(\Delta \Gamma^e)$ to $A(\Gamma\G;\F)$, there exists a $\Gamma\G$-equivariant continuous map $\widetilde{\sigma} \colon \overline{\Gamma\G}^{\F} \to \Delta \Gamma^e$ such that $\sigma(f) = f \circ \widetilde{\sigma}$ for any $f \in C(\Delta \Gamma^e)$. The map $\Psi \colon  \overline{\Gamma\G}^{\F} \to \Delta \Gamma^e \times \overline{\Gamma\G}^{\F}$ defined by $\Psi(x)=(\widetilde{\sigma}(x), x)$ is continuous and $\Gamma\G$-equivariant. 

Hence, the action $\Gamma\G \act \overline{\Gamma\G}^{\F}$ is topologically amenable by Lemma \ref{lem:G-equivariant map passes topological amenability}. By Proposition \ref{prop:equivalent condition of biexact group}, this implies that $\Gamma\G$ is bi-exact relative to $\F$.
\end{proof}

Although primeness of most graph products of groups was already proved in \cite[Corollary 7.1]{CRW18}, we record a new proof using relative bi-exactness in Corollary \ref{cor:graph product becomes prime} below. For this, we first prove Proposition \ref{prop:relatively biexact groups are prime}. Proposition \ref{prop:relatively biexact groups are prime} is well-known to experts, but has not been written down. I would like to thank Changying Ding for teaching me its proof.

\begin{defn}
    Let $G$ be a group. A subgroup $H$ of $G$ is called \emph{proper} if $H \lneqq G$ and \emph{almost malnormal inside} $G$ if for any $g \in G \setminus H$, $H \cap gHg^{-1}$ is finite.
\end{defn}

\begin{prop}\label{prop:relatively biexact groups are prime}
    Let $G$ be a non-amenable countable group and $\G$ be a countable collection of subgroups of $G$ such that every $H \in \G$ is proper and almost malnormal inside $G$. If $G$ is bi-exact relative to $\G$, then the group von Neumann algebra $L(G)$ is prime.
\end{prop}

\begin{proof}
    Suppose for contradiction that there exist infinite dimensional von Neumann subalgebras $P,Q \subset L(G)$ such that $L(G) = P \tensor Q$. Since $L(G)$ is tracial, so are $P$ and $Q$. By this and infinite dimensionality, $P$ and $Q$ are diffuse. Since $L(G)$ is non-amenable, one of $P$ or $Q$ must be non-amenable. Assume that $P$ is non-amenable without loss of generality. By \cite[Theorem 15.1.5]{BO08}, there exists $H \in \G$ such that $Q$ embeds in $L(H)$ inside $L(G)$ (see \cite[Definition F.13]{BO08} and \cite[Defintion A.3.1, Theorem A.3.3]{Bou14}). Here, we used that a tracial von Neumann algebra is amenable if and only if it is injective. Since $H$ is almost malnormal inside $G$, the inclusion $L(H) \subset L(G)$ is mixing relative to $\CC \cdot 1$ by \cite[Example A.1.5]{Bou14} (see \cite[Definition A.1.1, Proposition A.1.3, Definiton A.4.2, Definiton A.4.3]{Bou14}). Since $Q$ is diffuse, $Q$ doesn't embed in $\CC \cdot 1$ inside $L(G)$. For the projection $1 \in L(G)$ and the subalgebra $Q \subset 1 L(G) 1 \,(=L(G))$, we have $\mathcal{Q}\mathcal{N}_{L(G)}(Q)''=L(G)$ by $L(G)=P \tensor Q$ (see \cite[p.21]{Bou14}). Hence by \cite[Proposition A.4.6 (1)]{Bou14}, there exists a non-zero partial isometry $v \in L(G)$ such that $v^* L(G)v \subset L(H)$. This implies that $L(G)$ embeds in $L(H)$ inside $L(G)$ by \cite[Theorem F.12 (4)]{BO08}. Hence, $[G:H] < \infty$ i.e. $H$ has finite index in $G$ by \cite[Lemma 2.2]{CI18}. Hence, there exists a normal subgroup $K$ of $G$ such that $K\subset H$ and $[G:K] < \infty$. Since $H$ is proper in $G$, we can take $g \in G\setminus H$ and have $K \subset H\cap gHg^{-1}$. This implies $|K| <\infty$ since $H$ is almost malnormal inside $G$. By this and $[G:K] < \infty$, $G$ is finite. This contradicts that $G$ is non-amenable.
    \end{proof}

\begin{cor}\label{cor:graph product becomes prime}
    Suppose that $\Gamma$ is a uniformly fine hyperbolic countable graph with $\girth(\Gamma) > 20$ and $\diam_\Gamma(\Gamma) > 2$ and that $\G=\{G_v\}_{v \in V(\Gamma)}$ is a collection of non-trivial finite groups. Then, the group von Neumann algebra $L(G)$ is prime.
\end{cor}

\begin{proof}
    By \cite[Corollary 3.6]{Oya24b}, we have $\la G_w \mid w\in \lk_{\Gamma}(v) \ra \subset \la G_w \mid w\in \st_{\Gamma}(v) \ra = \stab_{\Gamma\G}(v)$ for any $v \in V(\Gamma)$. Hence, the group $\Gamma\G$ is bi-exact relative to the collection of subgroups $\{\stab_{\Gamma\G}(v)\}_{v \in V(\Gamma)}$ by Theorem \ref{thm:graph product is relatively bi-exact}. By $\diam_\Gamma(\Gamma) > 2$, $\Gamma\G$ is non-amenable (since it contains $F_2$) and $\stab_{\Gamma\G}(v)$ is proper in $\Gamma\G$ for any $v \in V(\Gamma)$. Also, by $\girth(\Gamma) > 4$ and \cite[Corollary 3.16 (2)]{Oya24b}, $\stab_{\Gamma\G}(v)$ is almost malnormal inside $\Gamma\G$ for any $v \in V(\Gamma)$. Hence, $L(G)$ is prime by Proposition \ref{prop:relatively biexact groups are prime}.
\end{proof}

\section{Strong solidity and bi-exactness}
\label{sec:Strong solidity and bi-exactness}

The goal of this section is to prove Theorem \ref{thm:graph product becomes bi-exact and strongly solid} and Corollary \ref{cor:graph product strongly solid}, which corresponds to Theorem \ref{thm:biexactness of graph wreath product} and Theorem \ref{thm:intro graph product becomes bi-exact and strongly solid} respectively. We start by preparing two auxiliary lemmas, Lemma \ref{lem:Cayley graph is hyperbolic} and Lemma \ref{lem:Cayley graph becomes fine}. Theorem \ref{thm:bigon implies hyperbolic} below follows from the proof of \cite[Theorem 1.4]{Pap95} (see also \cite[Theorem 22]{CN07}). The difference from the thin bigon condition (i.e. the case $d_X(b,c) = 1$) arises because we restrict the endpoints of geodesics to vertices.

\begin{thm}\label{thm:bigon implies hyperbolic}
    Let $X$ be a connected graph. Suppose that there exists $\delta \in \RR_{>0}$ such that for any $a,b,c \in V(X)$ with $d_X(b,c) \le 1$ and $d_X(a,b) = d_X(a,c)$, and any $p \in \geo_X(a,b)$ and $q \in \geo_X(a,c)$, we have $p \subset \N_X(q,\delta)$ and $q \subset \N_X(p,\delta)$. Then, $X$ is hyperbolic.
\end{thm}

\begin{lem}\label{lem:Cayley graph is hyperbolic}
    Suppose that $\Gamma$ is a graph with $\girth(\Gamma)>4$ and that $\G=\{G_v\}_{v \in \Gamma}$ is a collection of groups. Then, the Cayley graph $X$ of $\Gamma\G$ with respect to $\bigcup_{v \in V(\Gamma)}(G_v \setminus \{1\})$ is hyperbolic.
\end{lem}

\begin{proof}
    We first show the following claim ($\ast$) by induction on $d_X(a,b)$.
    \begin{align*}
        &\text{($\ast$) $~~$ For any $a,b \in \Gamma\G$ and any $p,q \in \geo_X(a,b)$,}\\
        &\text{$~~~~~~~~V(p)\setminus\{a,b\} \subset \N_X(V(q)\setminus\{a,b\}, 2)$ and $V(q)\setminus\{a,b\} \subset \N_X(V(p)\setminus\{a,b\}, 2)$}.
    \end{align*}
    If $d_X(a,b) \le 2$, then the claim $(\ast)$ follows trivially. Let $n \in \NN$. We assume that the claim $(\ast)$ holds for any $a',b' \in \Gamma\G$ and $p',q' \in \geo_X(a',b')$ with $d_X(a',b') < n$. Let $a,b \in \Gamma\G$ and $p,q \in \geo_X(a,b)$ satisfy $d_X(a,b)=n$. Since $\Gamma\G$ acts on $X$ transitively, we assume $a=1$ without loss of generality. Note $\|b\| = d_X(1,b) = n$. Let $s_1\cdots s_n$ and $t_1\cdots t_n$ be labels of the paths $p$ and $q$ respectively with $s_i,t_i \in \bigcup_{v \in V(\Gamma)}(G_v \setminus \{1\})$. By Theorem \ref{thm:normal form theorem}, there exists a permutation $\sigma$ of $\{1,\cdots,n\}$ such that $t_i = s_{\sigma(i)}$ for any $i \in \{1,\cdots,n\}$. Note that if $i<j$ and $\sigma(i) > \sigma(j)$, then $(\supp(s_i),\supp(s_j))\in E(\Gamma)$.
    
    (I) When $\sigma(1) = 1$, we have $t_1 = s_1$. This implies $p_{[s_1s_2,b]}\setminus \{b\} \subset \N_X(V(q_{[t_1t_2,b]})\setminus\{b\}, 2)$ and $V(q_{[t_1t_2,b]})\setminus\{b\} \subset \N_X(V(p_{[s_1s_2,b]})\setminus\{b\}, 2)$ by applying the assumption of induction to $p_{[s_1,b]}, q_{[t_1,b]} \in \geo_X(s_1,b)$.

    (II) When $\sigma(1) \neq 1$, we show that either $\sigma(1) = 2$ or $\sigma^{-1}(1) = 2$ holds. Suppose $\sigma(1) > 2$ and $\sigma^{-1}(1) > 2$ for contradiction. This implies $\sigma(1) < \sigma(2)$. Indeed, if $\sigma(2) < \sigma(1)$, then the vertices $\{\supp(s_1),\supp(s_2), \supp(s_{\sigma^{-1}(1)})\}$ form a triangle in $\Gamma$ by $1 = \sigma(\sigma^{-1}(1)) < \sigma(2)<\sigma(1)$, which contradicts $\girth(\Gamma)>4$. By the same argument for $\sigma^{-1}$, we also have $\sigma^{-1}(1) < \sigma^{-1}(2)$. Note $\supp(s_1) \neq \supp(s_2)$ and $\supp(s_{\sigma^{-1}(1)}) = \supp(t_1) \neq \supp(t_2) = \supp(s_{\sigma^{-1}(1)})$ since $s_1\cdots s_n$ and $t_1\cdots t_n$ are geodesic words. Hence, the vertices $\{\supp(s_1), \supp(s_{\sigma^{-1}(1)}), \supp(s_2), \supp(s_{\sigma^{-1}(2)})\}$ form a square in $\Gamma$ by $1<2<\sigma^{-1}(1) < \sigma^{-1}(2)$ and $\sigma(\sigma^{-1}(1))<\sigma(\sigma^{-1}(2))<\sigma(1)<\sigma(2)$. This contradicts $\girth(\Gamma) > 4$.

    Thus, $\sigma(1) = 2$ or $\sigma^{-1}(1) = 2$ holds. Assume $\sigma^{-1}(1) = 2$ without loss of generality. Note $(\supp(s_1),\supp(s_2)) \in E(\Gamma)$ by $\sigma(2) = 1 < \sigma(1)$. Hence, the vertices $t_1 \,(= s_2)$ and $s_1s_2 \,(= s_2s_1)$ in $X$ are connected by the edge $e$ whose label is $s_1$. By applying the assumption of induction to $ep_{[s_1s_2,b]}, q_{[t_1,b]} \in \geo_X(t_1,b)$, we have $p_{[s_1s_2,b]}\setminus \{b\} \subset \N_X(V(q_{[t_1t_2,b]})\setminus\{b\}, 2)$ and $V(q_{[t_1t_2,b]})\setminus\{b\} \subset \N_X(V(p_{[s_1s_2,b]})\setminus\{b\}, 2)$. Also, $d_X(s_1,t_1) \le 2$.

    By (I) and (II), the claim ($\ast$) holds for $a,b,p,q$. Hence, the claim ($\ast$) holds for any $d_X(a,b) (= n)$ by induction. By Theorem \ref{thm:bigon implies hyperbolic}, it remains to see that for any $a,b,c \in V(X)$ with $d_X(b,c) = 1$ and $d_X(a,b) = d_X(a,c)$, and any $p \in \geo_X(a,b)$ and $q \in \geo_X(a,c)$, we have $p \subset \N_X(q,4)$ and $q \subset \N_X(p,4)$. This follows from the claim ($\ast$) and the triangle condition of quasi-median graphs (see \cite[Definition 2.1]{Gen17}) since $X$ is a quasi-median graph by \cite[Proposition 8.2]{Gen17}.
\end{proof}

Let $G$ be a group and $S$ be a generating set of $G$. For a path $\gamma$ in the Cayley graph of $G$ with respect to $S$, we denote the label of $\gamma$ by $\L(\gamma)$. Note that $\L(\gamma)$ is a word in $S\cup S^{-1}$. See Definition \ref{def:concepts in graph theory} for $\C_X(e,n)$.

\begin{lem}\label{lem:Cayley graph becomes fine}
    Suppose that $\Gamma$ is a locally finite graph and that $\G=\{G_v\}_{v \in \Gamma}$ is a collection of finite groups. Then, the Cayley graph $X$ of $\Gamma\G$ with respect to $\bigcup_{v \in V(\Gamma)}(G_v \setminus \{1\})$ is fine.
\end{lem}

\begin{proof}
    Define $L(e,n), M(e,n)\subset E(X)$ by
    \begin{align*}
        L(e,n)&=\bigcup\{ E(\gamma) \mid \gamma \in \C_X(e,n) \}, \\
        M(e,n)&= \bigcup\{E(\gamma) \mid \text{$\gamma \in \C_X(e,n)$ s.t. $\forall\, e' \in E(\gamma)$, $d_\Gamma(\supp(\L(e')),\supp(\L(e))) \le 2$}\}.
    \end{align*}
    We have $|M(e,n)| \le n \cdot \big(\sum_{v \,\in\, \N_\Gamma(\supp(\L(e)),2)}|G_v|\big)^n < \infty$. In the following, we will show that for any $n\in\NN$,
    \begin{align}\tag{$\ast$}\label{eq:fine graph}
        \forall\, e \in E(X),\, |\C_X(e,n)| < \infty
    \end{align}
    holds by induction on $n$. The statement \eqref{eq:fine graph} is true for $n = 3$. Indeed, if the label of $e \in E(X)$ is in $G_v$ with $v \in V(\Gamma)$, then we have $|\C_X(e,3)| \le |G_v|^2$ since $X$ is a simple graph and the edges of a triangle in $X$ are labeled by the same vertex group by \cite[Lemma 8.5]{Gen17}. 
    
    Given $N > 3$, assume that the statement \eqref{eq:fine graph} is true for any $n$ with $n < N$. Let $e\in E(X)$ and $p \in \C_X(e,N)$. Let $p'$ be the subpath of $p$ from $e_-$ to $e_+$ that complements $e$ to form $p$ i.e. $p = ep^{\prime -1}$. Let $\L(p') = s_1\cdots s_m$ with $s_i \in \bigcup_{v \in V(\Gamma)}(G_v \setminus \{1\})$. Since $p'$ is not geodesic in $X$ by $|p|\ge3$, there exists $i,j$ with $1\le i<j\le n$ such that $\supp(s_i) = \supp(s_j)$ and $\{\supp(s_k) \mid i<k<j \}\subset \lk_\Gamma(\supp(s_i))$ by Theorem \ref{thm:normal form theorem}. By taking a pair $(i,j)$ such that $j-i$ is the minimum among all such pairs, we may assume that the subword $s_{i+1}\cdots s_{j-1}$ is geodesic without loss of generality. Let $\alpha$ be the subpath of $p'$ whose label is $s_i\cdots s_j$. Let $\beta$ be the path in $X$ from $\alpha_-$ to $\alpha_+$ whose label is $s_{i+1}\cdots s_{j-1}(s_i s_j)$. Note $|\beta| = |\alpha|-1$ if $s_i s_j \neq 1$ and $|\beta| = |\alpha|-2$ if $s_i s_j = 1$. Since $s_{i+1}\cdots s_{j-1}$ is a geodesic word and we have $\{\supp(s_k) \mid i<k<j \}\subset \lk_\Gamma(\supp(s_i))$, we can see that the path $\beta$ is geodesic in $X$ and the loop $\alpha\beta^{-1}$ is a circuit. Note $\alpha_-\neq \alpha_+$ since the circuit $p$ has no self-intersection.
    
    Let $q$ be the subpath of $p$ from $\alpha_-$ to $\alpha_+$ that complements $\alpha$ to form $p$ i.e. $p=\alpha q^{-1}$. Note $|\alpha| + |\beta^{-1}| \le |\alpha| + |q| \le N$ since $\beta$ is geodesic in $X$. We also have $e \in E(q)$ since $\alpha$ is a subpath of $p'$.
    
    Since $q$ has no self-intersection, there exists a subsequence $(\alpha_- \! =)\,x_0, \cdots, x_\ell \,(= \! \alpha_+)$ of $V(q)$ with $\forall i \ge 1, x_{i-1} \neq x_i$ such that $\{x_0,\cdots,x_\ell\} \subset \beta$ and for every $i \in \{1,\cdots,\ell\}$, letting $\beta_i$ be the subpath of $\beta$ or $\beta^{-1}$ from $x_{i-1}$ to $x_i$, either (1) or (2) holds, (1) $q_{[x_{i-1},x_i]}=\beta_i$, (2) the loop $q_{[x_{i-1},x_i]}\beta_i^{-1}$ is a circuit. In case (2), we have $|q_{[x_{i-1},x_i]}\beta_i^{-1}| \le |q|+|\alpha|-1 = N-1$ by $|\beta| \le |\alpha|-1$. Note $\ell\le N$ by $|q| \le N$ since $q$ has no self-intersection. Let $e \in E(q_{[x_{i_0-1},x_{i_0}]})$ for some $i_0\in\{1,\cdots,\ell\}$.
    
    When case (1) holds for $q_{[x_{i_0-1},x_{i_0}]}$, we have $E(\alpha\beta^{-1}) \subset M(e,N)$ by $|\alpha| + |\beta^{-1}| \le N$ and $\{\supp(s_k) \mid i<k<j \}\subset \lk_\Gamma(\supp(s_i))$. For each $i \in \{1,\cdots,\ell\}$, in case (1), we have $E(q_{[x_{i-1},x_i]}) \subset M(e,N)$, and in case (2), we have $q_{[x_{i-1},x_i]}\beta_i^{-1} \in \bigcup_{e_1 \in M(e,N)} \C_X(e_1,N-1)$. Thus,
    \[
    E(p) \subset M(e,N)\cup \bigcup\nolimits_{e_1 \in M(e,N)}L(e_1,N-1).
    \]
    
    When case (2) holds for $q_{[x_{i_0-1},x_{i_0}]}$, we have $q_{[x_{i_0-1},x_{i_0}]}\beta_{i_0}^{-1} \in \C_X(e,N-1)$. This implies $E(\alpha\beta^{-1}) \subset \bigcup_{e_1 \in L(e,N-1)}M(e_1,N)$. Hence, for each $i \in \{1,\cdots,\ell\}$, we have
    \begin{align*}
        &\text{$E(q_{[x_{i-1},x_i]}) \subset \bigcup\nolimits_{e_1 \in L(e,N-1)}M(e_1,N)$ in case (1) and} \\
        &\text{$q_{[x_{i-1},x_i]}\beta_i^{-1} \in \bigcup \big\{ \C_X(e_2,N-1) \,\big|\, e_2 \in \bigcup\nolimits_{e_1 \in L(e,N-1)}M(e_1,N)\big\}$ in case (2).}
    \end{align*}
    Hence, $E(p) \subset \bigcup_{e_1 \in L(e,N-1)}M(e_1,N) \cup \bigcup \{L(e_2,N-1) \mid e_2 \in \bigcup_{e_1 \in L(e,N-1)}M(e_1,N)\}$.

    Thus, in either case, we have $E(p) \subset M(e,N)\cup \bigcup_{e_1 \in M(e,N)} L(e_1,N-1) \cup \bigcup_{e_1 \in L(e,N-1)}M(e_1,N) \cup \bigcup \{L(e_2,N-1) \mid e_2 \in \bigcup_{e_1 \in L(e,N-1)}M(e_1,N)\}$. This implies $|L(e,N)|<\infty$ since for any $e' \in E(X)$, we have $|M(e',N)|<\infty$ and $|\C_X(e',N-1)|<\infty$, where the latter follows from the assumption of induction. Thus, the statement \eqref{eq:fine graph} is true for $n=N$.
\end{proof}

Lemma \ref{lem:Cayley graph becomes uniformly fine} below also holds by the same proof as Lemma \ref{lem:Cayley graph becomes fine}.

\begin{lem}\label{lem:Cayley graph becomes uniformly fine}
    Suppose that $\Gamma$ is a uniformly locally finite graph and that $\G=\{G_v\}_{v \in \Gamma}$ is a collection of finite groups with $\sup_{v \in V(\Gamma)}|G_v| < \infty$. Then, the Cayley graph $X$ of $\Gamma\G$ with respect to $\bigcup_{v \in V(\Gamma)}(G_v \setminus \{1\})$ is uniformly fine.
\end{lem}

Now, we are ready to prove Theorem \ref{thm:graph product becomes bi-exact and strongly solid}. In Theorem \ref{thm:graph product becomes bi-exact and strongly solid}, $\Gamma$ does not need to be connected.

\begin{defn}\label{def:graph wreath product}
     Let a group $G$ act on a simplicial graph $\Gamma$. Let $\G=\{G_v\}_{v \in \Gamma}$ be a collection of groups such that $G_v = G_{gv}$ for any $g \in G$ and $v \in V(\Gamma)$. For each $g \in G$, the identity map $G_v \ni a \mapsto a \in G_{gv}$ defined on each $v \in V(\Gamma)$ extends to the group automorphism $\alpha_g \colon \Gamma\G \to \Gamma\G$. This defines a group homomorphism $\alpha \colon G \ni g \mapsto \alpha_g \in \mathrm{Aut}(\Gamma\G)$, hence the semi-direct product $\Gamma\G \rtimes G$.
\end{defn}

\begin{rem}\label{rem:graph wreath product}
    When $G_v = G_w$ for any $v,w \in V(\Gamma)$, the semi-direct product $\Gamma\G \rtimes G$ in Definition \ref{def:graph wreath product} is called the graph-wreath product.
\end{rem}

\begin{thm}\label{thm:graph product becomes bi-exact and strongly solid}
    Suppose that $\Gamma$ is a uniformly locally finite countable graph with $\girth(\Gamma)>4$ and that a countable group $G$ acts on $\Gamma$ satisfying $|\stab_G(v)|<\infty$ for any $v \in V(\Gamma)$. Let $\G=\{G_v\}_{v \in \Gamma}$ be a collection of finite groups such that $\sup_{v \in V(\Gamma)}|G_v| < \infty$ and $G_v = G_{gv}$ for any $g \in G$ and $v \in V(\Gamma)$. Then, the following holds.
    \begin{itemize}
        \item[(1)]
        If $G$ is amenable, then $\Gamma\G \rtimes G$ is bi-exact.
        \item[(2)]
        If $G$ is bi-exact, then $\Gamma\G \rtimes G$ is bi-exact relative to $\{\Gamma\G\}$.
    \end{itemize}
\end{thm}

\begin{proof}
    Let $X$ be the Cayley graph of $\Gamma\G$ with respect to $\bigcup_{v \in V(\Gamma)}(G_v \setminus \{1\})$. By Lemma \ref{lem:Cayley graph is hyperbolic} and Lemma \ref{lem:Cayley graph becomes uniformly fine}, $X$ is hyperbolic and uniformly fine. The group $G$ acts on $X$ as graph automorphism by $G \times \Gamma\G \ni (g,x) \mapsto gxg^{-1} \in \Gamma\G$ i.e. conjugation in $\Gamma\G \rtimes G$, which may not preserve labels of $X$. This and the action $\Gamma\G \act X$ by left multiplication extend to the action $\Gamma\G \rtimes G \act X$. Since we have $\stab_{\Gamma\G \rtimes G}(1) = G$ and the action $\Gamma\G \rtimes G \act V(X)$ is transitive, every vertex stabilizer of the action $\Gamma\G \rtimes G \act X$ is conjugate to $G$. Also, for any $v\in V(\Gamma)$ and $g \in G_v\setminus\{1\}$, $\stab_{\Gamma\G \rtimes G}(1) \cap \stab_{\Gamma\G \rtimes G}(g) = \stab_G(v)$. Note $|\stab_G(v)|<\infty$. Hence, every edge stabilizer of the action $\Gamma\G \rtimes G \act X$ is finite.
    
    Since $X$ is hyperbolic and uniformly fine, by \cite[Lemma 8]{Oza06} there exists a sequence $(\eta_n)_{n=1}^\infty$ of Borel maps from $V(X) \times \partial X$ to $\Prob(V(X))$ such that for any $d \in \NN$,
    \begin{align}\label{eq:eta_n(x,z) for bi-exactness}
        \lim_{n \to \infty}\sup_{z \in \partial X} \sup_{x,x' \in V(X), d_{X}(x,x') \le d} \|\eta_n(x,z) -\eta_n(x',z)\|_1 = 0.
    \end{align}
    In the same way as in the proof of Theorem \ref{thm:graph product is relatively bi-exact}, we can see from \eqref{eq:eta_n(x,z) for bi-exactness} that there exists a sequence $(\zeta_n)_{n=1}^\infty$ of Borel maps $\zeta_n \colon \Delta X \to \Prob(V(X))$ such that for any $g \in \Gamma\G \rtimes G$, $\lim_{n \to\infty} \sup_{z \, \in \Delta X}\|\zeta_n(gz) - g.\zeta_n(z)\|_1 =0$.

     Fix $o \in V(X)$ and define the map $\sigma \colon C(\Delta X) \to \ell^\infty(\Gamma\G\rtimes G)$ by $\sigma(f)(g) = f(go)$, where $f \in C(\Delta X)$ and $g \in \Gamma\G\rtimes G$. In the same way as in the proof of Theorem \ref{thm:graph product is relatively bi-exact}, we can see $\sigma(C(\Delta X)) \subset A(\Gamma\G\rtimes G; \{\la 1 \ra\})$ by Lemma \ref{lem:extension graph has absorbing dynamics}.

    (1) Let $G$ be amenable, then for any $x \in V(X)$, the action $\stab_{\Gamma\G \rtimes G}(x) \act \overline{\Gamma\G \rtimes G}^{\{\la 1 \ra\}}$ is topologically amenable since $\stab_{\Gamma\G \rtimes G}(x)$ is conjugate to $G$. Hence, the diagonal action $\Gamma\G\rtimes G \act \Delta X \times \overline{\Gamma\G \rtimes G}^{\{\la 1 \ra\}}$ is topologically amenable by Theorem \ref{prop:Ozawa Prop 11}. By $\sigma(C(\Delta X)) \subset A(\Gamma\G\rtimes G, \{\la 1 \ra\})$, $\sigma$ induces a $\Gamma\G \rtimes G$-equivariant continuous map $\widetilde{\sigma} \colon \overline{\Gamma\G \rtimes G}^{\{\la 1 \ra\}} \to \Delta X$. Since the map $\Psi \colon  \overline{\Gamma\G\rtimes G}^{\{\la 1 \ra\}} \to \Delta X \times \overline{\Gamma\G\rtimes G}^{\{\la 1 \ra\}}$ defined by $\Psi(x)=(\widetilde{\sigma}(x), x)$ is continuous and $\Gamma\G\rtimes G$-equivariant, the action $\Gamma\G \rtimes G \act \overline{\Gamma\G \rtimes G}^{\{\la 1 \ra\}}$ is topologically amenable by Lemma \ref{lem:G-equivariant map passes topological amenability}. By Proposition \ref{prop:equivalent condition of biexact group}, this implies that $\Gamma\G \rtimes G$ is bi-exact.

    (2) Let $G$ be bi-exact. Define the map $P \colon \ell^\infty(G) \to \ell^\infty(\Gamma\G \rtimes G)$ by $P(f)(gh) = f(h)$, where $f \in \ell^\infty(G)$, $g \in \Gamma\G$, and $h \in G$. We can see $P(A(G;\{\la 1 \ra\})) \subset A(\Gamma\G \rtimes G ; \{\Gamma\G\})$ and $P$ is $G$-equivariant. Hence, $P$ induces a $G$-equivariant continuous map $\widetilde{P} \colon \overline{\Gamma\G \rtimes G}^{\{\Gamma\G\}} \to \overline{G}^{\{\la 1 \ra\}}$. This implies topological amenability of $G \act \overline{\Gamma\G \rtimes G}^{\{\Gamma\G\}}$ since $G$ is bi-exact. Hence, for any $x \in V(X)$, the action $\stab_{\Gamma\G \rtimes G}(x) \act \overline{\Gamma\G \rtimes G}^{\{\Gamma\G\}}$ is topologically amenable since $\stab_{\Gamma\G \rtimes G}(x)$ is conjugate to $G$. By $\sigma(C(\Delta X)) \subset A(\Gamma\G\rtimes G; \{\la 1 \ra\}) \subset A(\Gamma\G \rtimes G ; \{\Gamma\G\})$, $\sigma$ induces a $\Gamma\G \rtimes G$-equivariant continuous map $\widetilde{\sigma} \colon \overline{\Gamma\G \rtimes G}^{\{\Gamma\G\}} \to \Delta X$. Thus, in the same way as Theorem \ref{thm:graph product becomes bi-exact and strongly solid} (1), we can see topological amenability of $\Gamma\G \rtimes G \act \overline{\Gamma\G \rtimes G}^{\{\Gamma\G\}}$ by Theorem \ref{prop:Ozawa Prop 11}. Hence, $\Gamma\G \rtimes G$ is bi-exact relative to $\{\Gamma\G\}$ by Proposition \ref{prop:equivalent condition of biexact group}.
\end{proof}

\begin{rem}\label{rem:girth is essential in bi-exactness}
    In Theorem \ref{thm:graph product becomes bi-exact and strongly solid} (1), the condition $\girth(\Gamma)>4$ is essential for showing bi-exactness. Indeed, if $\Gamma$ is a simplicial graph that contains a subset $C \subset V(\Gamma)$ whose induced subgraph is a circuit of length 4 and a collection $\G = \{G_v\}_{v \in V(\Gamma)}$ of groups satisfies $\min_{v \in C}|G_v| \ge 3$, then the graph product $\Gamma\G$ contains $F_2 \times F_2$, hence $\Gamma\G \rtimes G$ is not bi-exact. Note that if $H$ and $K$ are infinite groups and $H \times K$ is non-amenable, then $H \times K$ is not bi-exact. For the same reason, finiteness of vertex groups is also essential. Indeed, if $\Gamma$ is a simplicial graph in which $a,b,c \in V(\Gamma)$ satisfy $d_\Gamma(a,b)=d_\Gamma(b,c)=1$ and $d_\Gamma(a,c)>1$ and a collection $\G = \{G_v\}_{v \in V(\Gamma)}$ of groups satisfies $|G_b|=\infty$ and $\min\{|G_a|,|G_c| \}>2$, then $\Gamma\G$ contains $G_b \times F_2$, hence $\Gamma\G \rtimes G$ is not bi-exact.
\end{rem}

Next, we prove Theorem \ref{thm:intro graph product becomes bi-exact and strongly solid}. A group $G$ is called \emph{i.c.c. (infinite conjugacy class)} if for any $g \in G\setminus\{1\}$, the set $\{hgh^{-1} \mid h \in G\}$ is infinite.

\begin{cor}\label{cor:graph product strongly solid}
    Suppose that $\Gamma$ is a uniformly locally finite countable graph with $\girth(\Gamma)>4$ and $\diam_\Gamma(\Gamma) > 2$, and that $\G=\{G_v\}_{v \in \Gamma}$ is a collection of finite groups with $\sup_{v \in V(\Gamma)}|G_v| < \infty$, then the group von Neumann algebra $L(\Gamma\G)$ is strongly solid.
\end{cor}

\begin{proof}
    By considering the trivial action of the trivial group $\{1\}$ on $\Gamma$, $\Gamma\G\,(=\Gamma\G\rtimes \{1\})$ is bi-exact by Theorem \ref{thm:graph product becomes bi-exact and strongly solid} (1). It's not difficult to see that $\Gamma\G$ is i.c.c. by $\diam_\Gamma(\Gamma) > 2$. For any finite subset $F$ of $V(\Gamma)$, the subgroup $\la G_v \mid v \in F \ra$ of $\Gamma\G$ is weakly amenable with Cowling-Haagerup constant $1$ by \cite[Theorem 5.5]{Rec17}, since every finite group is weakly amenable with Cowling-Haagerup constant $1$. Hence, $\Gamma\G$ is weakly amenable with Cowling-Haagerup constant $1$ by \cite[Example 12.3.2]{BO08}. Since $\Gamma\G$ is bi-exact and weakly amenable, the group von Neumann algebra $L(\Gamma\G)$ is strongly solid by \cite[Corollary 0.2]{CSU13}.
\end{proof}

Finally, as a byproduct of Lemma \ref{lem:Cayley graph is hyperbolic} and Lemma \ref{lem:Cayley graph becomes fine}, we will improve \cite[Corollary 1.4]{Oya24b}. See Remark \ref{rem:graph wreath product} for the definition of graph-wreath product $\Gamma\G \rtimes G$.

\begin{thm}
    Suppose that $\Gamma$ is a locally finite hyperbolic graph with $\girth(\Gamma)>4$ and that a group $G$ acts on $\Gamma$ properly and cocompactly. Let $H$ be a finite group and define $\G=\{G_v\}_{v \in V(\Gamma)}$ by $G_v=H$ for any $v \in V(\Gamma)$. Then, the graph-wreath product $\Gamma\G \rtimes G$ is hyperbolic.
\end{thm}

\begin{proof}
   We'll show that $\Gamma\G \rtimes G$ is hyperbolic relative to $G$ by checking the conditions in \cite[Theorem 2.33]{Oya24b}, which is \cite[Theorem 7.10]{Bow12} and \cite[Theorem 6.1]{Dah03}. Note that $G$ is hyperbolic, hence finitely generated, since $G$ acts on the proper hyperbolic space $\Gamma$ properly and cocompactly. By this and $|H|<\infty$, $\Gamma\G \rtimes G$ is also finitely generated. Let $X$ be the Cayley graph of $\Gamma\G$ with respect to $\bigcup_{v \in V(\Gamma)}(G_v \setminus \{1\})$. By Lemma \ref{lem:Cayley graph is hyperbolic} and Lemma \ref{lem:Cayley graph becomes fine}, $X$ is hyperbolic and fine. The group $G$ acts on $X$ as graph automorphism by conjugation in $\Gamma\G \rtimes G$ (i.e. $G \times \Gamma\G \ni (g,x) \mapsto gxg^{-1} \in \Gamma\G$), which may not preserve labels of $X$. This and the action $\Gamma\G \act X$ by left multiplication extend to the action $\Gamma\G \rtimes G \act X$.
   
    Since we have $\stab_{\Gamma\G \rtimes G}(1) = G$ and the action $\Gamma\G \rtimes G \act V(X)$ is transitive, every vertex stabilizer of the action $\Gamma\G \rtimes G \act X$ is conjugate to $G$. Also, for any $v\in V(\Gamma)$ and $g \in G_v\setminus\{1\}$, we have $\stab_{\Gamma\G \rtimes G}(1) \cap \stab_{\Gamma\G \rtimes G}(g) = \stab_G(v)$. Note $|\stab_G(v)|<\infty$ since the action $G \act \Gamma$ is proper. Hence, every edge stabilizer of the action $\Gamma\G \rtimes G \act X$ is finite. Finally, since the action $G \act \Gamma$ is cocompact, the vertex orbit set $V(\Gamma)/G$ is finite. By this and $|H|<\infty$, the edge orbit set $E(X)/\Gamma\G \rtimes G$ is finite.
    
    Thus, $\Gamma\G \rtimes G$ is hyperbolic relative to $G$ by \cite[Theorem 2.33]{Oya24b}. Since $G$ is hyperbolic, $\Gamma\G \rtimes G$ is hyperbolic by \cite[Corollary 2.41]{Osi06}.
\end{proof}

Finiteness of vertex groups is essential in Theorem \ref{thm:action on extension graph is convergence action} and Theorem \ref{thm:graph product becomes bi-exact and strongly solid} (1) because of Proposition \ref{prop: infinite case doesn't admit non-elementary convergence action} and Remark \ref{rem:girth is essential in bi-exactness} respectively. However, in Corollary \ref{cor:properly proximal}, Theorem \ref{thm:graph product is relatively bi-exact}, and Theorem \ref{thm:graph product becomes bi-exact and strongly solid} (2), which are about proper proximality and relative bi-exactness, infiniteness of vertex groups by itself doesn't seem to become an obstruction of these properties. This suggests possible room for generalization of these results to broader classes (e.g. amenable or exact vertex groups) by a different approach.



\begin{thebibliography}{CRKdlNG24}

\bibitem[AP]{AP}
Claire Anantharaman and Sorin Popa, \emph{An introduction to {$\rm II_{1}$} factors}, \url{https://www.math.ucla.edu/~popa/Books/IIunV15.pdf}.

\bibitem[BC24]{BC24}
Matthijs Borst and Martijn Caspers, \emph{Classification of right-angled {C}oxeter groups with a strongly solid von {N}eumann algebra}, J. Math. Pures Appl. (9) \textbf{189} (2024), 103591. \MR{4779391}

\bibitem[BCC24]{BCC24}
Matthijs Borst, Martijn Caspers, and Enli Chen, \emph{Rigid graph products}, 2024, Preprint.

\bibitem[BH99]{BH99}
Martin~R. Bridson and Andr\'e Haefliger, \emph{Metric spaces of non-positive curvature}, Grundlehren der mathematischen Wissenschaften [Fundamental Principles of Mathematical Sciences], vol. 319, Springer-Verlag, Berlin, 1999. \MR{1744486}

\bibitem[BHV18]{BCV18}
R\'emi Boutonnet, Cyril Houdayer, and Stefaan Vaes, \emph{Strong solidity of free {A}raki-{W}oods factors}, Amer. J. Math. \textbf{140} (2018), no.~5, 1231--1252. \MR{3862063}

\bibitem[BIP21]{BIP21}
R\'emi Boutonnet, Adrian Ioana, and Jesse Peterson, \emph{Properly proximal groups and their von {N}eumann algebras}, Ann. Sci. \'Ec. Norm. Sup\'er. (4) \textbf{54} (2021), no.~2, 445--482. \MR{4258166}

\bibitem[BO08]{BO08}
Nathanial~P. Brown and Narutaka Ozawa, \emph{{$C^*$}-algebras and finite-dimensional approximations}, Graduate Studies in Mathematics, vol.~88, American Mathematical Society, Providence, RI, 2008. \MR{2391387}

\bibitem[Bor24]{Bor24}
Matthijs Borst, \emph{The {CCAP} for graph products of operator algebras}, J. Funct. Anal. \textbf{286} (2024), no.~8, Paper No. 110350, 41. \MR{4705106}

\bibitem[Bou14]{Bou14}
R{\'e}mi Boutonnet, \emph{{Several rigidity features of von Neumann algebras}}, Theses, {Ecole normale sup{\'e}rieure de lyon - ENS LYON}, June 2014.

\bibitem[Bow98]{Bow98}
Brian~H. Bowditch, \emph{A topological characterisation of hyperbolic groups}, J. Amer. Math. Soc. \textbf{11} (1998), no.~3, 643--667. \MR{1602069}

\bibitem[Bow99a]{Bow99a}
B.~H. Bowditch, \emph{Convergence groups and configuration spaces}, Geometric group theory down under ({C}anberra, 1996), de Gruyter, Berlin, 1999, pp.~23--54. \MR{1714838}

\bibitem[Bow99b]{Bow99b}
\bysame, \emph{Treelike structures arising from continua and convergence groups}, Mem. Amer. Math. Soc. \textbf{139} (1999), no.~662, viii+86. \MR{1483830}

\bibitem[Bow02]{Bow02}
Brian~H. Bowditch, \emph{Groups acting on {C}antor sets and the end structure of graphs}, Pacific J. Math. \textbf{207} (2002), no.~1, 31--60. \MR{1973344}

\bibitem[Bow12]{Bow12}
B.~H. Bowditch, \emph{Relatively hyperbolic groups}, Internat. J. Algebra Comput. \textbf{22} (2012), no.~3, 1250016, 66. \MR{2922380}

\bibitem[Cas20]{Cas20}
Martijn Caspers, \emph{Absence of {C}artan subalgebras for right-angled {H}ecke von {N}eumann algebras}, Anal. PDE \textbf{13} (2020), no.~1, 1--28. \MR{4047640}

\bibitem[CDD23]{CDD24b}
Ionut Chifan, Michael Davis, and Daniel Drimbe, \emph{Rigidity for von neumann algebras of graph product groups ii. superrigidity results}, 2023, Preprint.

\bibitem[CDD24]{CDD24a}
\bysame, \emph{Rigidity for von neumann algebras of graph product groups. i. structure of automorphisms}, 2024, To appear in Analysis and PDE.

\bibitem[CdSH{\etalchar{+}}24a]{IRBDSB24a}
Ian Charlesworth, Rolando de~Santiago, Ben Hayes, David Jekel, Srivatsav~Kunnawalkam Elayavalli, and Brent Nelson, \emph{On the structure of graph product von neumann algebras}, 2024, Preprint.

\bibitem[CdSH{\etalchar{+}}24b]{IRBDSB24b}
\bysame, \emph{Strong 1-boundedness, $l^2$-betti numbers, algebraic soficity, and graph products}, 2024, Preprint.

\bibitem[CdSS18]{CRW18}
Ionut Chifan, Rolando de~Santiago, and Wanchalerm Sucpikarnon, \emph{Tensor product decompositions of {$\rm II_1$} factors arising from extensions of amalgamated free product groups}, Comm. Math. Phys. \textbf{364} (2018), no.~3, 1163--1194. \MR{3875825}

\bibitem[CF17]{CF17}
Martijn Caspers and Pierre Fima, \emph{Graph products of operator algebras}, J. Noncommut. Geom. \textbf{11} (2017), no.~1, 367--411. \MR{3626564}

\bibitem[CI18]{CI18}
Ionut Chifan and Adrian Ioana, \emph{Amalgamated free product rigidity for group von {N}eumann algebras}, Adv. Math. \textbf{329} (2018), 819--850. \MR{3783429}

\bibitem[CJ94]{CJ94}
Andrew Casson and Douglas Jungreis, \emph{Convergence groups and {S}eifert fibered {$3$}-manifolds}, Invent. Math. \textbf{118} (1994), no.~3, 441--456. \MR{1296353}

\bibitem[CKE24]{CK24}
Ionut Chifan and Srivatsav Kunnawalkam~Elayavalli, \emph{Cartan subalgebras in von {N}eumann algebras associated with graph product groups}, Groups Geom. Dyn. \textbf{18} (2024), no.~2, 749--759. \MR{4729825}

\bibitem[CN07]{CN07}
Indira Chatterji and Graham~A. Niblo, \emph{A characterization of hyperbolic spaces}, Groups Geom. Dyn. \textbf{1} (2007), no.~3, 281--299. \MR{2314046}

\bibitem[CRKdlNG24]{MIJ24}
Montserrat Casals-Ruiz, Ilya Kazachkov, and Javier de~la Nuez~Gonz\'alez, \emph{On the elementary theory of graph products of groups}, Dissertationes Math. \textbf{595} (2024), 92. \MR{4792087}

\bibitem[CS13]{CS13}
Ionut Chifan and Thomas Sinclair, \emph{On the structural theory of {${\rm II}_1$} factors of negatively curved groups}, Ann. Sci. \'Ec. Norm. Sup\'er. (4) \textbf{46} (2013), no.~1, 1--33. \MR{3087388}

\bibitem[CSU13]{CSU13}
Ionut Chifan, Thomas Sinclair, and Bogdan Udrea, \emph{On the structural theory of {$\rm{II}_1$} factors of negatively curved groups, {II}: {A}ctions by product groups}, Adv. Math. \textbf{245} (2013), 208--236. \MR{3084428}

\bibitem[Dah03a]{Dah03}
Fran\c~cois Dahmani, \emph{Classifying spaces and boundaries for relatively hyperbolic groups}, Proc. London Math. Soc. (3) \textbf{86} (2003), no.~3, 666--684. \MR{1974394}

\bibitem[Dah03b]{Dah03c}
\bysame, \emph{Combination of convergence groups}, Geom. Topol. \textbf{7} (2003), 933--963. \MR{2026551}

\bibitem[DKE24a]{DK24}
Changying Ding and Srivatsav Kunnawalkam~Elayavalli, \emph{Proper proximality among various families of groups}, Groups Geom. Dyn. \textbf{18} (2024), no.~3, 921--938. \MR{4760266}

\bibitem[DKE24b]{CKE24}
\bysame, \emph{Proper proximality among various families of groups}, Groups Geom. Dyn. \textbf{18} (2024), no.~3, 921--938. \MR{4760266}

\bibitem[DKE24c]{DK24b}
\bysame, \emph{Structure of relatively biexact group von {N}eumann algebras}, Comm. Math. Phys. \textbf{405} (2024), no.~4, Paper No. 104, 19. \MR{4733336}

\bibitem[EH24]{EH24}
Amandine Escalier and Camille Horbez, \emph{Graph products and measure equivalence: classification, rigidity, and quantitative aspects}, 2024.

\bibitem[Gab92]{Gab92}
David Gabai, \emph{Convergence groups are {F}uchsian groups}, Ann. of Math. (2) \textbf{136} (1992), no.~3, 447--510. \MR{1189862}

\bibitem[Gen17]{Gen17}
Anthony Genevois, \emph{Cubical-like geometry of quasi-median graphs and applications to geometric group theory}, 2017, Thèse de doctorat dirigée par Haïssinsky, Peter Mathématiques Aix-Marseille 2017.

\bibitem[Ger09]{Ger09}
Victor Gerasimov, \emph{Expansive convergence groups are relatively hyperbolic}, Geom. Funct. Anal. \textbf{19} (2009), no.~1, 137--169. \MR{2507221}

\bibitem[GM87]{GM87}
F.~W. Gehring and G.~J. Martin, \emph{Discrete quasiconformal groups. {I}}, Proc. London Math. Soc. (3) \textbf{55} (1987), no.~2, 331--358. \MR{896224}

\bibitem[Gre90]{Gre}
Elisabeth~Ruth Green, \emph{Graph products of groups}, 1990, Thesis, The University of Leeds.

\bibitem[Hou10]{Cyr10}
Cyril Houdayer, \emph{Strongly solid group factors which are not interpolated free group factors}, Math. Ann. \textbf{346} (2010), no.~4, 969--989. \MR{2587099}

\bibitem[Iso15]{Iso15}
Yusuke Isono, \emph{Examples of factors which have no {C}artan subalgebras}, Trans. Amer. Math. Soc. \textbf{367} (2015), no.~11, 7917--7937. \MR{3391904}

\bibitem[Kai04]{Kai04}
Vadim~A. Kaimanovich, \emph{Boundary amenability of hyperbolic spaces}, Discrete geometric analysis, Contemp. Math., vol. 347, Amer. Math. Soc., Providence, RI, 2004, pp.~83--111. \MR{2077032}

\bibitem[OP10]{OP10}
Narutaka Ozawa and Sorin Popa, \emph{On a class of {${\rm II}_1$} factors with at most one {C}artan subalgebra}, Ann. of Math. (2) \textbf{172} (2010), no.~1, 713--749. \MR{2680430}

\bibitem[Osi06]{Osi06}
Denis~V. Osin, \emph{Relatively hyperbolic groups: intrinsic geometry, algebraic properties, and algorithmic problems}, Mem. Amer. Math. Soc. \textbf{179} (2006), no.~843, vi+100. \MR{2182268}

\bibitem[Oya23a]{Oya23a}
Koichi Oyakawa, \emph{Bi-exactness of relatively hyperbolic groups}, J. Funct. Anal. \textbf{284} (2023), no.~9, Paper No. 109859, 28. \MR{4545157}

\bibitem[Oya23b]{Oya23b}
Koichi Oyakawa, \emph{Small cancellation groups are bi-exact}, 2023, Journal of Topology and Analysis, to appear.

\bibitem[Oya24a]{Oya24a}
Koichi Oyakawa, \emph{Hyperfiniteness of boundary actions of acylindrically hyperbolic groups}, Forum Math. Sigma \textbf{12} (2024), Paper No. e32, 31. \MR{4715159}

\bibitem[Oya24b]{Oya24b}
Koichi Oyakawa, \emph{Infinite graph product of groups {I}: Geometry of the extension graph}, 2024, Preprint.

\bibitem[Oza06a]{Oza06}
Narutaka Ozawa, \emph{Boundary amenability of relatively hyperbolic groups}, Topology Appl. \textbf{153} (2006), no.~14, 2624--2630. \MR{2243738}

\bibitem[Oza06b]{Oza06b}
\bysame, \emph{A {K}urosh-type theorem for type {$\rm II_1$} factors}, Int. Math. Res. Not. (2006), Art. ID 97560, 21. \MR{2211141}

\bibitem[Pap95]{Pap95}
P.~Papasoglu, \emph{Strongly geodesically automatic groups are hyperbolic}, Invent. Math. \textbf{121} (1995), no.~2, 323--334. \MR{1346209}

\bibitem[PV14]{PV14}
Sorin Popa and Stefaan Vaes, \emph{Unique {C}artan decomposition for {$\rm II_{1}$} factors arising from arbitrary actions of hyperbolic groups}, J. Reine Angew. Math. \textbf{694} (2014), 215--239. \MR{3259044}

\bibitem[Rec17]{Rec17}
Eric Reckwerdt, \emph{Weak amenability is stable under graph products}, J. Lond. Math. Soc. (2) \textbf{96} (2017), no.~1, 133--155. \MR{3687943}

\bibitem[Rud91]{Rud91}
Walter Rudin, \emph{Functional analysis}, second ed., International Series in Pure and Applied Mathematics, McGraw-Hill, Inc., New York, 1991. \MR{1157815}

\bibitem[Sak09]{Sak09}
Hiroki Sako, \emph{The class {$S$} as an {ME} invariant}, Int. Math. Res. Not. IMRN (2009), no.~15, 2749--2759. \MR{2525839}

\bibitem[Sun19]{Sun19}
Bin Sun, \emph{A dynamical characterization of acylindrically hyperbolic groups}, Algebr. Geom. Topol. \textbf{19} (2019), no.~4, 1711--1745. \MR{3995017}

\bibitem[Tuk88]{Tuk88}
Pekka Tukia, \emph{Homeomorphic conjugates of {F}uchsian groups}, J. Reine Angew. Math. \textbf{391} (1988), 1--54. \MR{961162}

\bibitem[Tuk94]{Tuk94}
\bysame, \emph{Convergence groups and {G}romov's metric hyperbolic spaces}, New Zealand J. Math. \textbf{23} (1994), no.~2, 157--187. \MR{1313451}

\bibitem[Tuk98]{Tuk98}
\bysame, \emph{Conical limit points and uniform convergence groups}, J. Reine Angew. Math. \textbf{501} (1998), 71--98. \MR{1637829}

\bibitem[Yam04]{Yam04}
Asli Yaman, \emph{A topological characterisation of relatively hyperbolic groups}, J. Reine Angew. Math. \textbf{566} (2004), 41--89. \MR{2039323}

\end{thebibliography}

\newcommand{\etalchar}[1]{$^{#1}$}
\providecommand{\bysame}{\leavevmode\hbox to3em{\hrulefill}\thinspace}
\providecommand{\MR}{\relax\ifhmode\unskip\space\fi MR }
\providecommand{\MRhref}[2]{%
  \href{http://www.ams.org/mathscinet-getitem?mr=#1}{#2}
}
\providecommand{\href}[2]{#2}

\vspace{5mm}

\noindent  Department of Mathematics and Statistics, McGill University, Burnside Hall,

\noindent 805 Sherbrooke Street West, Montreal, QC, H3A 0B9, Canada.

\noindent E-mail: \emph{koichi.oyakawa@mail.mcgill.ca}

\end{document}